\documentclass[12pt]{article}
\usepackage{graphicx}
\usepackage{amsmath}
\usepackage{amsfonts}
\usepackage{enumerate}
\usepackage{latexsym}
\usepackage{epsfig}
\usepackage{float}
\usepackage{color}
\usepackage{amscd}
\usepackage{amssymb}
\usepackage{epstopdf}

\newtheorem{theorem}{Theorem}[section]
\newtheorem{lemma}[theorem]{Lemma}

\def\proofbox{\begin{picture}(6.5,6.5)
\put(0,0){\framebox(6.5,6.5){}}\end{picture}}
\newenvironment{proof}{\noindent{\it Proof.\quad}}{\hfill\proofbox}

\begin{document}

\title{Exhausting Curve Complexes by Finite Rigid Sets on Nonorientable Surfaces\\}
\author{Elmas Irmak}

\maketitle

\renewcommand{\sectionmark}[1]{\markright{\thesection. #1}}

\thispagestyle{empty}
\maketitle
\begin{abstract} Let $N$ be a compact, connected, nonorientable surface of genus $g$ with $n$ boundary components. Let $\mathcal{C}(N)$ be the curve complex of $N$. 
We prove that if $(g,n) = (3,0)$ or $g + n \geq 5$, then there is an exhaustion of $\mathcal{C}(N)$ by a sequence of finite rigid sets. This improves the author's result on exhaustion of $\mathcal{C}(N)$ by a sequence of finite superrigid sets.\end{abstract}
  
{\small Key words: Curve complexes, nonorientable surfaces, rigidity

}
 
\section{Introduction} Let $N$ be a compact, connected, nonorientable surface of genus $g$ with $n$ boundary components. Let
$Mod_N$ denote the mapping class group, the group of isotopy classes of all self-homeomorphisms, of $N$. Let $\mathcal{C}(N)$ denote the complex of curves of $N$. This is an abstract simplicial complex. Its vertex set 
is the set of isotopy classes of nontrivial simple closed curves on $N$ where nontivial means it does not bound a disk, it does not bound a  M\"{o}bius band and it is not isotopic to a boundary component of $N$. A set of $n$ vertices forms an $n-1$ dimensional simplex if
its elements can be represented by pairwise disjoint simple closed curves on the surface. There is a natural action of $Mod_N$ on $\mathcal{C}(N)$ by automorphisms. If $[h] \in Mod_N$, then $[h]$ induces an automorphism 
$[h]_*: \mathcal{C}(N) \rightarrow \mathcal{C}(N)$ where 
$[h]_*([a]) = [h(a)]$ for every vertex $[a]$ of $\mathcal{C}(N)$.
Let $\mathcal{B}$ be a subcomplex of $\mathcal{C}(N)$.
A simplicial map $\lambda :\mathcal{B} \rightarrow \mathcal{C}(N)$ is called locally injective if it is injective on the star of every vertex in $\mathcal{B}$. We will say that 
$\mathcal{B}$ is rigid if every locally injective simplicial map $\lambda : \mathcal{B} \rightarrow \mathcal{C}(N)$ is induced by a homeomorphism, i.e. if $\lambda :\mathcal{B} \rightarrow \mathcal{C}(N)$ is a locally injective map then there exists a homeomorphism $h : N \rightarrow N$ such that $[h]_*(\alpha) = \lambda(\alpha)$ for every vertex $\alpha$ in $\mathcal{B}$. 

Aramayona-Leininger proved that there is an exhaustion of the complex of curves by a sequence of finite rigid sets on compact, connected, orientable surfaces in \cite{AL2}. We will prove that result for nonorientable surfaces in this paper. We will use the same notation for any set and the subcomplex of $\mathcal{C}(N)$ that is spanned by that set. The main result is the following:
 
\begin{theorem} \label{B} If $(g,n) = (3,0)$ or
$g + n \geq 5$, then there exists a sequence $\mathcal{E}_1 \subset \mathcal{E}_2 \subset \dots \subset \mathcal{E}_n \subset \dots$  such that 

(i) $\mathcal{E}_i$ is a finite rigid set in $\mathcal{C}(N)$ for all $i \in \mathbb{N}$,

(ii) when $g + n \geq 5$, $\mathcal{E}_i$ has trivial pointwise stabilizer in $Mod_N$ for each $i \in \mathbb{N}$, and 

(iii) $\bigcup_{i \in \mathbb{N}} \mathcal{E}_i = \mathcal{C}(N)$. \end{theorem}
  
Ivanov proved that automorphisms of complex of curves are induced by homeomorphisms on compact, connected, orientable surfaces when genus is at least two, and using this result he classified injective homomorphisms between finite index subgroups of mapping class groups in \cite{Iv1}. After that many results were proven about simplicial maps of complex of curves on orientable and nonorientable surfaces. Superinjective simplicial maps of complex of curves were first studied by the author in \cite{Ir1} to classify injective homomorphisms from finite index subgroups of mapping class groups to the whole group for orientable surfaces. We remind that a simplicial map is called superinjective if it preserves geometric intersection zero and nonzero properties of vertices. The author also proved that superinjective simplicial maps of the complexes of curves on nonorientable surfaces are induced by homeomorphisms in \cite{Ir4}. Superinjective simplicial maps of two-sided curve complexes were classfied by Irmak-Paris on nonorientable surfaces in \cite{IrP1}. Using this result Irmak-Paris gave a classification of injective homomorphisms from finite index subgroups to the whole mapping class group on these surfaces in \cite{IrP2}. 
Shackleton proved that locally injective simplicial maps of the curve complex are induced by homeomorphisms in \cite{Sh} for orientable surfaces. 
Aramayona-Leininger proved that there exists a finite rigid subcomplex in the curve complex of a compact, connected, orientable surface $R$ in \cite{AL1}, and using this result they proved that  there is an exhaustion of the curve complex by a sequence of finite rigid sets on such orientable surfaces in \cite{AL2}, i.e. there exists a sequence $\mathcal{X}_1 \subset \mathcal{X}_2 \subset \dots \subset \mathcal{X}_n \subset \dots$  such that $\mathcal{X}_i$ is a finite rigid set in $\mathcal{C}(R)$ for each $i$, and $\bigcup_{i \in \mathbb{N}} \mathcal{X}_i = \mathcal{C}(R)$. For nonorientable surfaces, Ilbira-Korkmaz proved the existence of finite rigid subcomplexes in  $\mathcal{C}(N)$ when $g+n \neq 4$ in \cite{IlK}. The author proved that if $(g, n) \neq (1,2)$ and $g + n \neq 4$, then there is an exhaustion of $\mathcal{C}(N)$ by a sequence of finite superrigid sets in 
\cite{Ir10} (a subcomplex $X$ of $\mathcal{C}(N)$ is called superrigid if every superinjective simplicial map $\lambda : X \rightarrow \mathcal{C}(N)$ is induced by a homeomorphism.) 

In this paper 
the author improves her result given in 
\cite{Ir10} and proves that there is an exhaustion of $\mathcal{C}(N)$ by a sequence of finite rigid sets except for a few cases. 

{\bf Remark:} Proving the exhaustion of $\mathcal{C}(N)$ by finite rigid sets is a harder problem than proving the exhaustion of $\mathcal{C}(N)$ by supperrigid sets since controlling the images of finite sets under locally injective simplicial maps is harder than controlling the images of finite sets under superinjective simplicial maps. In this paper we had to come up with different curve configurations 
and hence a different sequence compared to the work given by the author in \cite{Ir10}. Differences also show up for small genus cases (which effects the construction of the sequence for higher genus cases) that we discuss in section one. For example, we show that for $(g,n)=(2,1)$ there is no exhaustion of $\mathcal{C}(N)$ by finite rigid sets, even though there is an exhaustion of $\mathcal{C}(N)$ by finite superrigid sets which was proved in \cite{Ir10}.
 
\section{Exhaustion of $\mathcal{C}(N)$ by finite rigid sets for some small genus cases}
\begin{figure}
	\begin{center}
		\hspace{0.005cm} \epsfxsize=1.001in \epsfbox{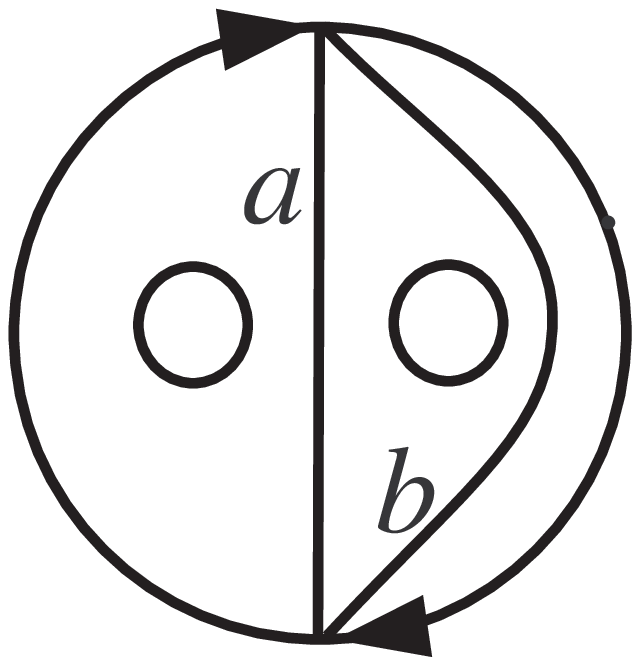} \hspace{1.1cm} \epsfxsize=1.93in \epsfbox{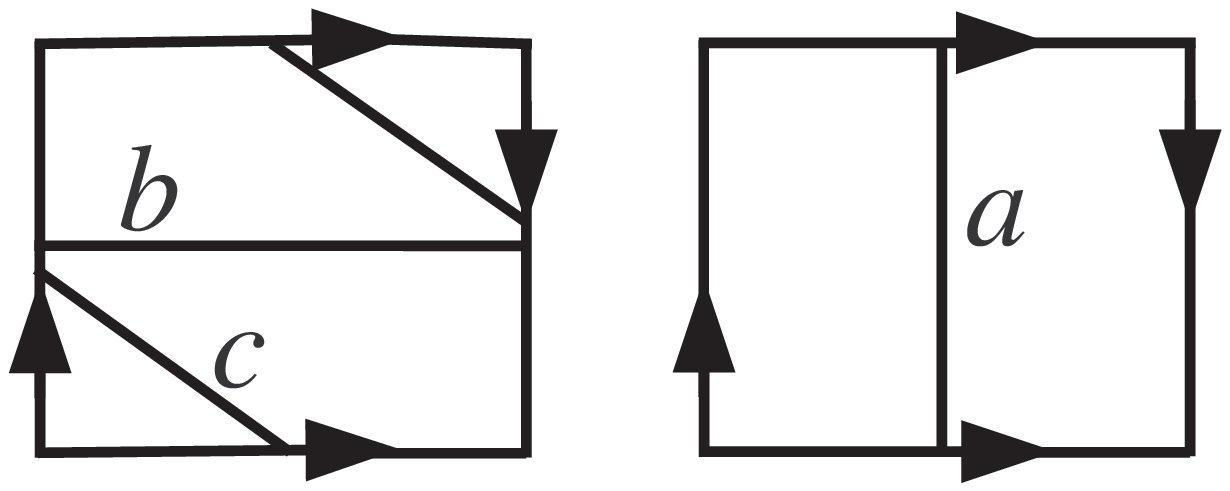}
		
		\hspace{0.1cm} (i) \hspace{4.3cm} (ii) \hspace{1.11cm}
		
		\caption{(i) $(g, n)=(1,2)$;\ \ \ (ii) $(g, n)=(2,0)$}
		\label{fig-zeros-1}
	\end{center}
\end{figure}   

We will show that the statement that there exists an exhaustion of the complex of curves by a sequence of finite rigid sets is true for 
$(g,n)=(1,0), (g,n)=(1,1)$ and $(g,n)=(3,0)$, and not true for $(g,n)=(1,2), (g,n)=(2,0)$ and $(g,n)=(2,1)$. 

Let $a$ be a simple closed curve on $N$. The curve $a$ is called  1-sided if a regular neighborhood of $a$ is homeomorphic to a M\"obius band, and it is called 2-sided if a regular neighborhood of $a$ is homeomorphic to an annulus. We recall that if $v$ is a vertex, then its star, $St(v)$, is the subcomplex of $\mathcal{C}(N)$ whose simplices are the simplices of $\mathcal{C}(N)$ that contain $v$ and the faces of such simplices. 

If $(g,n) = (1,0)$, then $N$ is the projective plane. If $(g,n) = (1,1)$, then 
$N$ is Mobius band. If $(g,n) = (1,0)$ or $(g,n) = (1,1)$, then there is only one vertex in the curve complex. It is represented by a 1-sided curve and we see that $\mathcal{C}(N)$ is a rigid set. If $(g,n) = (1,2)$, then the vertex set is $\{[a], [b]\}$ (see \cite{Sch}) where 
the 1-sided curves $a$ and $b$ are as shown in Figure \ref{fig-zeros-1} (i). A simplicial map sending both $[a]$ and $[b]$ to $[a]$ is locally injective but not induced by a homeomorphism. 

If $(g,n) = (2,0)$, then the vertex set for the complex of curves is $\{[a], [b], [c]\}$ where the curves $a$, $b$ and $c$ are as shown in Figure \ref{fig-zeros-1} (ii) (see \cite{Sch}). A simplicial map sending both $[a]$ and $[b]$ to $[b]$, and $[c]$ to $[c]$ is locally injective but not induced by a homeomorphism.  

\begin{figure}
	\begin{center}
		\hspace{0.005cm}	  \epsfxsize=1.18in \epsfbox{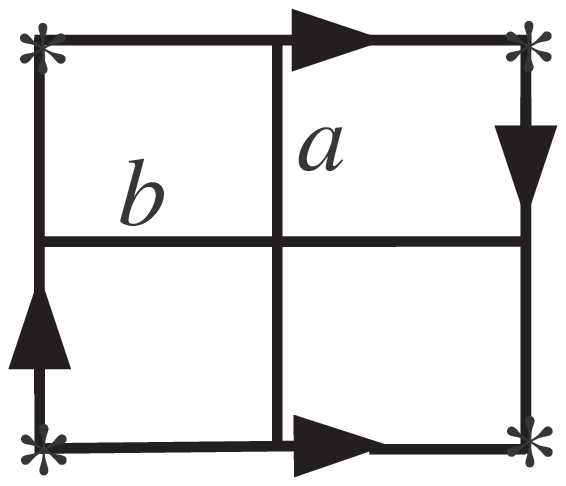} \hspace{0.2cm}
		\epsfxsize=3.72in \epsfbox{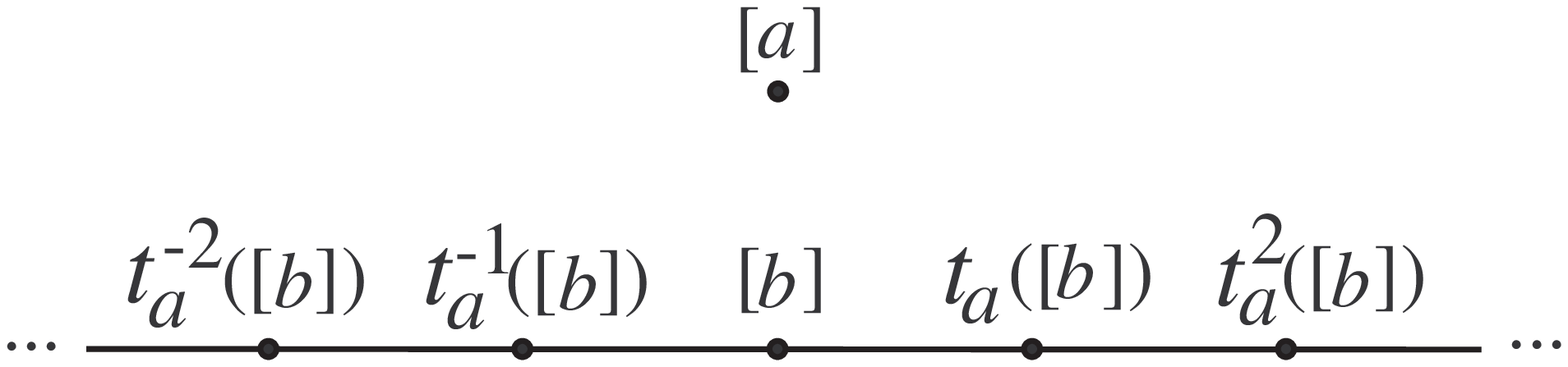}  
		
		\hspace{-3.1cm}	 (i) \hspace{6cm} (ii) 
		\caption{ $(g, n)=(2,1)$}
		\label{fig-zeros-2}
	\end{center}
\end{figure}

When $(g, n) = (2, 1)$ the complex of curves is given by Scharlemann (see \cite{Sch}). The vertex set is $\{[a], [b], t_a^m ([b]) : m \in \mathbb{Z} \}$ where the 2-sided curve $a$ and 1-sided curve $b$ are as shown in
Figure \ref{fig-zeros-2} (i) (the figure is drawn with a puncture instead of a boundary component) and $t_x$ is the dehn twist about $x$. The complex is given in Figure \ref{fig-zeros-2} (ii). If $F$ is a finite set of vertices containing $[a]$, we can see that the simplicial map fixing $[x]$ for every $[x] \in F \setminus \{[a]\}$ and sending $[a]$ to $t_a^m ([b])$ 
for some $m \in \mathbb{Z}$ such that $t_a^m ([b])$ is in the complement of $F$ is locally injective and injective but it is not induced by a homeomorphism. So, there exists no finite rigid set containing $[a]$.\medskip  

If $f: N \rightarrow N$ is a homeomorphism, then we will use the same notation for $f$ and $[f]_*$. Now we will prove Theorem 1.1 for $(g,n)=(3,0)$. In this section from now on we assume that $(g,n)=(3,0)$. Let $\mathcal{B} = \{a_1, a_2, a_3, c_1, c_2, d, e, f, j, l,$ $ u, v, w\}$ where the curves are as shown in Figure \ref{fig-new}. We use cross signs in the figures to mean that we remove the interiors of the disks which have cross signs inside
and identify the antipodal points on the resulting boundary components. Note that the curve $l$ is 1-sided and $l$ has orientable complement on $N$. The set $\mathcal{B}$ has nontrivial curves of every topological type on $N$. 

\begin{figure}
\begin{center}
\epsfxsize=1.71in \epsfbox{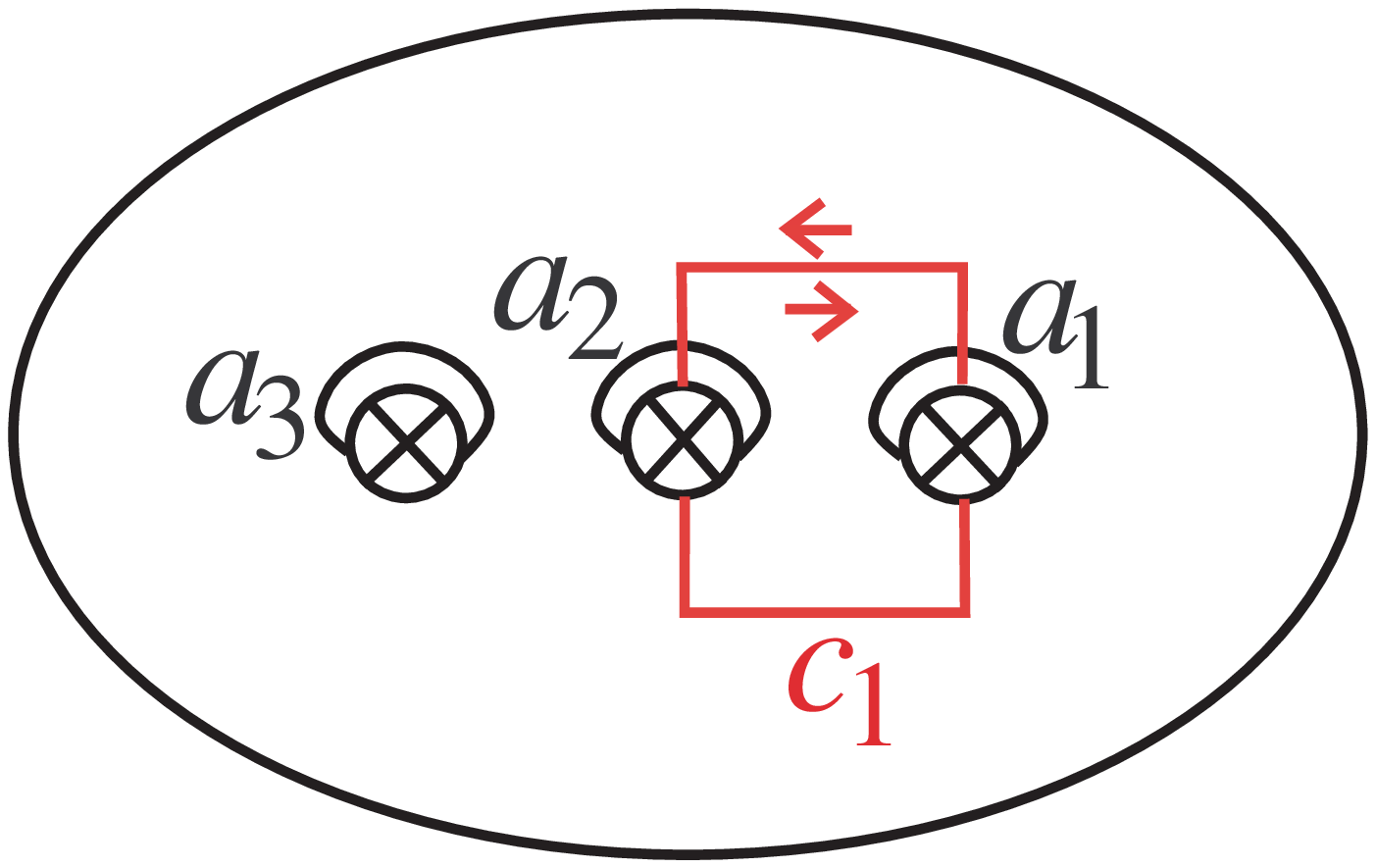}   \hspace{0.19cm}   	\epsfxsize=1.71in \epsfbox{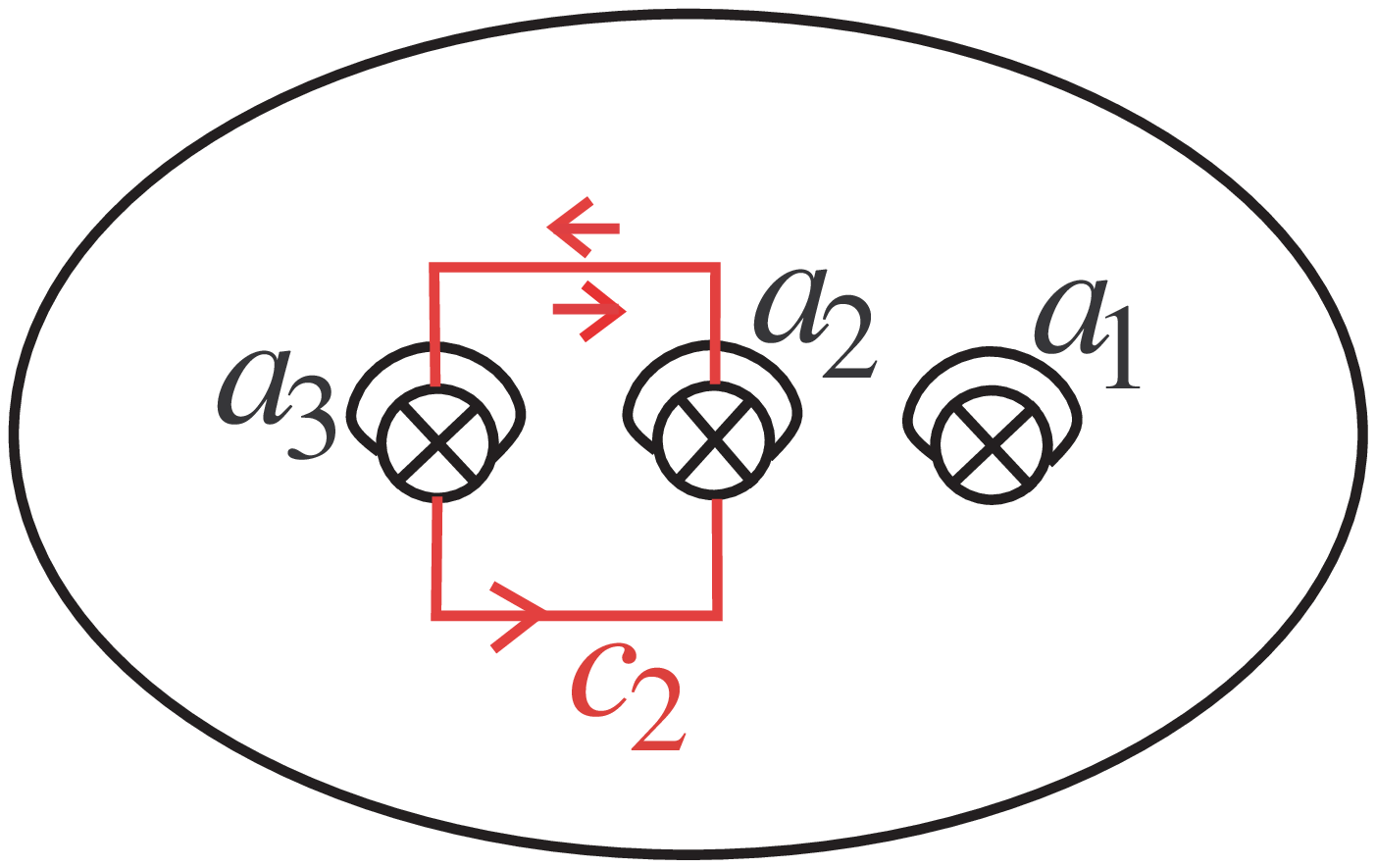}   \hspace{0.19cm}  	\epsfxsize=1.71in \epsfbox{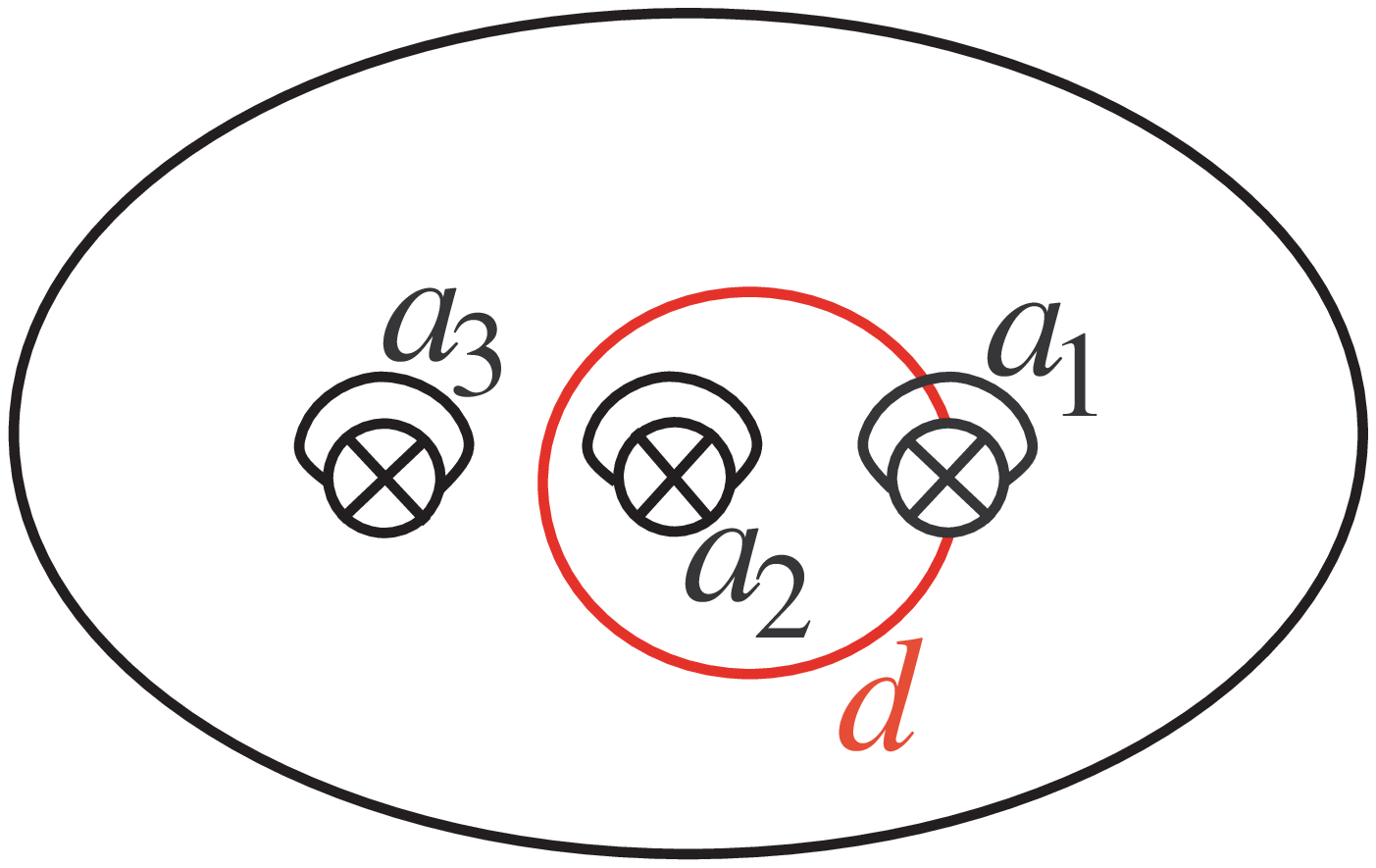} \vspace{0.2cm}
		 
\epsfxsize=1.71in \epsfbox{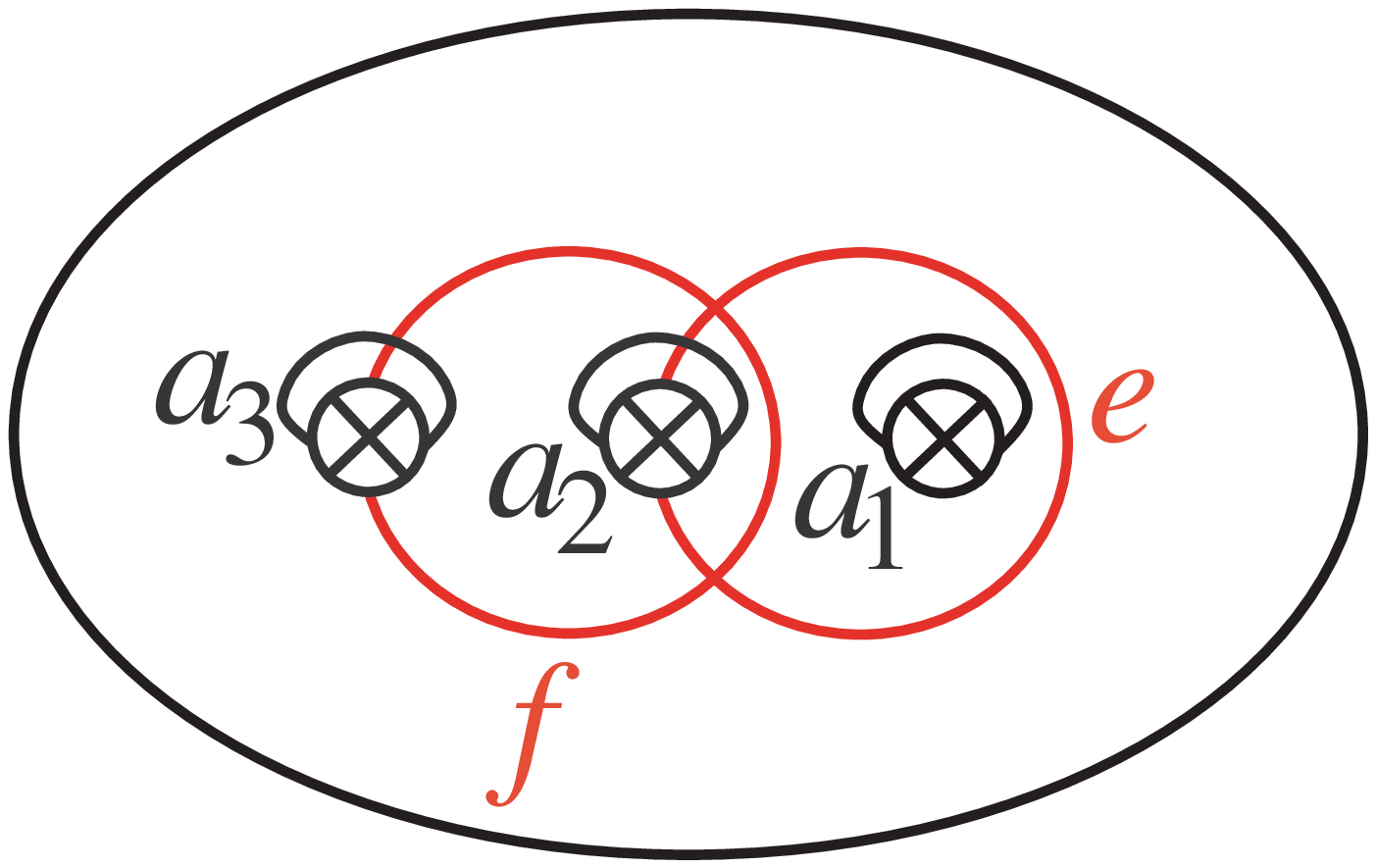} \hspace{0.19cm}   			\epsfxsize=1.71in \epsfbox{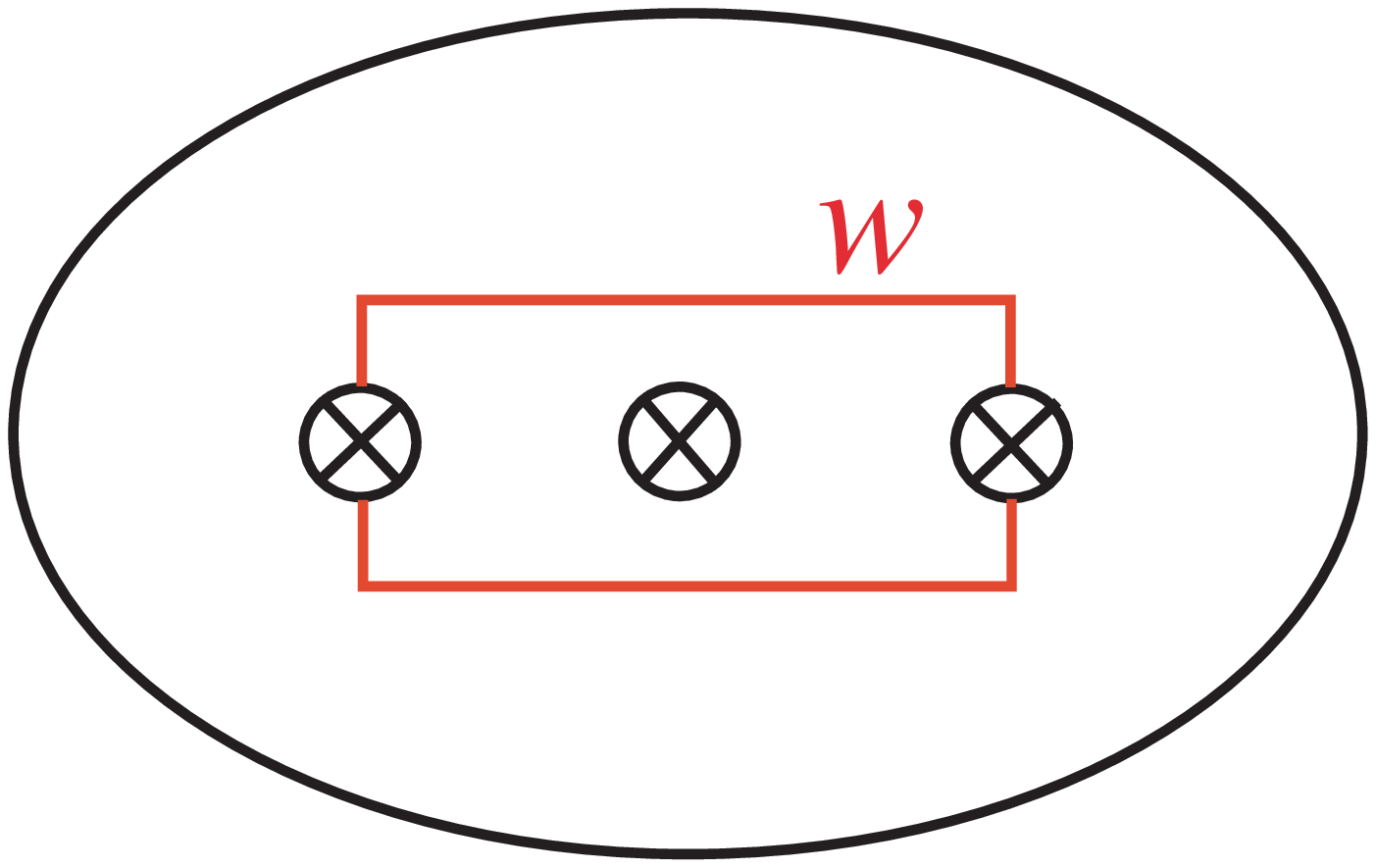}   \hspace{0.19cm}   	\epsfxsize=1.71in \epsfbox{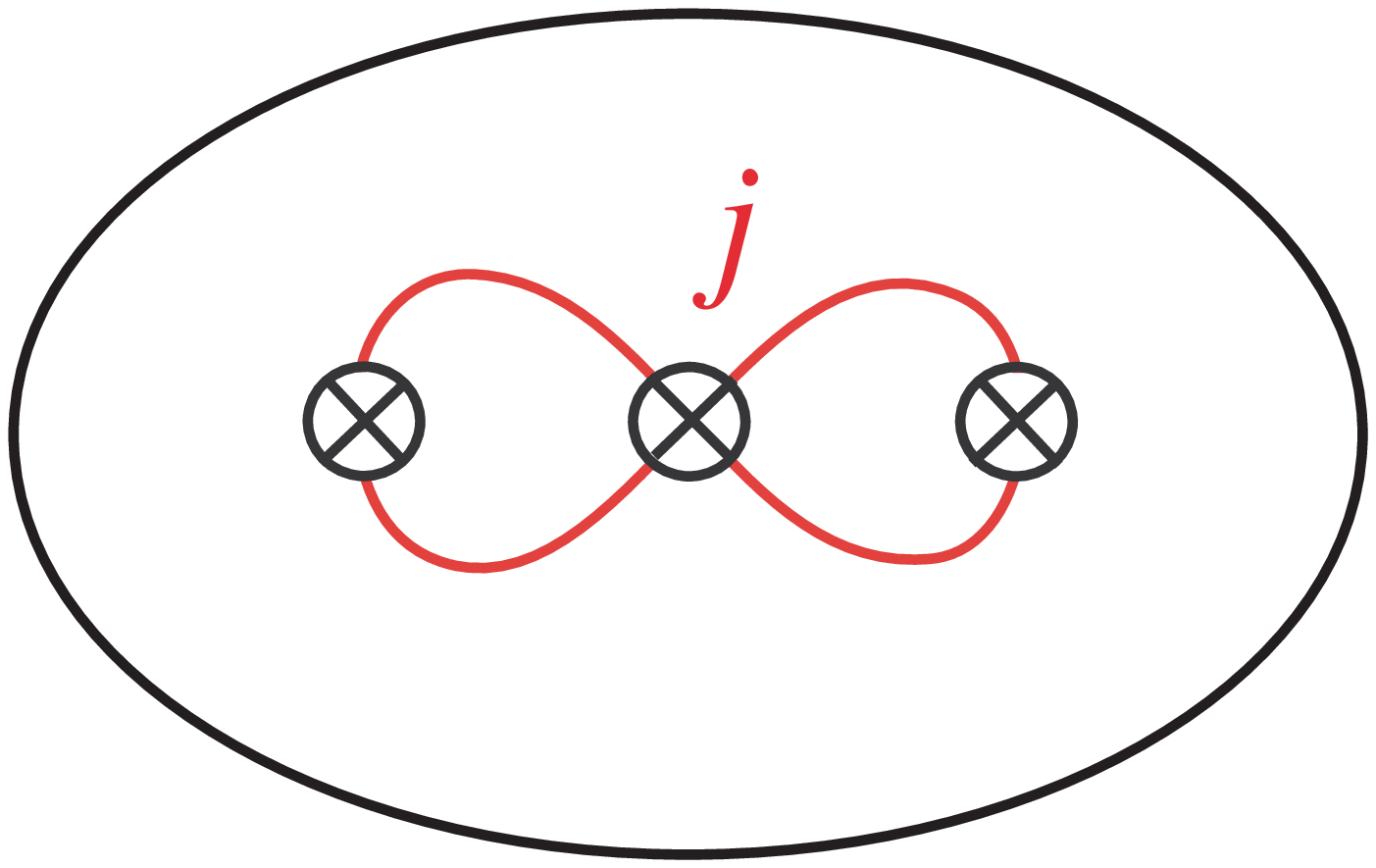} \vspace{0.19cm}
			
\epsfxsize=1.71in \epsfbox{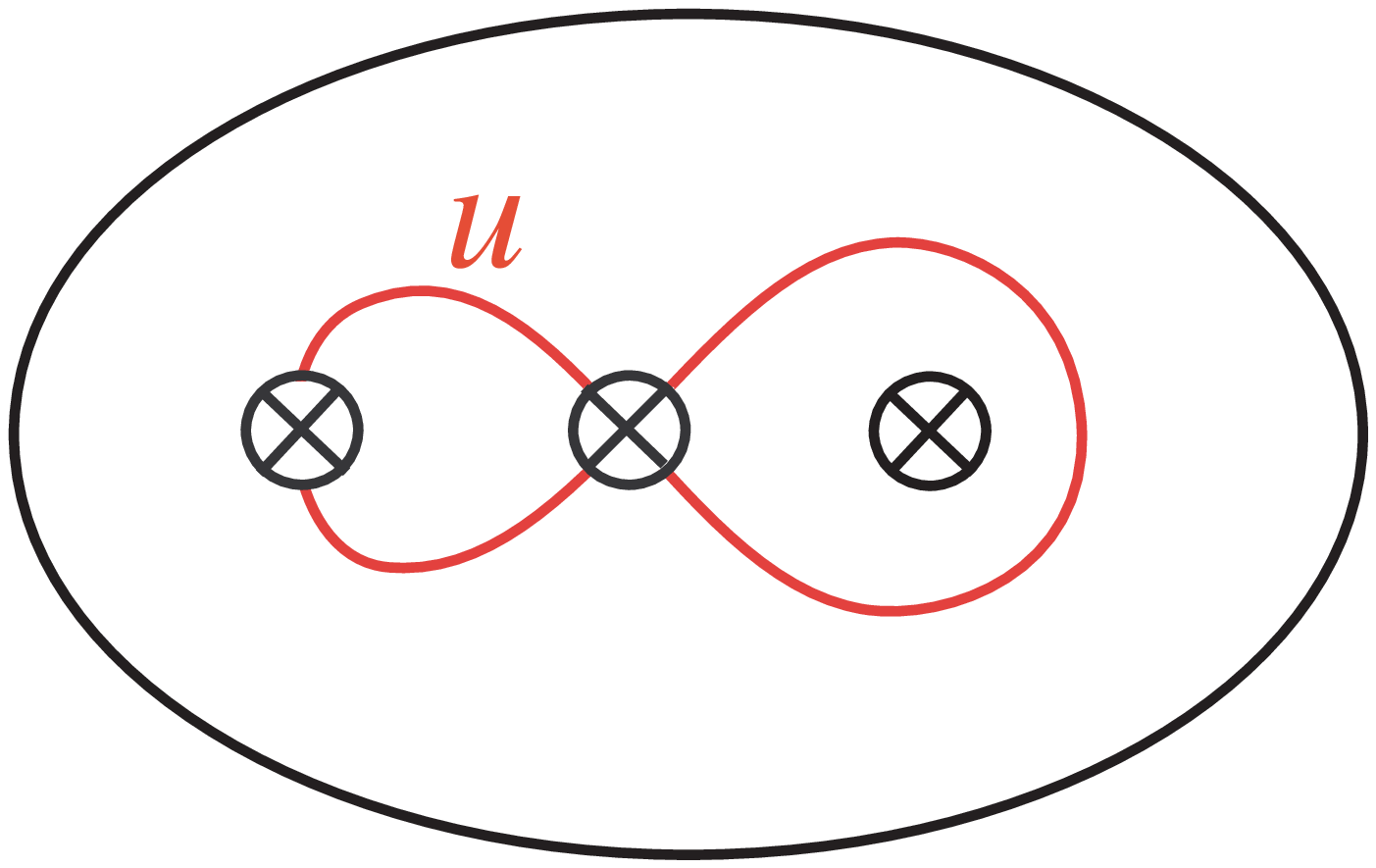}   \hspace{0.19cm}   	\epsfxsize=1.71in  \epsfbox{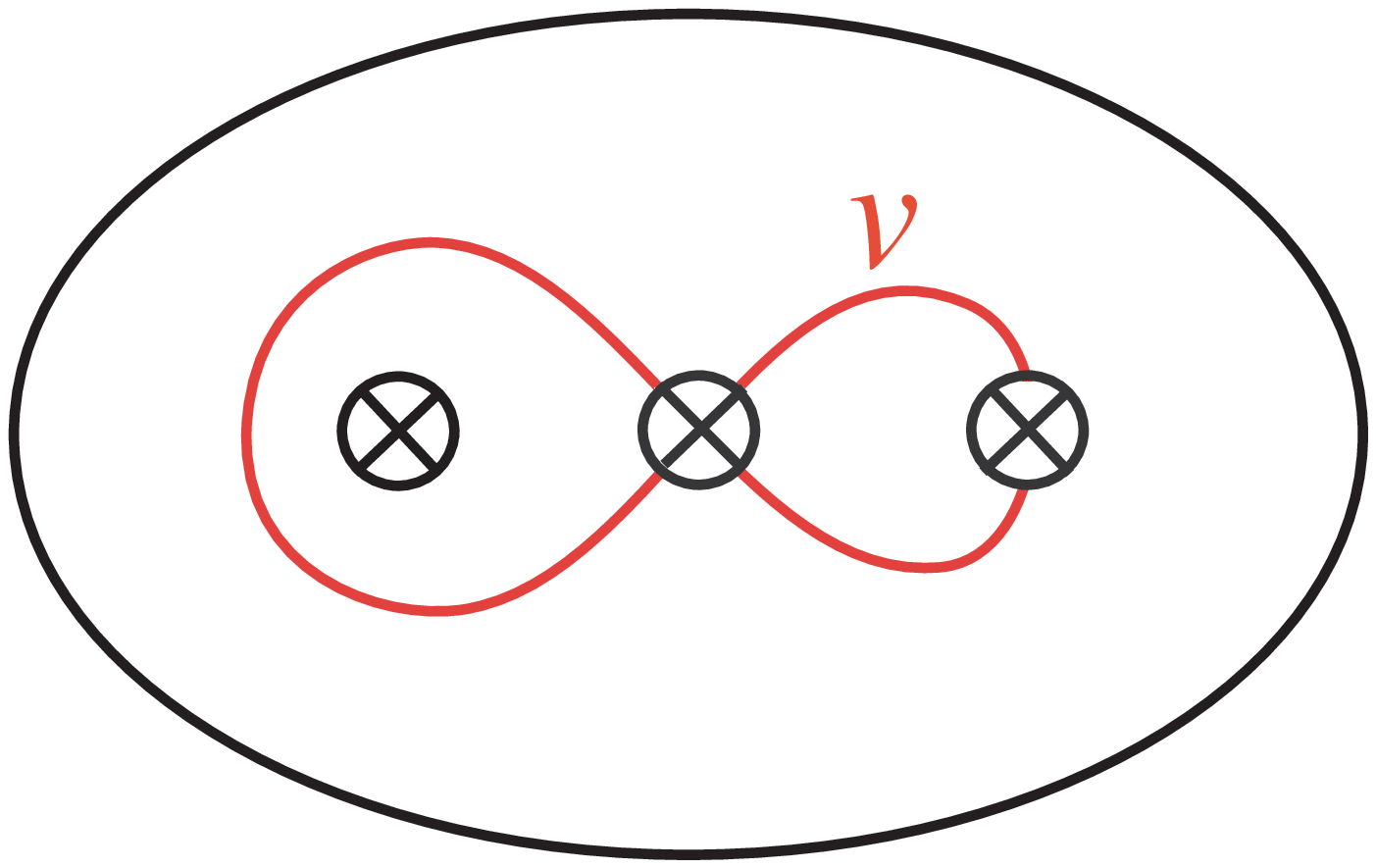}  \hspace{0.19cm}   	\epsfxsize=1.71in \epsfbox{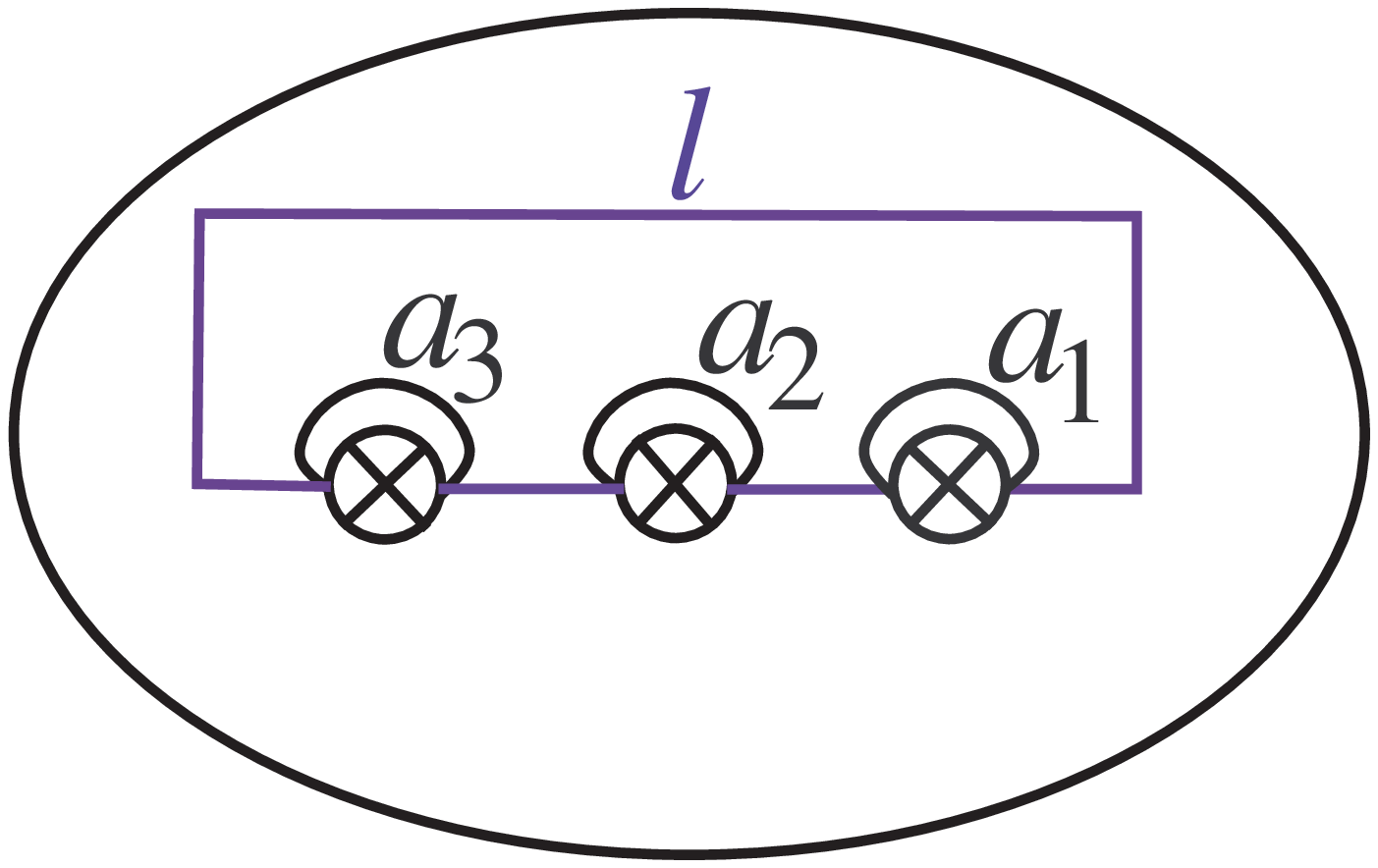} 
		
\caption{Curves in $\mathcal{B}$ }
	
\label{fig-new}
\end{center}
\end{figure} 

\begin{figure}[t]
	\begin{center}
\hspace{-0.8cm} \epsfxsize=2.3in \epsfbox{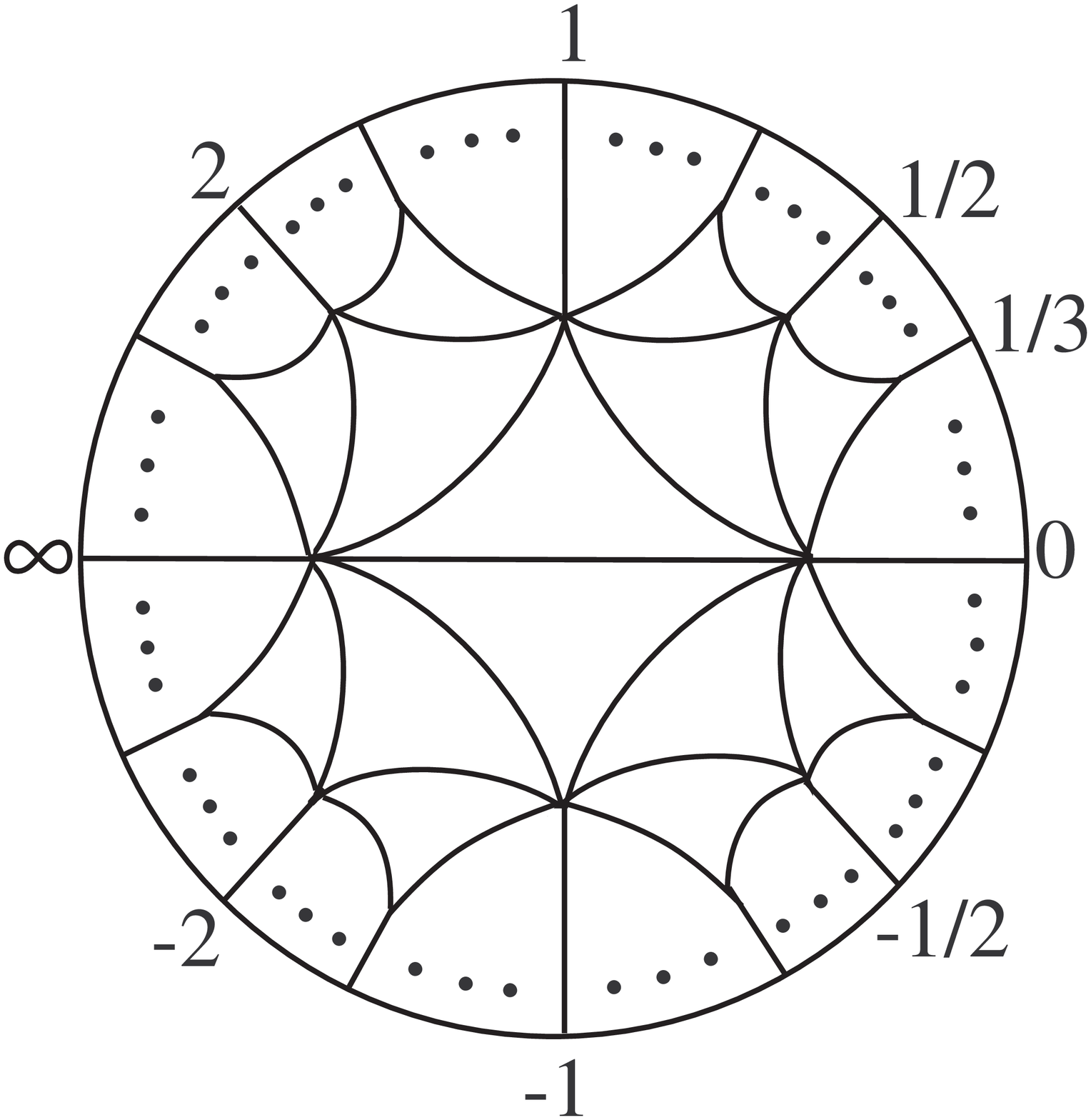} 
\hspace{-0.1cm}	\epsfxsize=2.3in \epsfbox{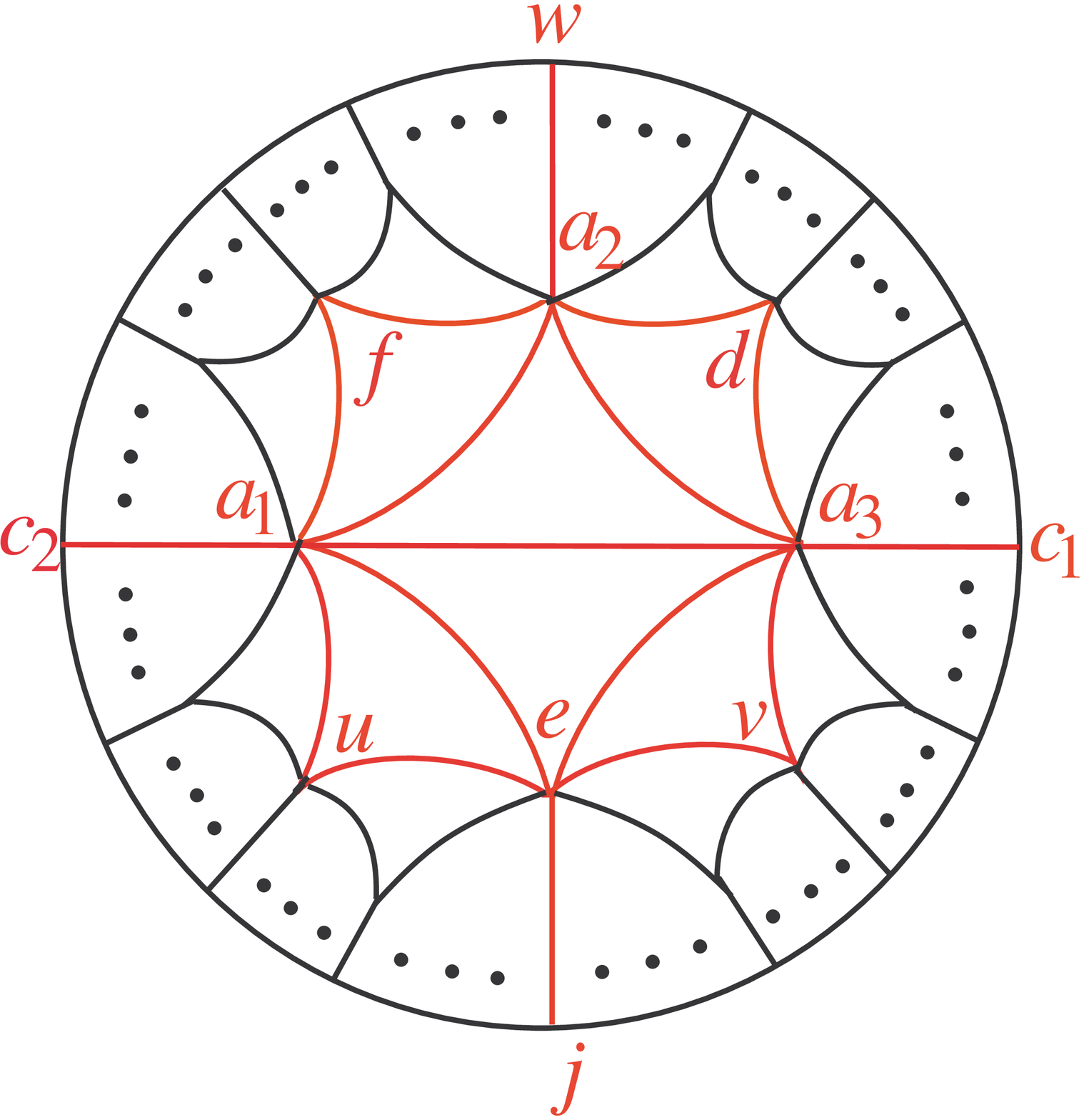} 
\hspace{-0.9cm}	\epsfxsize=1.9in \epsfbox{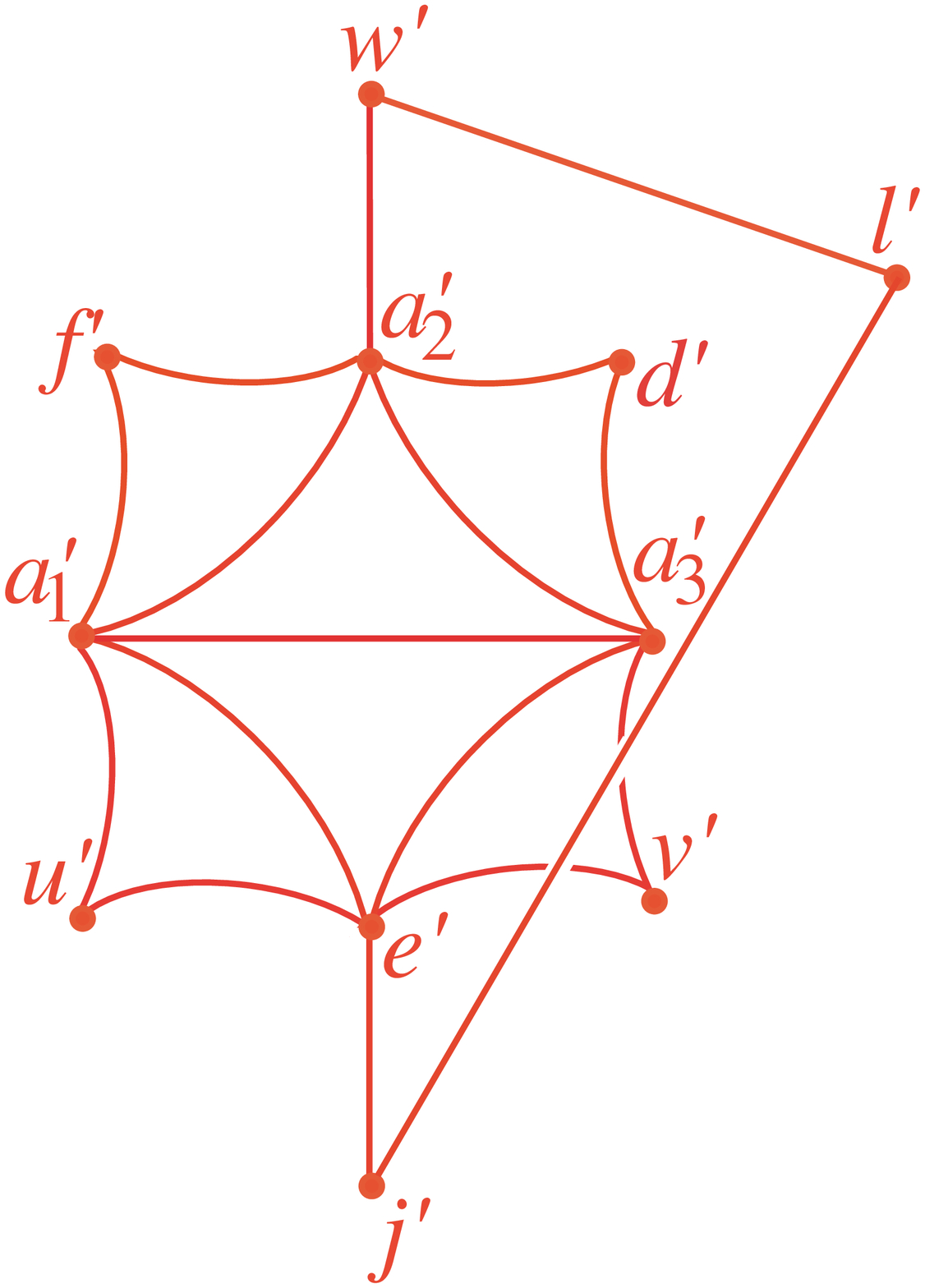} 

\hspace{0.6cm}(i) \hspace{4.9cm} (ii) \hspace{4.7cm} (iii)
		\caption{Curve complex when $(g,n)=(3,0)$ }
		\label{fig-new-2}
	\end{center}
\end{figure} 

\begin{lemma} \label{rigid} $\mathcal{B} = \{a_1, a_2, a_3, c_1, c_2, d, e, f, j, l,$ $ u, v, w\}$ is a rigid set.\end{lemma}

\begin{proof} Let $\lambda: \mathcal{B} \rightarrow \mathcal{C}(N)$ be a  locally injective simplicial map. Since $a_1, a_2, a_3$ are pairwise disjoint, $\lambda$ is locally injective and $(g,n)= (3,0)$, we can see easily that $\lambda([a_1])$, $\lambda([a_2])$ and $\lambda([a_3])$ have pairwise disjoint representatives which are 1-sided curves on $N$. So, there exists a homeomorphism $h$ of $N$ such that $h([x])=\lambda([x])$ for each $x \in \{a_1, a_2, a_3\}$. Since $d$ is the unique nontrivial simple closed curve up to isotopy disjoint from and nonisotopic to each of $a_2, a_3$, and also nonisotopic to $a_1$ we see that $h([d])=\lambda([d])$. Similarly, we can see that $h([e])=\lambda([e])$ and $h([f])=\lambda([f])$. Since $v$ is the unique nontrivial simple closed curve up to isotopy disjoint from and nonisotopic to each of $e, a_3$, and also nonisotopic to $a_1$ we have $h([v])=\lambda([v])$. Similarly, $h([u])=\lambda([u])$.
	
To see that $h([l])=\lambda([l])$ we will consider the description of this curve complex given by Scharlemann in \cite{Sch}: the complex is as given in the first part of Figure \ref{fig-new-2} except that there is also a cone point, the vertex $[l]$, which is connected only to every vertex on the boundary of the disk. The vertices in the interior of the disk correspond to 1-sided curves which have nonorientable complements on $N$. The vertices on the boundary of the disk correspond to 2-sided curves. The cone point is $[l]$ and $l$ is the unique 1-sided curve up to isotopy which has orientable complement on $N$ (see \cite{Sch}). We will consider the path of lenght four, $[e] \rightarrow [j] \rightarrow [l] \rightarrow [w] \rightarrow [a_2]$, between the vertices $[e]$ and $[a_2]$. 

Let $a'_1, a'_2, a'_3, d', e', f', u', v', w', j', l'$ be minimally intersecting representatives of $\lambda([a_1]),$ $\lambda([a_2]), \dots, \lambda([l])$ respectively. By the above arguments we see that $a'_1, a'_2, a'_3, d', $ $e', f', u',$ $ v'$ are as shown in Figure \ref{fig-new-2} (iii). Since $\lambda$ is locally injective, $\lambda([w])$ is not equal to any of $\lambda([f])$, $\lambda([d])$, $\lambda([a_1])$, $\lambda([a_3])$. Since there is an edge between $w$ and $a_2$, there will be an edge between $w'$ and $a'_2$ as shown in Figure \ref{fig-new-2} (iii). Similarly, $\lambda([j])$ is not equal to any of $\lambda([u])$, $\lambda([v])$, $\lambda([a_1])$, $\lambda([a_3])$. Since there is an edge between $e$ and $j$, there will be an edge between $e'$ and $j'$ as shown in Figure \ref{fig-new-2} (iii).

Since $\lambda$ is locally injective it sends edges to edges. Since there is an edge between $w$ and $l$, there will be an edge between $w'$ and $l'$. Since there is an edge between $j$ and $l$, there will be an edge between $j'$ and $l'$. Hence, 
we can see that $\lambda([e]) \rightarrow \lambda([j]) \rightarrow \lambda([l]) \rightarrow \lambda([w]) \rightarrow \lambda([a_2])$ is a path of lenght four and 
$j', l', w'$ are as shown in Figure \ref{fig-new-2} (iii). This is possible only 
when $l'$ is isotopic to $l$. Hence, $\lambda([l])=[l]$. Since $l$ is the unique 1-sided curve up to isotopy which has orientable complement on $N$, we also have
$h([l])=[l]$. So, $h([l])=\lambda([l])$. 

Since $c_1$ is the unique nontrivial simple closed curve up to isotopy disjoint from and nonisotopic to each of $a_3$ and $l$ we have $h([c_1])=\lambda([c_1])$. Since $c_2$ is the unique nontrivial simple closed curve up to isotopy disjoint from and nonisotopic to each of $a_1$ and $l$ we have $h([c_2])=\lambda([c_2])$. Since $w$ is the unique nontrivial simple closed curve up to isotopy disjoint from and nonisotopic to each of $a_2$ and $l$ we have $h([w])=\lambda([w])$. Since $j$ is the unique nontrivial simple closed curve up to isotopy disjoint from and nonisotopic to each of $e$ and $l$ we have $h([j])=\lambda([j])$. We proved that $h([x])=\lambda([x])$ for every $x \in \mathcal{B}$. Hence, $\mathcal{B}$ is a rigid set.\end{proof}\medskip   

  
\begin{figure}
	\begin{center}
		\epsfxsize=1.71in \epsfbox{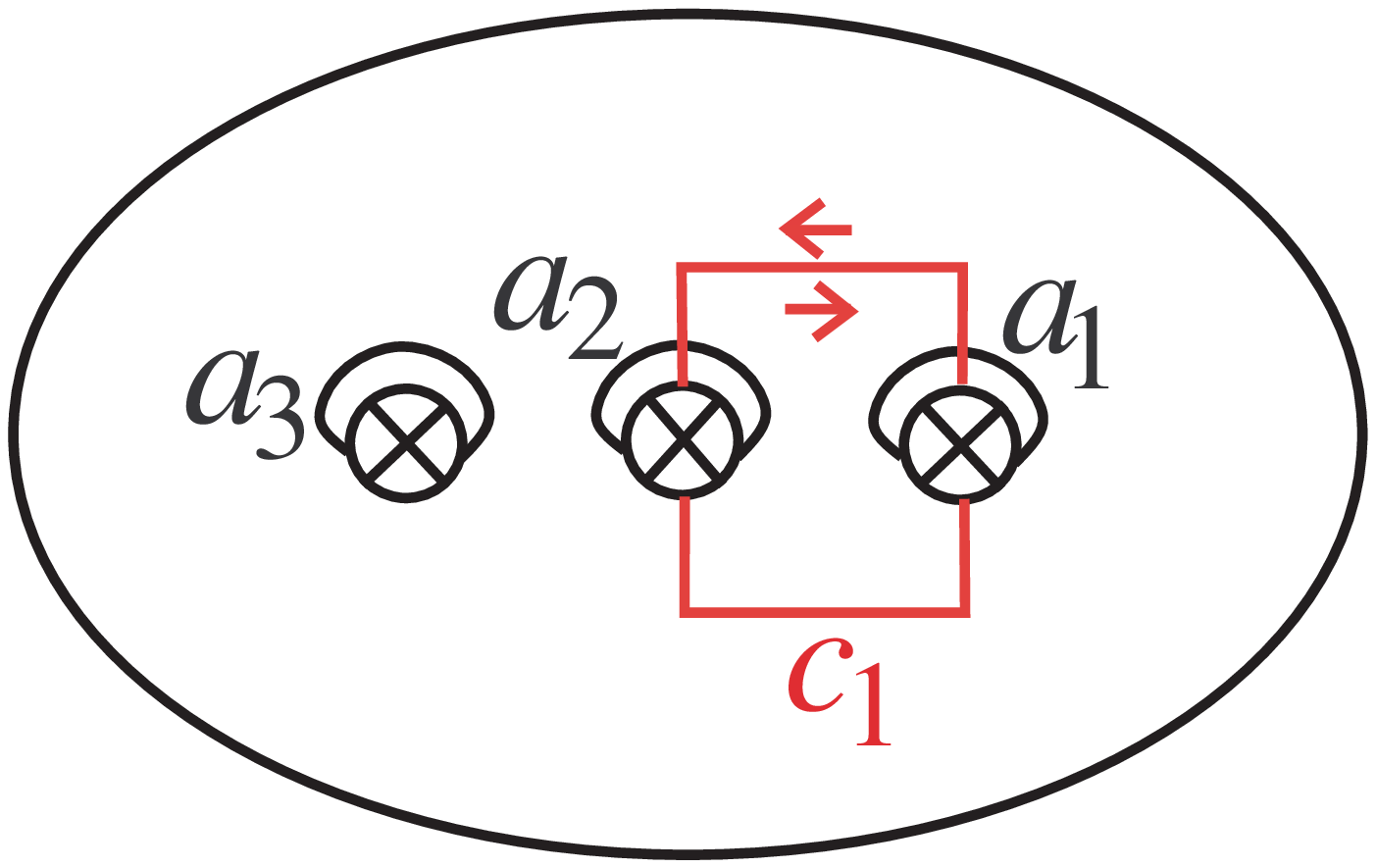} \hspace{0.19cm}    
		\epsfxsize=1.71in \epsfbox{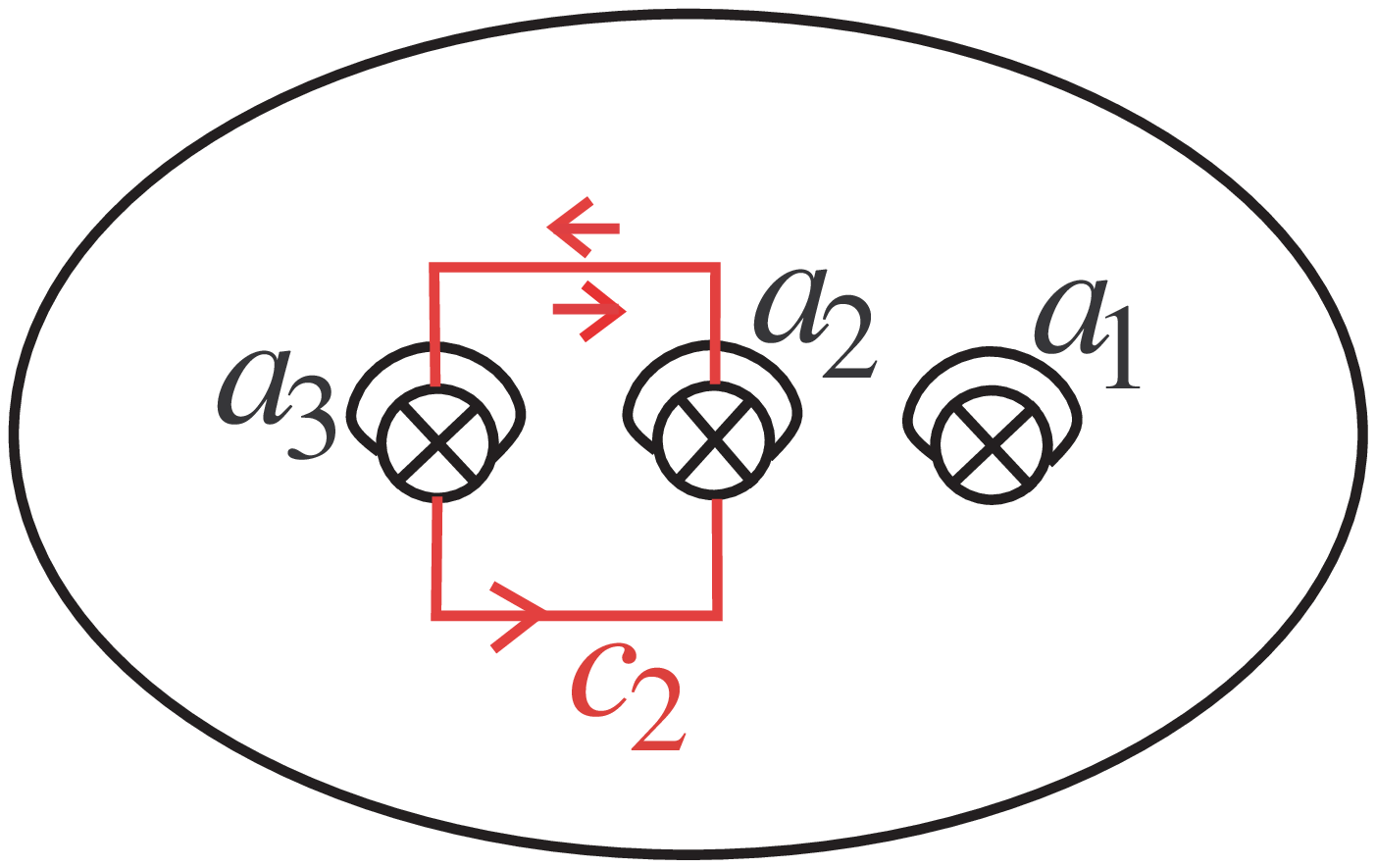}    	
	
		\caption{ $(g,n)=(3,0)$}
				\label{fig-gen}
	\end{center}
\end{figure} 
  
$G= \{t_{c_1}, t_{c_2}, y\}$ is a generating set for the mapping class group, $Mod_N$, where $y$ is the cross-cap slide of $a_3$ along $c_2$, and $t_x$ is the Dehn twist about $x$, see Figure \ref{fig-gen}, see also Theorem 3.15 given by the author in \cite{Ir10} and Theorem 4.14 given by Korkmaz in \cite{K2}. 
  
\begin{lemma} \label{L_f_3} $\forall \ f \in G= \{t_{c_1}, t_{c_2}, y\}$, $\exists$ a set $L_f \subset \mathcal{B}$ such that the pointwise stabilizer of $L_f$ is the center of $Mod_N$ and $f(L_f) \subset \mathcal{B}$.\end{lemma}

\begin{proof} The proof follows as in the proof of Lemma 4.2 in \cite{Ir10}. Consider the curves in Figure \ref{fig-new}. Let $V_1 = \{c_1, c_2, w\}$ and $V_2 = \{c_1, c_2, j\}$. 
The pointwise stabilizer of each of $V_1$ and $V_2$ is the center of $Mod_N$. The center of $Mod_N$ is isomorphic to $\mathbb{Z}_2$ (Theorem 6.1 in \cite{St2}). Let $f \in G$. For $f=c_1$, let $L_f = V_1$. We have $t_{c_1}(c_1)= c_1$, $t_{c_1}(c_2)= j$, $t_{c_1}(w)= c_2$ and $f(L_f) \subset \mathcal{B}$. For $f=c_2$, let $L_f = V_2$. We have $t_{c_2}(c_1)= w$, $t_{c_2}(c_2)= c_2$, $t_{c_2}(j)= c_1$ and $f(L_f) \subset \mathcal{B}$. For $f=y$, let 
$L_f= V_1$. We have $y(c_1)= c_1$, $y(c_2)= c_2$, $y(w)=j$ and
$f(L_f) \subset \mathcal{B}$.\end{proof}\\

The proof of the following theorem is similar to the proof of Theorem 3.14 given by the author in \cite{Ir10}. 
  
\begin{theorem} \label{B-3} If $(g,n) = (3,0)$, then there exists a sequence $\mathcal{E}_1 \subset \mathcal{E}_2 \subset \dots \subset \mathcal{E}_n \subset \dots$  such that 
		
		(i) $\mathcal{E}_i$ is a finite rigid set in $\mathcal{C}(N)$ for all $i \in \mathbb{N}$, and 
		
		(ii) $\bigcup_{i \in \mathbb{N}} \mathcal{E}_i = \mathcal{C}(N)$. \end{theorem}

\begin{proof} The proof will follow similar to the proof given by the author in section 4 in \cite{Ir10}. We note that the difference is that to control the image of rigid sets under locally injective simplicial maps, we used a bigger set (with more curves $u, v$) than the set considered in Lemma 4.1 in \cite{Ir10}. So, even though there are similarities, the sequence we will construct here whose elements are rigid sets will be different from the sequence we constructed whose elements were superrigid sets in \cite{Ir10}. 
	
We can see that if $x$ is a vertex in $\mathcal{C}(N)$ then there 
exists $r \in Mod_N$ and a vertex $y$ in $\mathcal{B}$ such that $r(y)=x$. Let $\mathcal{E}_1 = \mathcal{B}$ and $\mathcal{E}_k = \mathcal{E}_{k-1} \cup (\bigcup _{f \in G} (f(\mathcal{E}_{k-1}) \cup f^{-1}(\mathcal{E}_{k-1})))$ for $k \geq 2$. We see that 
$\bigcup _{k=1} ^{\infty} \mathcal{E}_k = \mathcal{C}(N)$. We will prove that $\mathcal{E}_k$ is a rigid set for every $k$ by induction on $k$. 
By Lemma \ref{rigid}, $\mathcal{E}_1$ is a rigid set. Assume that $\mathcal{E}_{k-1}$ is rigid for some $k \geq 2$. We want to see that $\mathcal{E}_{k}$ is rigid. Let $\lambda: \mathcal{E}_k \rightarrow \mathcal{C}(N)$ be a locally injective simplicial map. It is easy to see that $\lambda$ restricts to a locally injective simplicial map on $\mathcal{E}_{k-1}$. Since $\mathcal{E}_{k-1}$ is rigid, there exists a 
homeomorphism $h: N \rightarrow N$ such that $h([x]) = \lambda([x])$ for all $x \in \mathcal{E}_{k-1}$. We will show that $h([x]) = \lambda([x])$ for each $x \in \mathcal{E}_k = \mathcal{E}_{k-1} \cup (\bigcup _{f \in G} (f(\mathcal{E}_{k-1}) \cup f^{-1}(\mathcal{E}_{k-1})))$.
Let $f \in G$. Since $\mathcal{E}_{k-1}$ is a rigid set, $f(\mathcal{E}_{k-1})$ is also a rigid set. Since the map $\lambda$ restricts to a locally injective simplicial map on $f(\mathcal{E}_{k-1})$, there exists a homeomorphism $h_f$ of $N$ such that $h_f([x]) = \lambda([x])$ for all $x \in f(\mathcal{E}_{k-1})$. By Lemma \ref{L_f_3}, there exists $L_f \subset \mathcal{B}$ such that $f(L_f) \subset \mathcal{B}$ and 
the pointwise stabilizer of $L_f$ is the center of $Mod_N$. Since $L_f \subset \mathcal{B} \subset \mathcal{E}_{k-1}$ and $f(L_f) \subset \mathcal{B} \subset \mathcal{E}_{k-1}$, we have $f(L_f) \subset \mathcal{E}_{k-1} \cap f(\mathcal{E}_{k-1})$. Hence, $h_f = h$ or $h_f = h \circ i$ where $i$ is the generator for the center of $Mod_N$ (the center is isomorphic to $\mathbb{Z}_2$). Similarly, the map $\lambda$ restricts to a locally injective simplicial map on $f^{-1}(\mathcal{E}_{k-1})$. Since $\mathcal{E}_{k-1}$ is a rigid set, $f^{-1}(\mathcal{E}_{k-1})$ is also a rigid set. So, there exists a homeomorphism 
$h'_f$ of $N$ such that $h'_f([x]) = \lambda([x])$ for all $x \in f^{-1}(\mathcal{E}_{k-1})$. Since 
$L_f \subset \mathcal{B} \subset \mathcal{E}_{k-1}$ and $f(L_f) \subset \mathcal{B} \subset \mathcal{E}_{k-1}$, we have $L_f \subset \mathcal{E}_{k-1} \cap f^{-1}(\mathcal{E}_{k-1})$. So, $h'_f = h$ or $h'_f = h \circ i$. Since $i([x])=[x]$ for every vertex $[x] \in \mathcal{C}(N)$ we get $h([x]) = \lambda([x])$ for each $x \in f(\mathcal{E}_{k-1}) \cup f^{-1}(\mathcal{E}_{k-1})$. Since this is true for each $f \in G$, we have
$h([x]) = \lambda([x])$ for each $x \in \mathcal{E}_k$. So, $\mathcal{E}_k$ is a rigid set. By induction $\mathcal{E}_k$ is a finite rigid set for all $k \geq 1$.\end{proof}
 
\section{Exhaustion of $\mathcal{C}(N)$ by finite rigid sets when 
$g+n \geq 5$}

 
\begin{figure}
	\begin{center}  \epsfxsize=2.59in \epsfbox{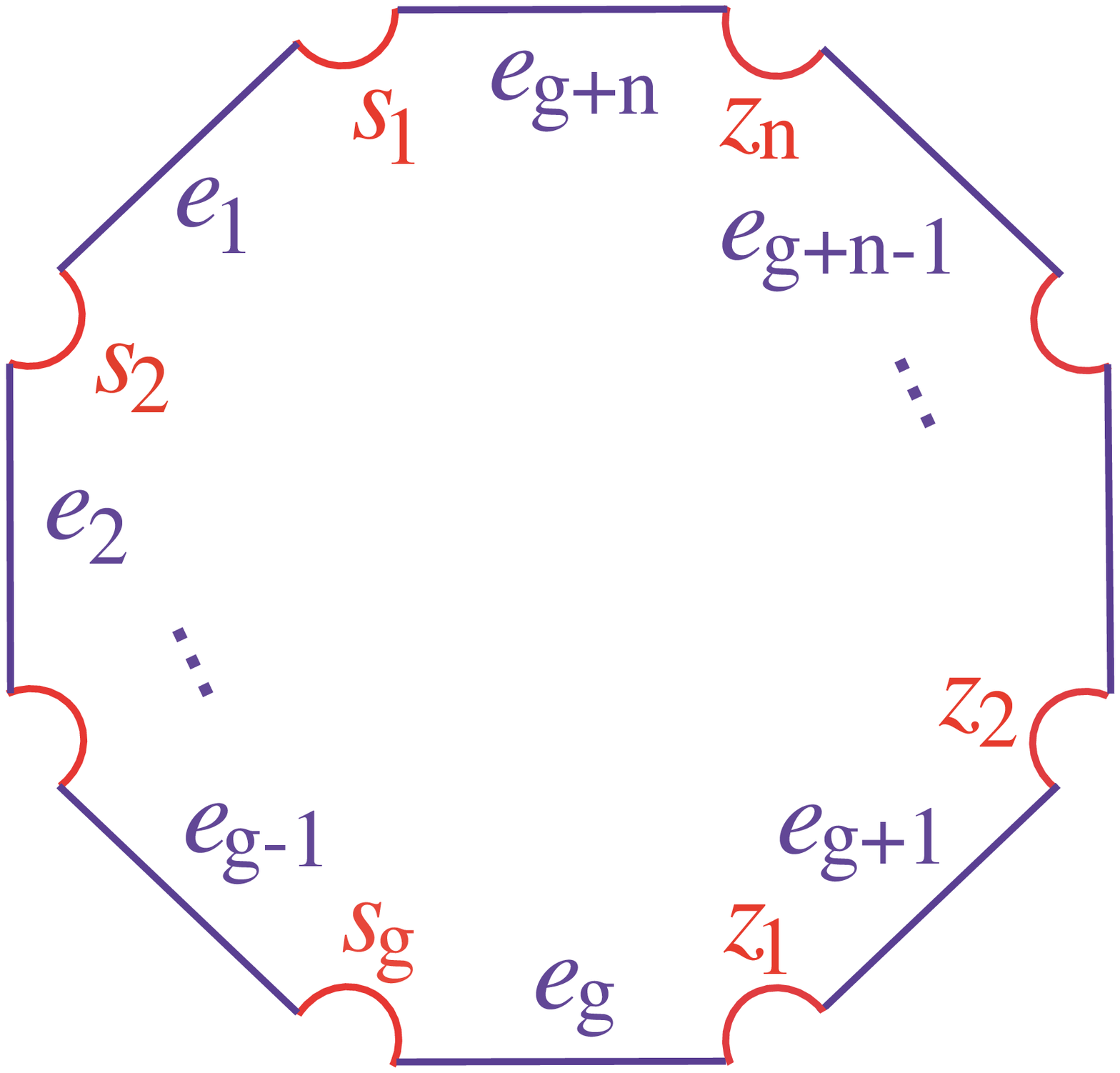}  \hspace{0.1cm}
		 \epsfxsize=2.29in \epsfbox{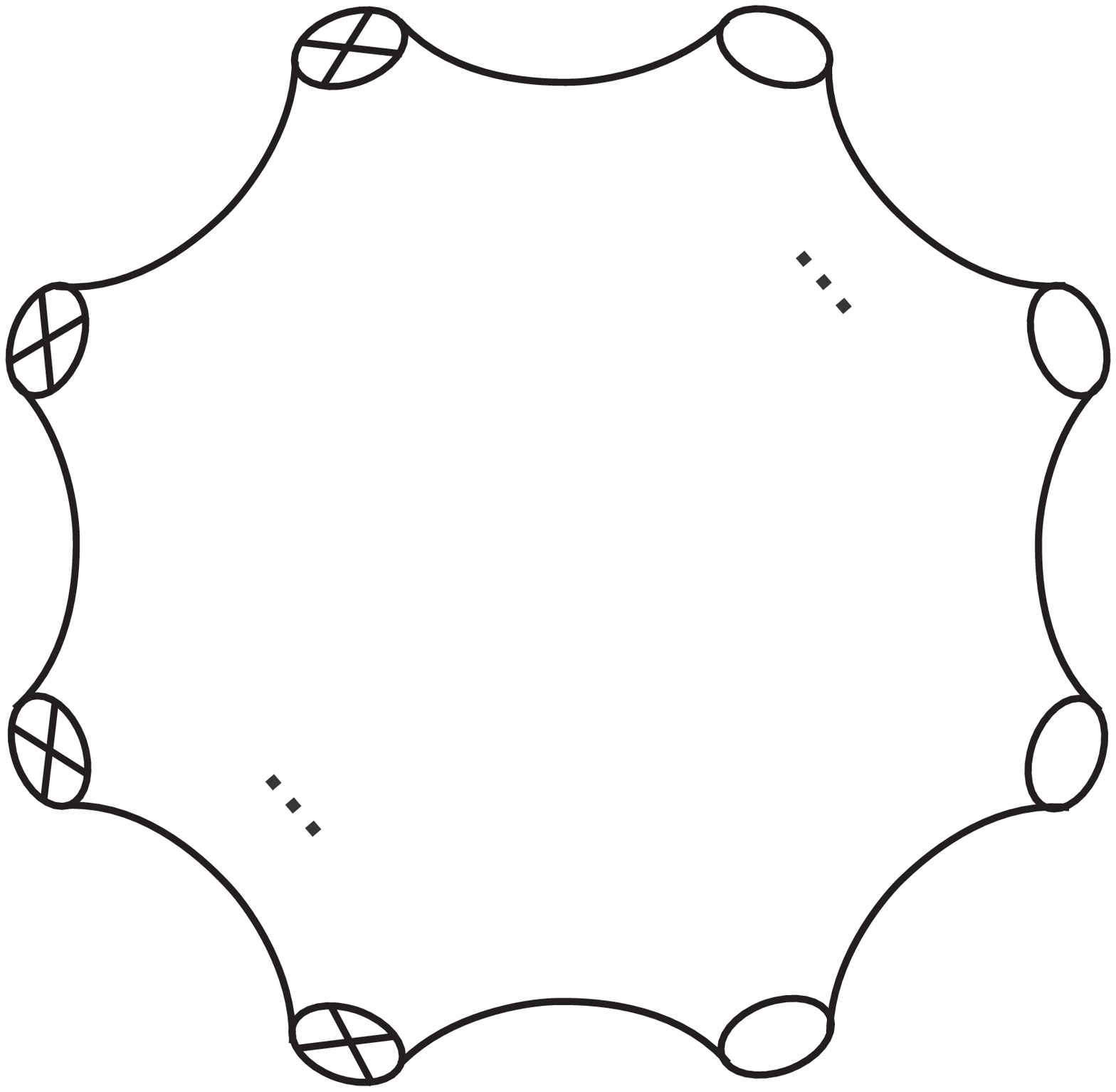} 
 
  \hspace{0.3cm}		(i)   \hspace{5.6cm} (ii)  
		\caption {Nonorientable surface obtained by gluing two $(2g + 2n)$-gons} \label{fig-222}
	\end{center}
\end{figure}

Ilbira-Korkmaz constructed finite rigid subcomplexes in $\mathcal{C}(N)$ in \cite{IlK}. Their construction starts with a disk $D$ in the plane whose boundary is a $(2g + 2n)$-gon with labeled edges $s_1, e_1, s_2, e_2, s_3, e_3, \dots, s_g, e_g, z_1, e_{g+1}, z_2, e_{g+2}, \dots, e_{g+n-1}, z_n, e_{g+n}$, see Figure \ref{fig-222} (i). 
By glueing two copies of this disk $D$ along $e_i$ for each $i$ one gets a sphere $S$ with 
$g+n$ holes. A nonorientable surface $N$ of genus $g$ with $n$ boundary components is obtained by identifying the antipodal points on each boundary component of $S$ that are formed by $s_i$ for each $i$. Let $a_i$ be the 1-sided curve on $N$ that corresponds to the arc $s_i$ on $D$.
Let $a_{i,j}$ be the 1-sided curve on $N$ that corresponds to a line segment joining the midpoint of $s_i$ and a point of $e_j$ on $D$. Let $b_{i,j}$ be the 2-sided curve on $N$ that corresponds to a line segment joining a point of $e_i$ to a point of $e_j$ on $D$.
Let $\mathcal{X} = \{a_i, a_{i,j} :  1 \leq i \leq g, 1 \leq j \leq g + n, j \neq i, j \neq i-1 \ (mod \ (g + n))\} \cup  \{b_{i,j} :  1 \leq i, j \leq g+n, 2 \leq |i-j| \leq g+n-2\}$. 

\begin{theorem} \label{result} (Ilbira-Korkmaz) Let $N$ be a compact, connected, nonorientable surface of genus $g$ with $n$ boundary components. If $g + n \neq 4$, then $\mathcal{X}$ is a finite rigid set in $\mathcal{C}(N)$.\end{theorem}

\begin{figure}[t]
	\begin{center} 
	    \epsfxsize=2.19in \epsfbox{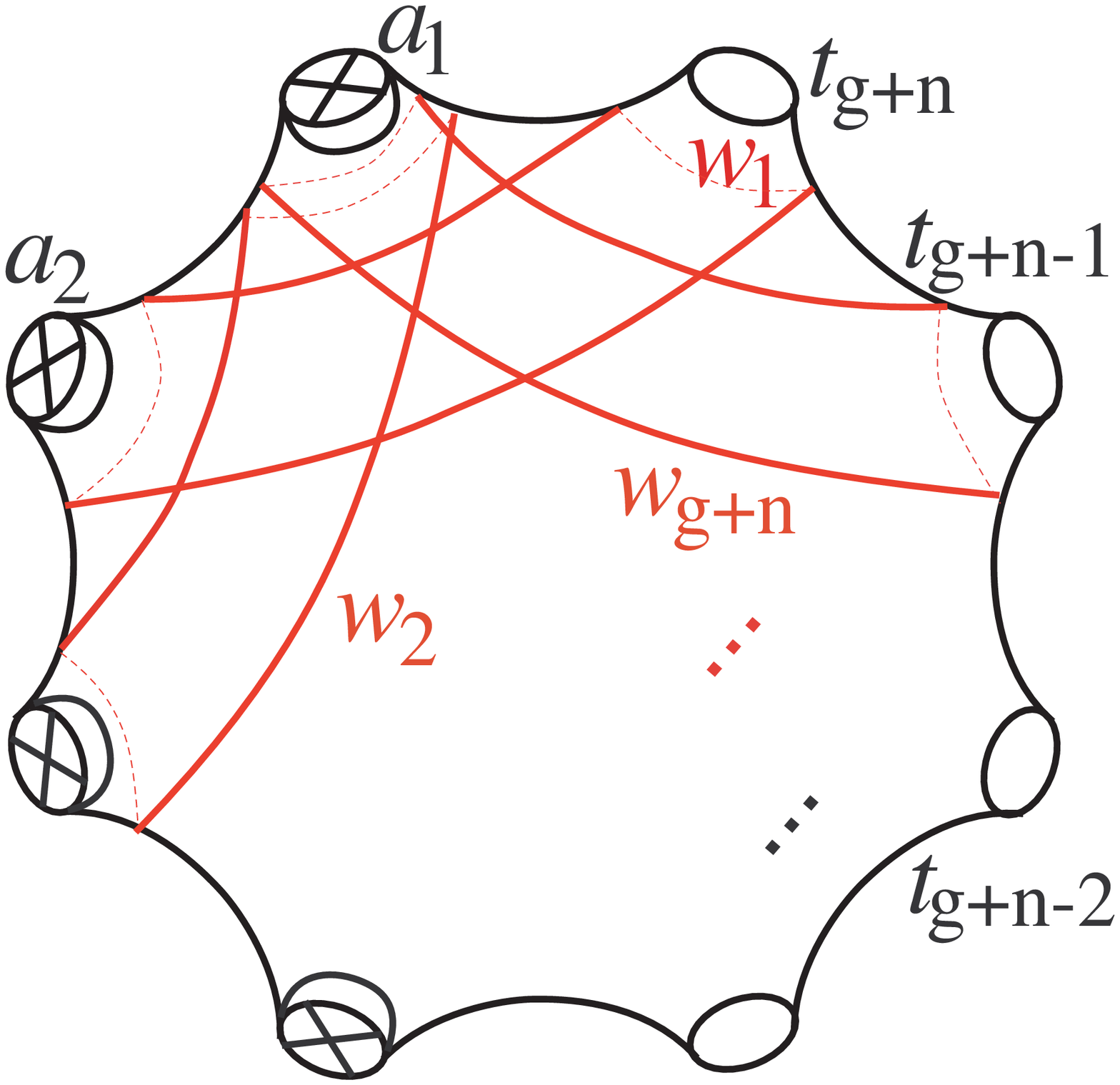} \hspace{0.19cm} 
		\epsfxsize=2.19in \epsfbox{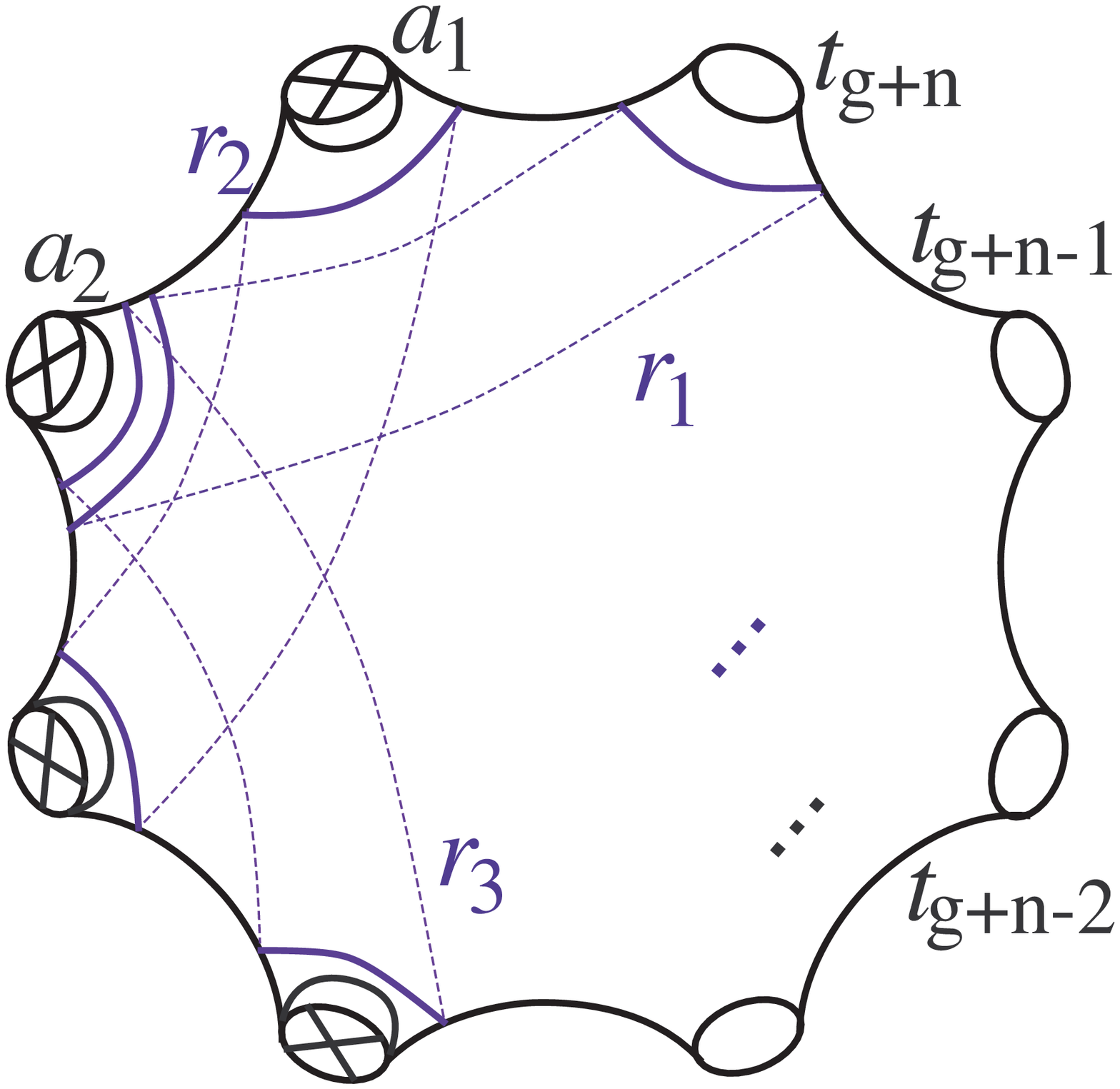}
		
		(i)   \hspace{5.4cm} (ii)  
		
		\epsfxsize=2.19in \epsfbox{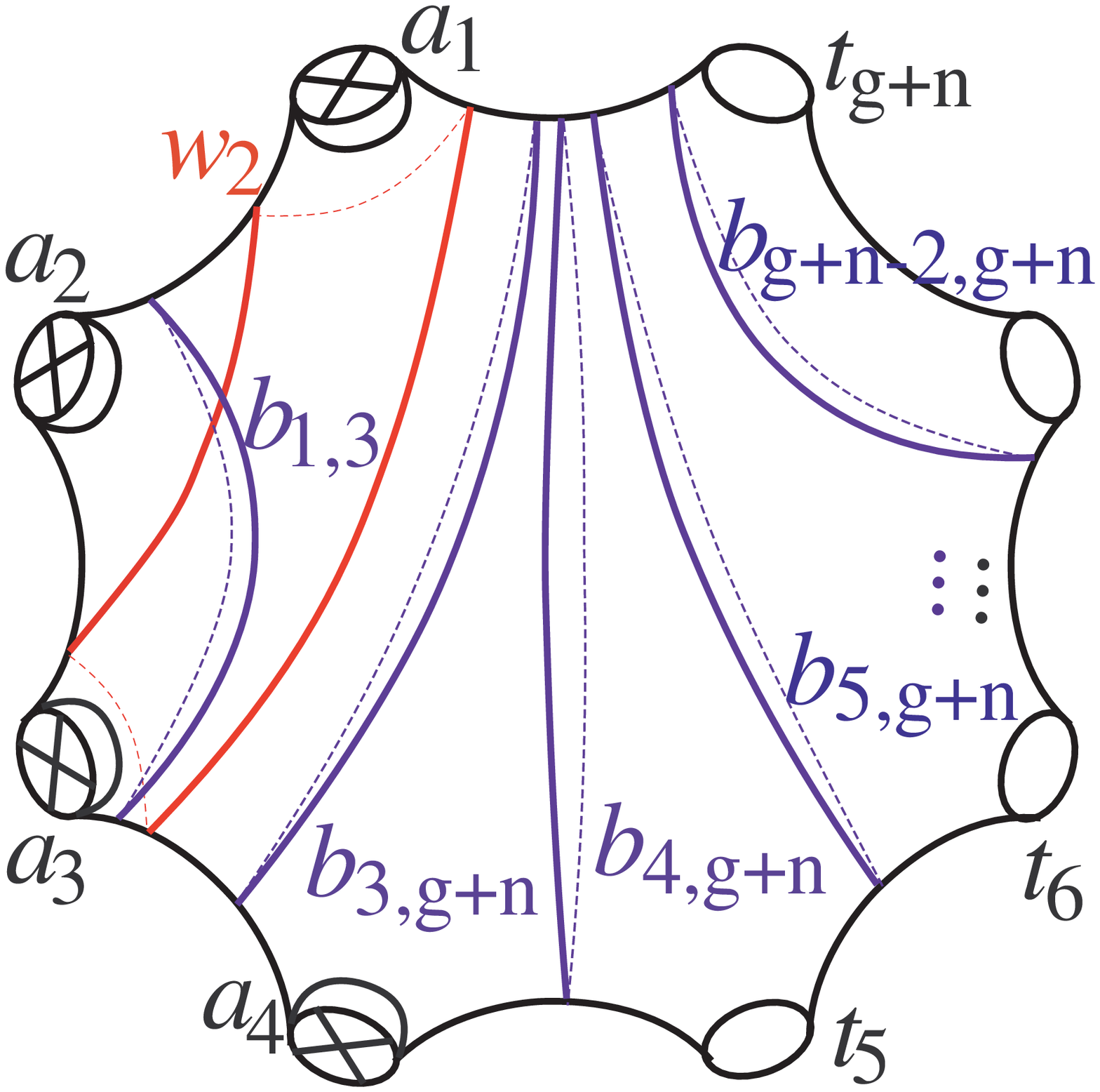} \hspace{0.19cm} 
		\epsfxsize=2.19in \epsfbox{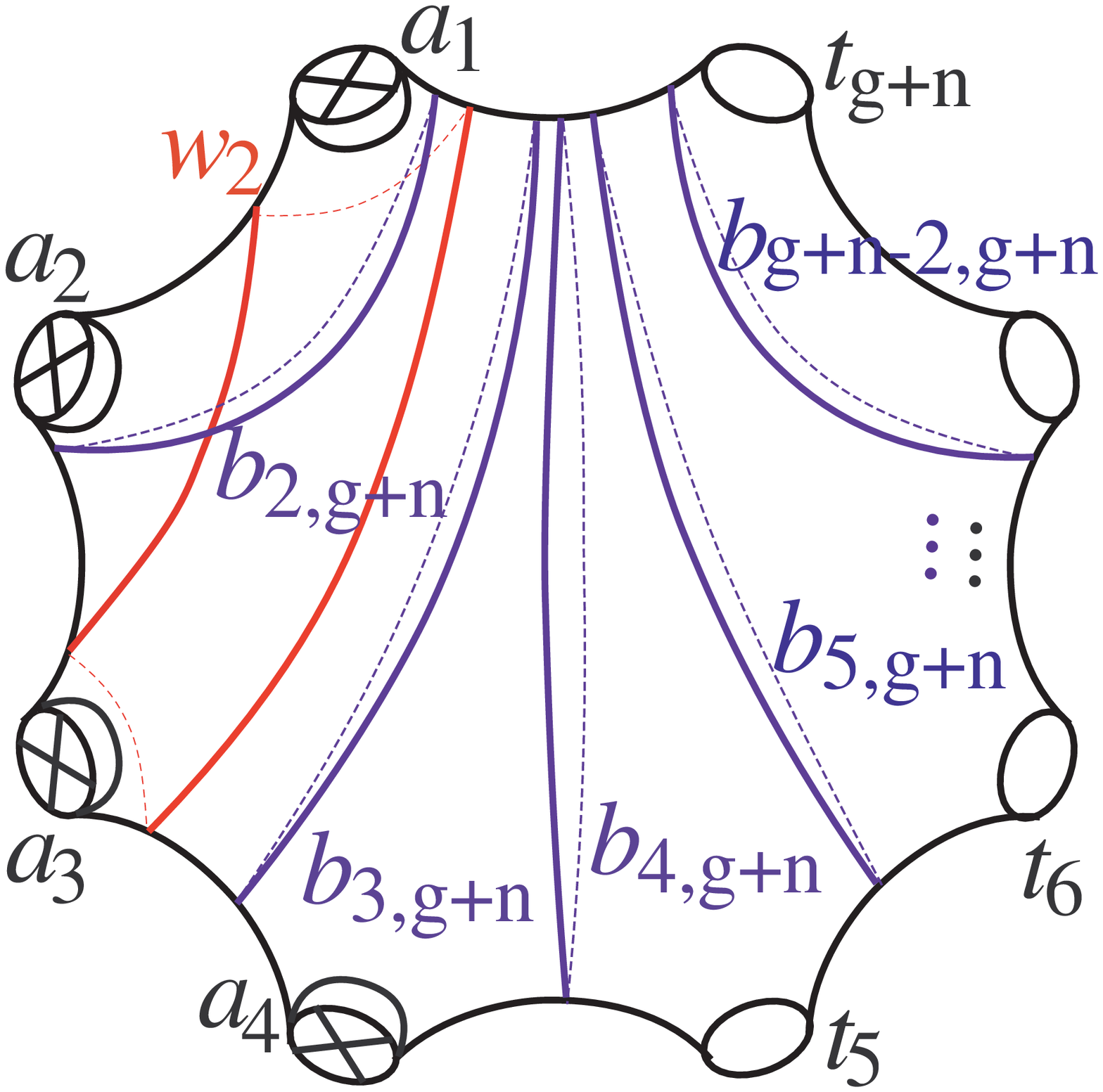}
		
		(iii)   \hspace{5.2cm} (iv)  
		\caption {Curves in $\mathcal{B}_2$ and controlling intersections} \label{fig-s11}
	\end{center}
\end{figure}

From now on we will always assume that $g+n \geq 5$.
Let $\mathcal{B}_1 = \mathcal{X}$. Let $w_i, r_i$ for $i= 1, 2, \dots, g+n$ be as shown in Figure \ref{fig-s11} (i) and (ii). 
Let $\mathcal{B}_2 = \{w_1, w_2,  \dots, w_{g+n}, r_1, r_2, \dots,$ $ r_{g+n}\}$ on $N$.

\begin{lemma} \label{1a} If $\lambda: \mathcal{B}_1 \cup \{w_1, w_2, \dots, w_{g+n}\} \rightarrow \mathcal{C}(N)$ is a locally injective simplicial map, then

$i(\lambda([w_1]), \lambda([b_{1,g+n-1}])) \neq 0$, $i(\lambda([w_1]), \lambda([b_{2,g+n}])) \neq 0$,

$i(\lambda([w_2]), \lambda([b_{1,3}])) \neq 0$, $i(\lambda([w_2]), \lambda([b_{2,g+n}])) \neq 0$,

$i(\lambda([w_3]), \lambda([b_{1,3}])) \neq 0$, $i(\lambda([w_3]), \lambda([b_{2,4}])) \neq 0, \dots,$

$ i(\lambda([w_{g+n}]), \lambda([b_{1,g+n-1}])) \neq 0$, $i(\lambda([w_{g+n}]), \lambda([b_{g+n-2,g+n}])) \neq 0$.\end{lemma}

\begin{proof} We will give the proof when $g=4$ and $n \geq 1$. The proof for the other cases is similar. By Theorem \ref{result}, there exists a homeomorphism $h: N \rightarrow N$ such that $h([x]) = \lambda([x])$ for every $x$ in $\mathcal{B}_1$. Consider the curves given in Figure \ref{fig-s11} (iii) and (iv). Let $P=\{a_1, a_2, a_3, a_4, b_{1,3}, b_{3,g+n}, b_{4,g+n}, b_{5,g+n}, \dots, b_{g+n-2,g+n}\}$. The set $P$ is a top dimensional pants decomposition on $N$. We see that $i([w_2], [x]) = 0$ for all $x \in P \setminus \{b_{1,3}\}$. Since $\lambda$ is a locally injective simplicial map and it preserves geometric intersection zero, we have $i(\lambda([w_2]), \lambda([x])) = 0$ for all $x \in P \setminus \{b_{1,3}\}$. This gives $i(\lambda([w_2]), \lambda([b_{1,3}])) \neq 0$. Let $Q=\{a_1, a_2, a_3, a_4, b_{2,8}, b_{3,g+n}, b_{4,g+n}, b_{5,g+n}, \dots, b_{g+n-2,g+n}\}$. The set $Q$ is a top dimensional pants decomposition on $N$. We see that $i([w_2], [x]) = 0$ for all $x \in Q \setminus \{b_{2,g+n}\}$. We have $i(\lambda([w_2]), \lambda([x])) = 0$ for all $x \in Q \setminus \{b_{2,g+n}\}$. This gives $i(\lambda([w_2]), \lambda([b_{2,g+n}])) \neq 0$. With similar arguments we get all the intersection information listed in the statement.\end{proof} 

\begin{lemma} \label{1b} If $\lambda: \mathcal{B}_1 \cup \{r_1, r_2, \dots, r_{g+n}\} \rightarrow \mathcal{C}(N)$ is a locally injective simplicial map, then

$i(\lambda([r_1]), \lambda([b_{1,g+n-1}])) \neq 0$, $i(\lambda([r_1]), \lambda([b_{2,g+n}])) \neq 0$,

$i(\lambda([r_2]), \lambda([b_{1,3}])) \neq 0$, $i(\lambda([r_2]), \lambda([b_{2,g+n}])) \neq 0$,

$i(\lambda([r_3]), \lambda([b_{1,3}])) \neq 0$, $i(\lambda([r_3]), \lambda([b_{2,4}])) \neq 0, \dots,$  

$i(\lambda([r_{g+n}]), \lambda([b_{1,g+n-1}])) \neq 0$, $i(\lambda([r_{g+n}]), \lambda([b_{g+n-2,g+n}])) \neq 0$. \end{lemma}

\begin{proof} The proof is similar to the proof of Lemma \ref{1a}. \end{proof}\\

\begin{figure}[t]
	\begin{center} 
		\epsfxsize=2.19in \epsfbox{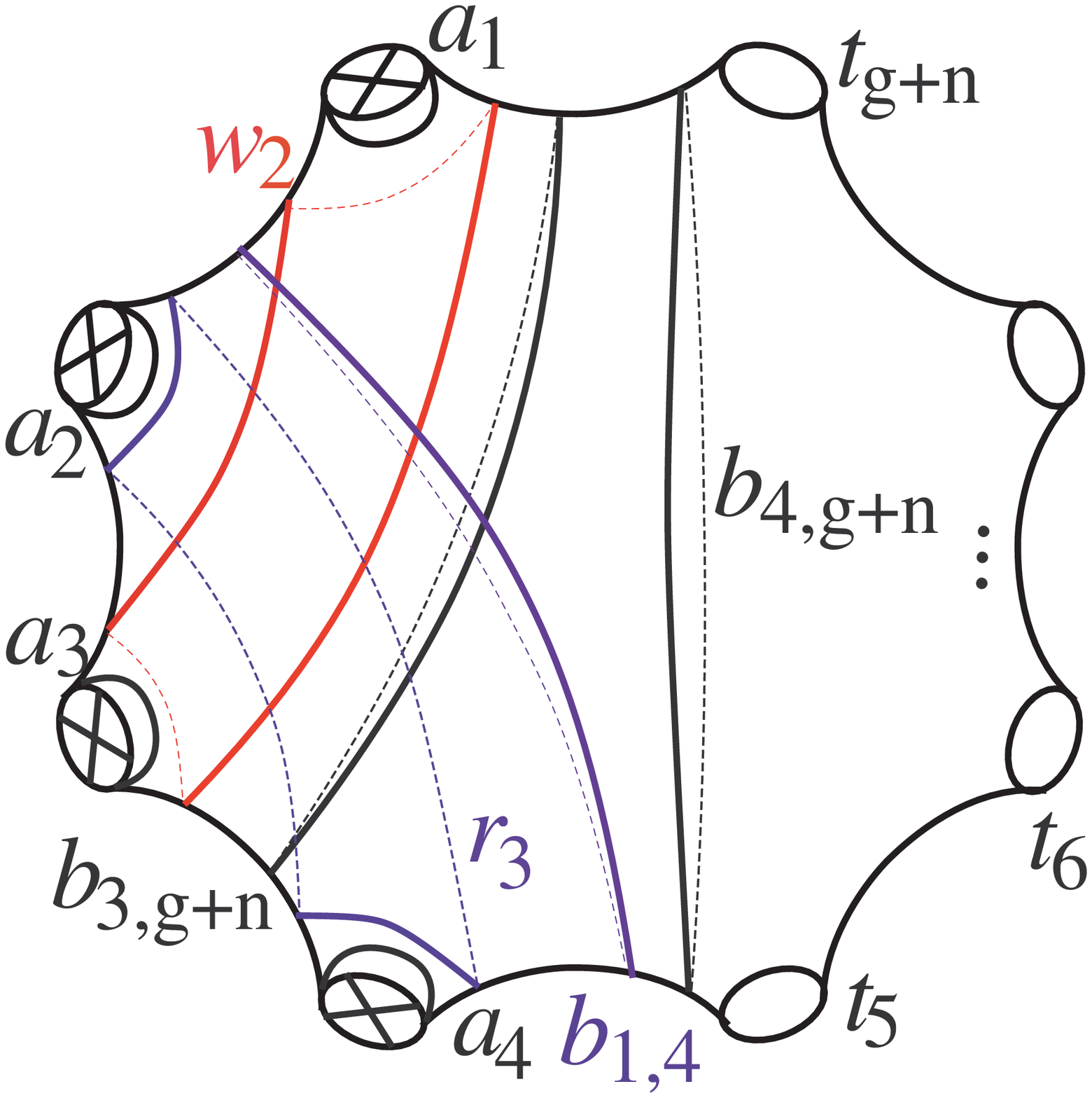} \hspace{0.19cm} 
		\epsfxsize=2.19in \epsfbox{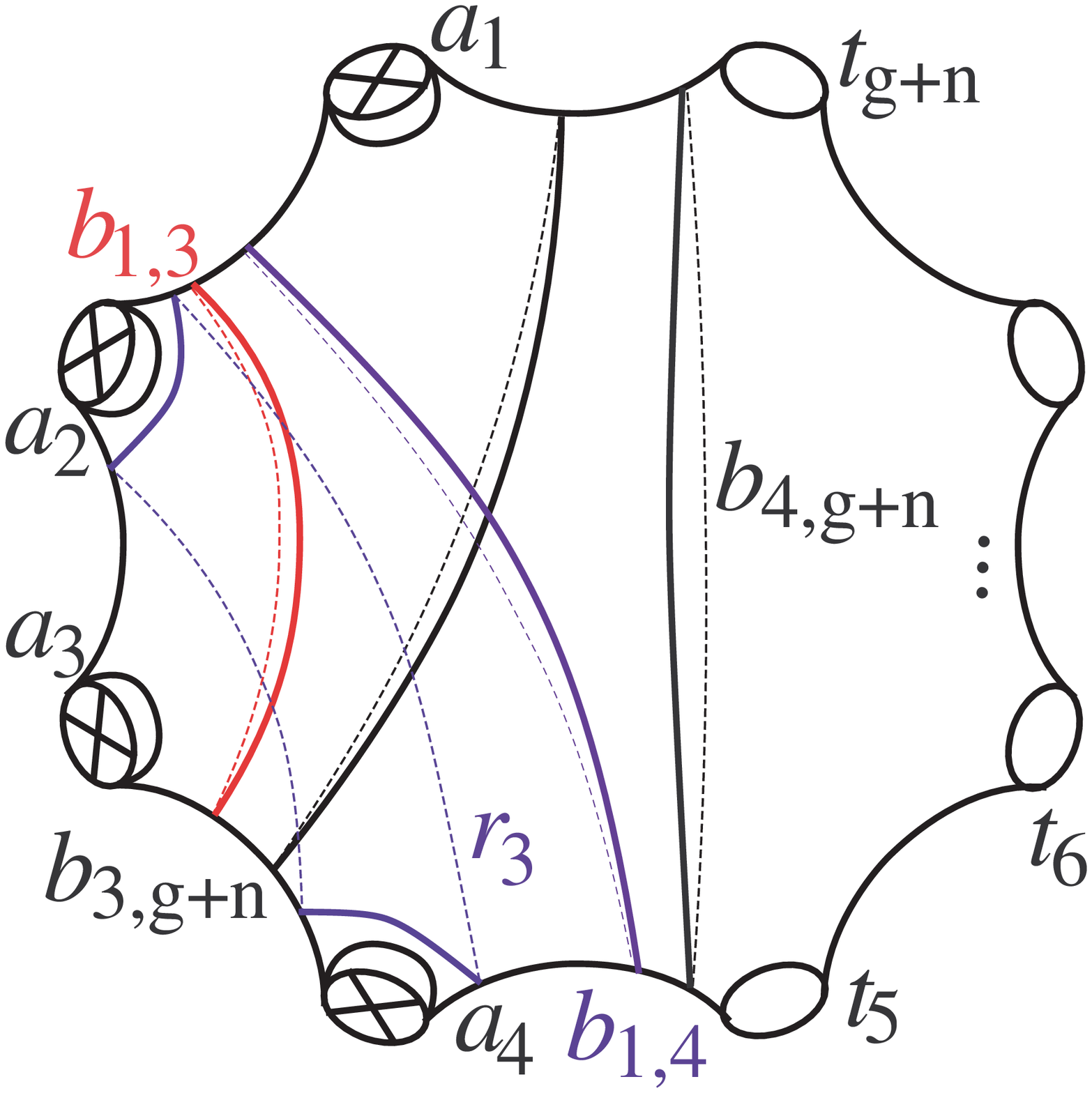}
		
\hspace{0.1cm}		(i)   \hspace{5.4cm} (ii)  
		
		\epsfxsize=2.19in \epsfbox{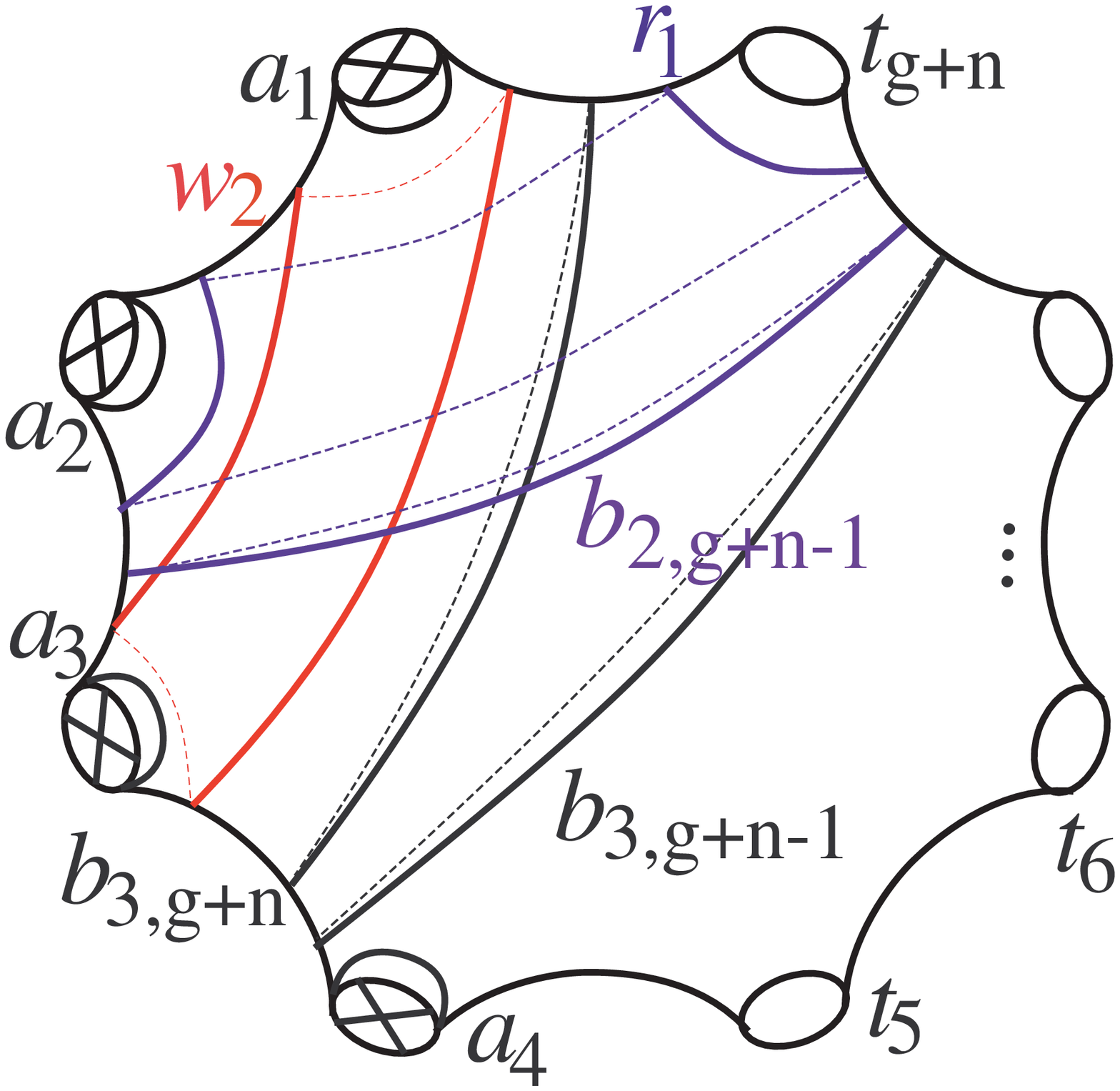} \hspace{0.19cm} 
		\epsfxsize=2.19in \epsfbox{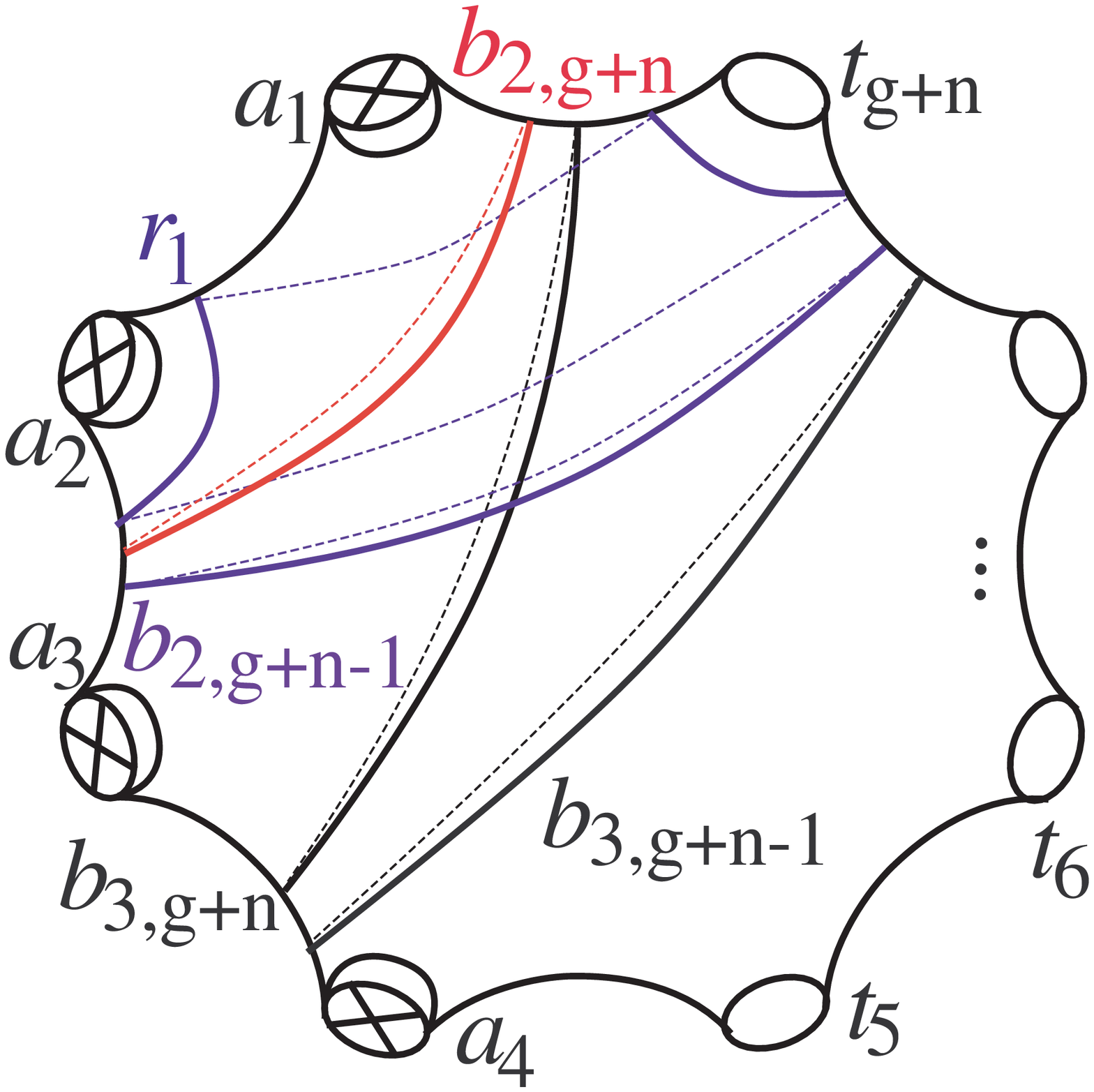}
		
	\hspace{0.2cm}	(iii)   \hspace{5.2cm} (iv)  
		\caption {Controlling geometric intersection two} \label{fig-2a}
	\end{center}
\end{figure}

\begin{figure}[t]
	\begin{center} 
		\epsfxsize=2.49in \epsfbox{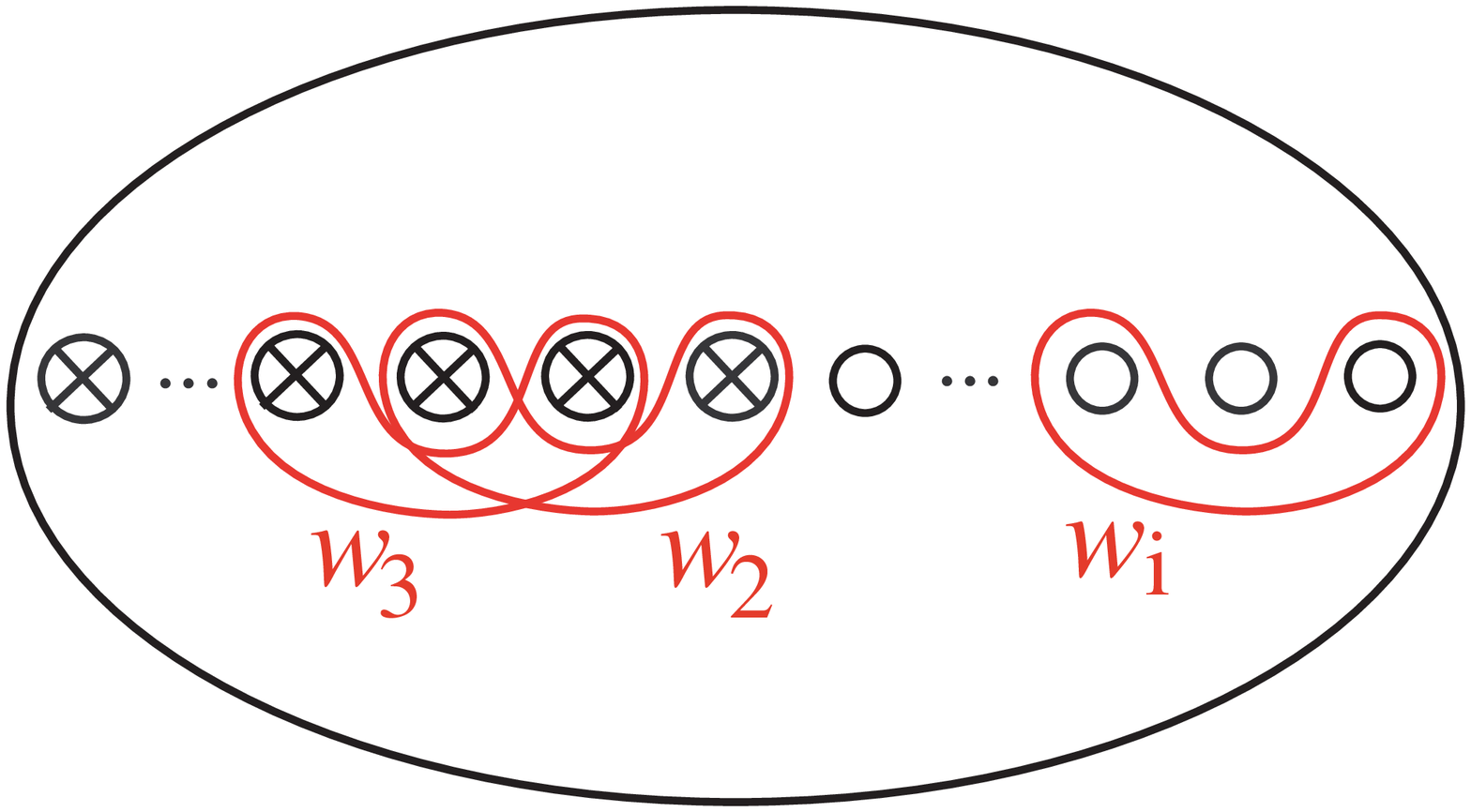} \hspace{0.09cm} 	
		\epsfxsize=2.49in \epsfbox{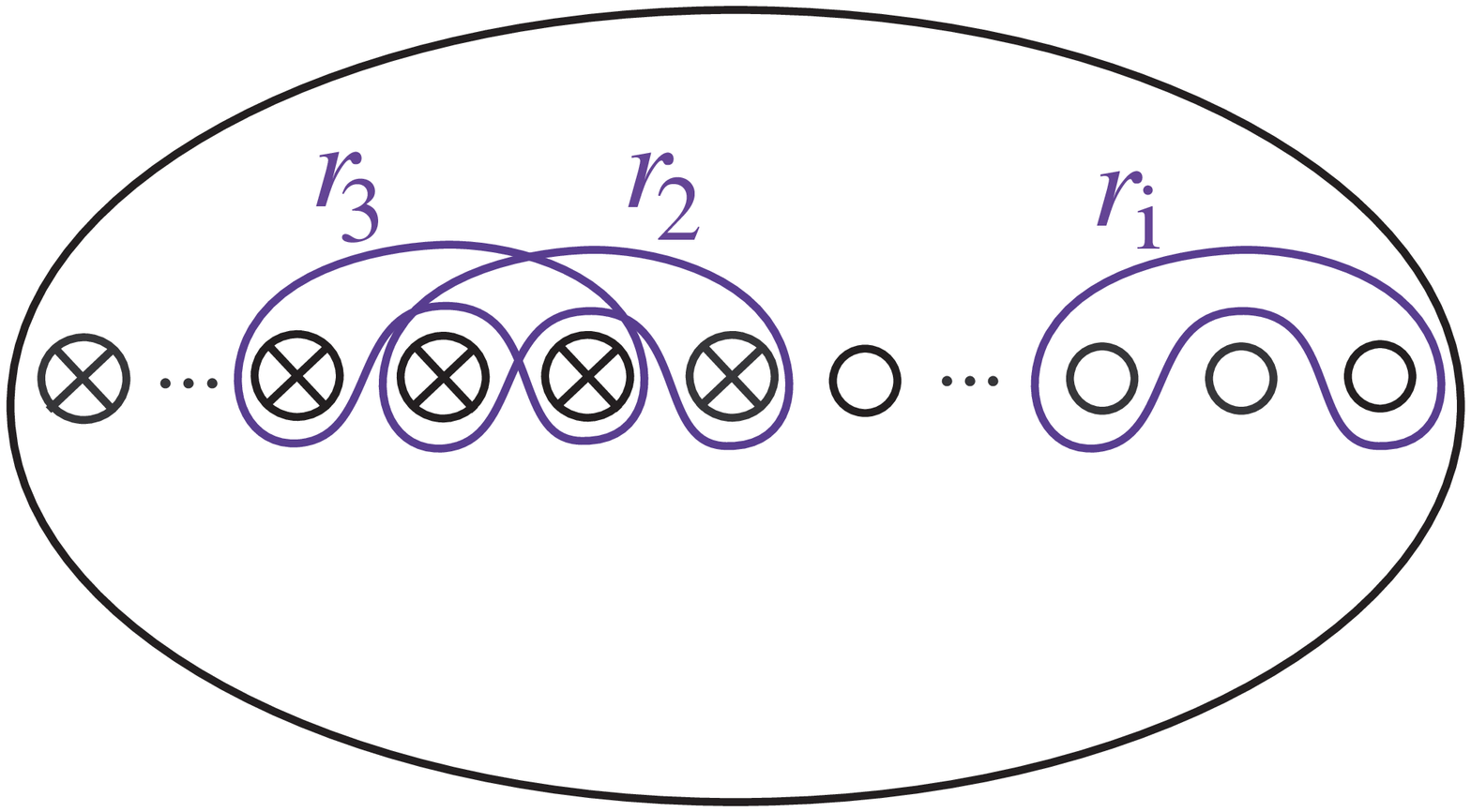} 
		
		\caption{Curves in $\mathcal{B}_2$ }
		\label{Fig-2-a}
	\end{center}
\end{figure}

In the proof of the following lemma we will be using pentagons in $\mathcal{C}(N)$. We remind that five vertices $v_1, v_2, \dots, v_5$ in $\mathcal{C}(N)$ form a pentagon if there are edges between $v_i$ and $v_{i+1}$ for $i=1, 2, 3, 4$ and there is an edge between $v_5$ and $v_1$ and there are no other edges between these five vertives in $\mathcal{C}(N)$.  

\begin{lemma} \label{1c} If $\lambda: \mathcal{B}_1 \cup \{w_1, w_2,  \dots, w_{g+n}, r_1, r_2, \dots, r_{g+n}\} \rightarrow \mathcal{C}(N)$ is a locally injective simplicial map, then we have the following geometric intersections
	
$i(\lambda([w_1]), \lambda([b_{1,g+n-1}])) =2$, $i(\lambda([w_1]), \lambda([b_{2,g+n}])) =2$,

$i(\lambda([w_2]), \lambda([b_{1,3}])) =2$, 
$i(\lambda([w_2]), \lambda([b_{2,g+n}])) =2$, 

$i(\lambda([w_3]), \lambda([b_{1,3}])) =2$, 
$i(\lambda([w_3]), \lambda([b_{2,4}])) =2, \dots,$ 

$i(\lambda([w_{g+n}]), \lambda([b_{1,g+n-1}])) =2$, $i(\lambda([w_{g+n}]), \lambda([b_{g+n-2,g+n}])) =2$.   \end{lemma}

\begin{proof} We will give the proof for $g=4, n \geq 1$. The other cases are similar. The map $\lambda$ restricts to locally injective simplicial maps on $\mathcal{B}_1$ and $\mathcal{B}_1 \cup \{w_1, w_2, \dots, w_{g+n}\}$. By Theorem \ref{result}, there exists a homeomorphism $h$ of $N$ such that $h([x]) = \lambda([x])$ for every $x$ in $\mathcal{B}_1$. To see that $i(\lambda([w_1]), \lambda([b_{1,3}])) =2$ consider the curves in Figure \ref{fig-2a} (i) and (ii). Let $S$ be the five holed sphere bounded by $a_1, a_2, a_3, a_4, b_{4,g+n}$. There is a pentagon formed by 
$b_{3, g+n}, b_{1,3}, b_{1,4}, r_3, w_2$ in $C(S)$. If $x, y \in \{b_{3, g+n}, b_{1,3}, b_{1,4}, r_3, w_2\}$ and $i([x], [y]) = 0$ then we can see that $i(\lambda([x]), \lambda([y])) = 0$ since $\lambda$ is locally injective. Let $x, y \in \{b_{3, g+n}, b_{1,3}, b_{1,4}, r_3, w_2\}$ such that $x$ and $y$ intersect. If both of $x, y$ are in $\mathcal{B}_1$ 
then we know that $i(\lambda([x]), \lambda([y])) \neq 0$ since there exists a homeomorphism $h: N \rightarrow N$ such that $h([x]) = \lambda([x])$ for every $x$ in $\mathcal{B}_1$. If both of $x, y$ are not in $\mathcal{B}_1$ we can still see that $i(\lambda([x]), \lambda([y])) \neq 0$ as follows: If Lemma \ref{1a} applies to the pair (for example it applies when $x= w_2, y=b_{1,3}$) we use the lemma to see the result. If Lemma \ref{1a} doesn't apply to the pair $x, y$ we use a proof similar to the proof of that lemma to see the result: We can complete $x$ to a top dimensional pants decomposition on $N$ using the curves in  $\mathcal{B}_1 \cup \{w_1, w_2,  \dots, w_{g+n}, r_1, r_2, \dots, r_{g+n}\}$
such that $i([y], [z]) = 0$ for all $z \in P \setminus \{x\}$. Since $\lambda$ is a locally injective simplicial map and it preserves geometric intersection zero, we will have $i(\lambda([y]), \lambda([z])) = 0$ for all $z \in P \setminus \{x\}$. This implies that $i(\lambda([x]), \lambda([y])) \neq 0$. This argument shows that minimally intersecting representatives of $\lambda([b_{3, g+n}]), \lambda([b_{1,3}]), \lambda([b_{1,4}]), \lambda([r_3]), \lambda([w_2])$ form a pentagon in the five holed sphere bounded by $a'_1, a'_2, a'_3, a'_4, b'_{4,g+n}$, where $a'_1, a'_2, a'_3, a'_4, b'_{4,g+n}$ are minimally intersecting representatives of $\lambda([a_1]),  \lambda([a_2]),$ $\lambda([a_3]),\lambda([a_4]), \lambda([b_{4,g+n}])$ respectively. This implies that $i(\lambda[w_2]), \lambda([b_{1,3}])) =2$ by Korkmaz's Theorem 3.2 in \cite{K1}. The other intersections can be obtained in a similar way.\end{proof}

\begin{lemma} \label{1d} If $\lambda: \mathcal{B}_1 \cup \{w_1, w_2,  \dots, w_{g+n}, r_1, r_2, \dots, r_{g+n}\} \rightarrow \mathcal{C}(N)$ is a locally injective simplicial map, then we have the following geometric intersections
	
$i(\lambda([r_1]), \lambda([b_{1,g+n-1}])) =2$, $i(\lambda([r_1]), \lambda([b_{2,g+n}])) =2$,

$i(\lambda([r_2]), \lambda([b_{1,3}])) =2$, 
$i(\lambda([r_2]), \lambda([b_{2,g+n}])) =2$, 

$i(\lambda([r_3]), \lambda([b_{1,3}])) =2$, 
$i(\lambda([r_3]), \lambda([b_{2,4}])) =2, \dots,$  

$i(\lambda([r_{g+n}]), \lambda([b_{1,g+n-1}])) =2$, $i(\lambda([r_{g+n}]), \lambda([b_{g+n-2,g+n}])) =2$.   \end{lemma}
 
\begin{proof} The proof is similar to the proof of Lemma \ref{1c} (use Lemma \ref{1b}).\end{proof}\\

Let $\mathcal{B}_2 = \{w_1, w_2,  \dots, w_{g+n}, r_1, r_2, \dots, r_{g+n}\}$. In Figure \ref{Fig-2-a} we show the curves $r_i, w_i$.   
  
\begin{lemma} \label{B_2} $\mathcal{B}_1 \cup \mathcal{B}_2$ is a finite rigid set with trivial pointwise stabilizer.\end{lemma}

\begin{proof} We can see the proof by using Lemma \ref{1c} and Lemma \ref{1d}. The proof is similar to the proof of 
Lemma 3.6 in \cite{Ir10}.\end{proof}\\

From now on in the figures we will fill some disks with stripes to mean that the boundary of each disk either bounds a M\"{o}bius band or is isotopic to a boundary component of $N$. In Figure \ref{Fig-29} (i) we see a disk which is filled with stripes. This will mean that (i) If $2 \leq i \leq g-1$, then the boundary of the disk bounds a  M\"{o}bius band and has $a_{i-1}$ as the core; (ii) If $i=1$ and $N$ is closed, then
the boundary of the disk bounds a M\"{o}bius band and has $a_{g}$ as the core; and (iii) If $i=1$ and $N$ is not closed, then the boundary of the disk is isotopic to a boundary component of $N$. In Figure \ref{Fig-29}, Figure \ref{Fig-30} and Figure \ref{Fig-3-aa} we assume that the indices $i$ and $j$ satisfy the equation $|j-i|=2 \ (mod \ ( g+n))$.  
Let $\mathcal{B}_3 = \{u_1, u_2, \dots, u_{g-1}, v_1, v_2, \dots, v_{g-1}, x_2, x_3, \dots, x_{g}\}$ where the curves are as shown in Figure \ref{Fig-29}.

\begin{lemma} \label{B_3} $\bigcup_{i=1} ^3 \mathcal{B}_i$ is a finite rigid set.\end{lemma}

\begin{figure}
	\begin{center} 
		\epsfxsize=1.60in \epsfbox{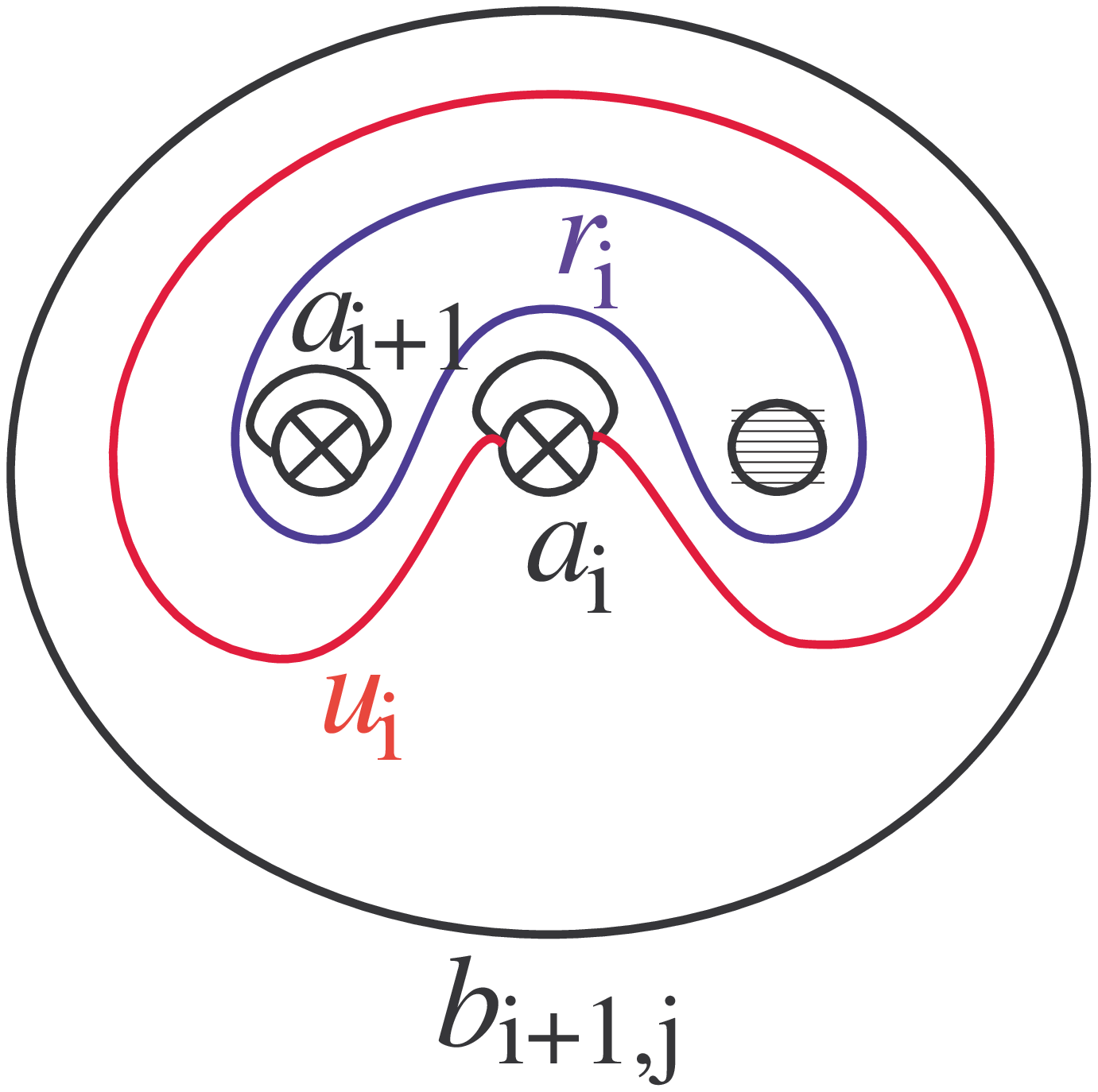} \hspace{0.1cm} 	
		\epsfxsize=1.60in \epsfbox{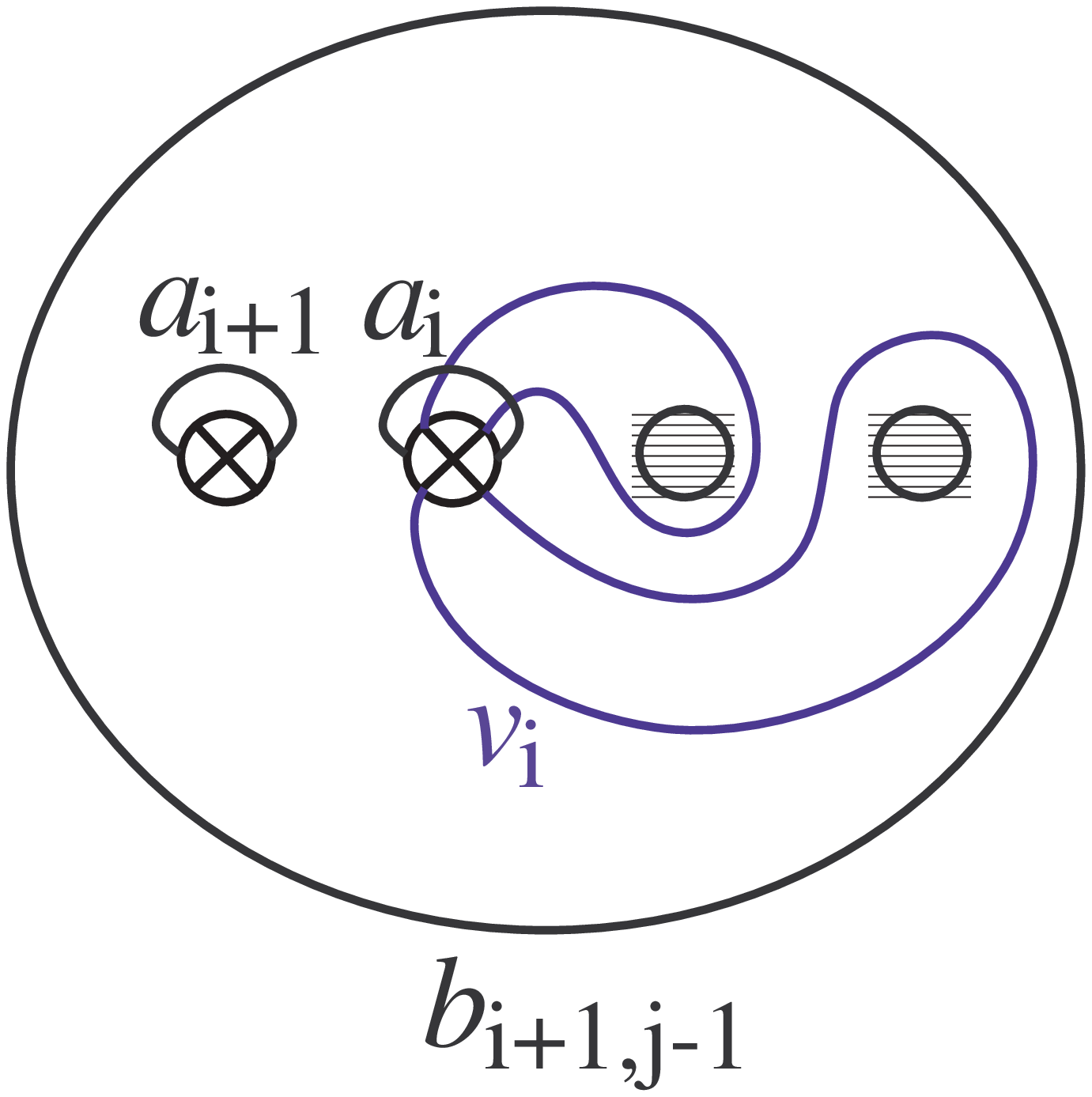} \hspace{0.1cm} 
		\epsfxsize=1.6in \epsfbox{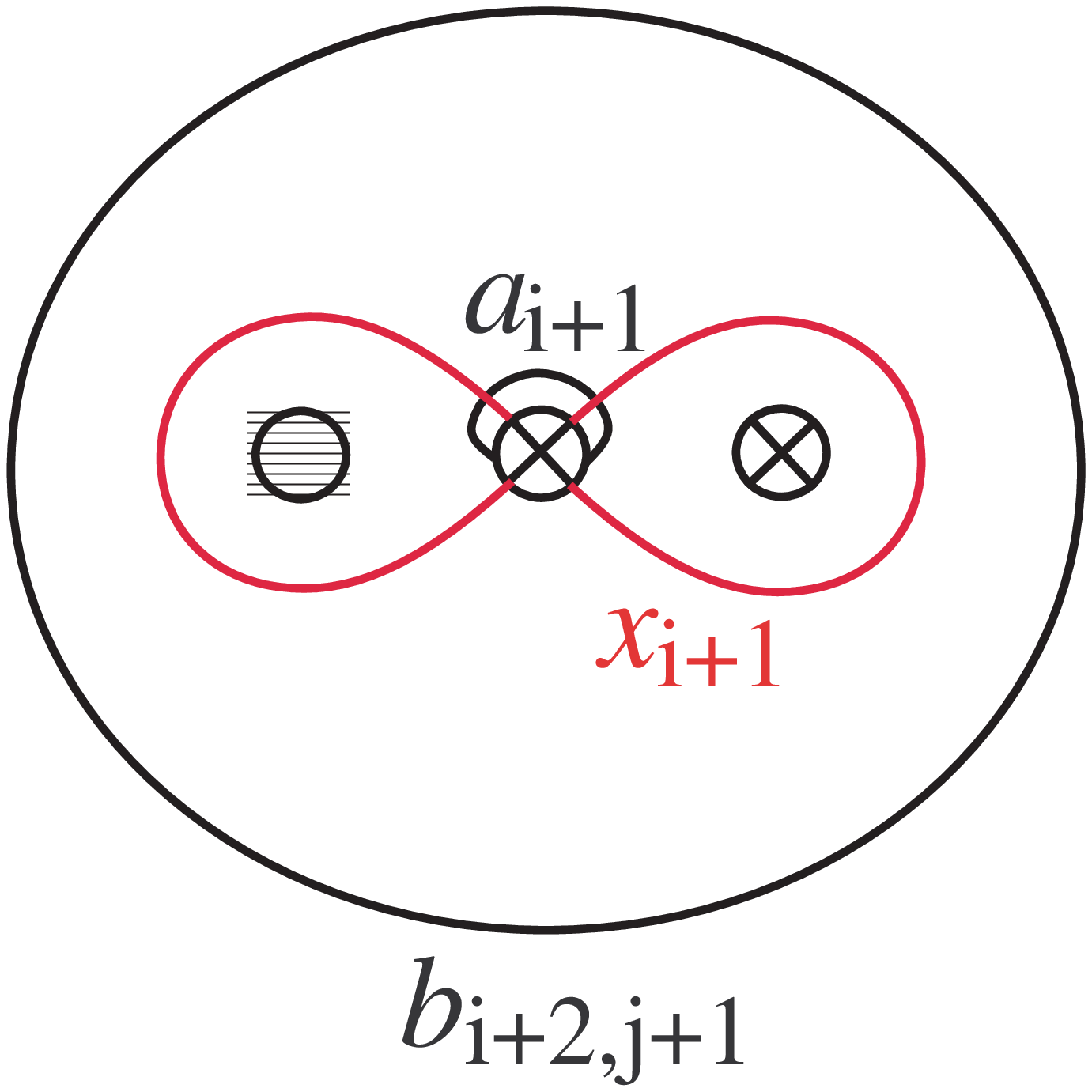}  \hspace{0.1cm}  
		
		(i)  \hspace{3.6cm}   (ii)  \hspace{3.6cm}   (iii) 
		
		\caption{Curves in $\mathcal{B}_3$}\label{Fig-29}
	\end{center}
\end{figure}

\begin{proof} We will give the proof when $n \geq 2$. The proofs for $n \leq 1$ cases are similar. Let $\lambda: \bigcup_{i=1} ^3 \mathcal{B}_i  \rightarrow \mathcal{C}(N)$ be a locally injective simplicial map. Since $\lambda$ restricts to a locally injective simplicial map on $\mathcal{B}_1 \cup \mathcal{B}_2$, by Lemma \ref{B_2} there exists a homeomorphism $h$ of $N$ such that $h([x]) = \lambda([x])$ for every $x$ in $\mathcal{B}_1 \cup \mathcal{B}_2$. 
	
Claim 1: $h([u_i]) = \lambda([u_i])$ for all $i =1, 2, \dots, g-1$.  
	
Proof of Claim 1: We first want to show that $h([u_1]) = \lambda([u_1])$. Using the curves in $\mathcal{B}_1 \cup \mathcal{B}_2$, we can complete the curves $r_1, a_1, a_2$ and 
$b_{2,g+n-1}$ to a top dimensional pants decomposition $P$ such that $u_1$ is disjoint from every element in $P \setminus \{a_1\}$. This implies that $\lambda([u_1])$ is disjoint from every element in $\lambda([P]) \setminus \{\lambda([a_1])\}$ since $\lambda$ is locally injective. Since $[u_1] \neq [a_1]$ and they are in the star of a vertex in $\mathcal{B}_1 \cup \mathcal{B}_2$, we see that $\lambda([u_{1}]) \neq \lambda([a_{1}])$. Since $h([x]) = \lambda([x])$ for every $x$ in $\mathcal{B}_1 \cup \mathcal{B}_2$, and $\lambda([u_{1}]) \neq \lambda([a_{1}])$ and $\lambda([u_1])$ is disjoint from every element in $\lambda([P]) \setminus \{\lambda([a_1])\}$, we see that 
$h([u_1]) = \lambda([u_1])$. Similarly we get $h([u_{i}]) = \lambda([u_{i}])$ for each $i$.
	
Claim 2: $h([v_i]) = \lambda([v_i])$ for all $i =1, 2, \dots, g-1$.
	
Proof of Claim 2: We will show that $h([v_1]) = \lambda([v_1])$. Cutting $N$ along the curve $b_{1,g+n-2}$ gives a nonorientable surface, $M_1$, of genus one with three boundary components, which contains the curve $a_{1,g+n-1}$. The curve $v_1$ is the unique curve up to isotopy disjoint from and nonisotopic to each of $u_1, a_{1,g+n-1}$ and $b_{1,g+n-2}$ and all the curves of $\mathcal{B}_1 \cup \mathcal{B}_2$ that are in $N \setminus M_1$. Since $h$ and $\lambda$ agree on the isotopy class of all these curves, we get $h([v_1]) = \lambda([v_1])$. Cutting $N$ along the curve $b_{2,g+n-1}$ gives a nonorientable surface, $M_2$, of genus two with two boundary components, which contains the curve $a_{2,g+n}$. The curve $v_2$ is the unique curve up to isotopy disjoint from and nonisotopic to each of $u_1, a_{2,g+n}, a_1$ and $b_{2,g+n-1}$ and all the curves of $\mathcal{B}_1 \cup \mathcal{B}_2$ that are in $N \setminus M_2$. Since $h$ and $\lambda$ agree on the isotopy class of all these curves, we see that $h([v_2]) = \lambda([v_2])$.
Cutting $N$ along the curve $b_{3,g+n}$ gives a nonorientable surface, $M_3$, of genus three with one boundary component, which contains the curve $a_{3,1}$. 
The curve $v_3$ is the unique curve up to isotopy disjoint from and nonisotopic to each of $u_{2}, a_{3,1}, a_1, a_2$ and $b_{3,g+n}$ and all the curves of $\mathcal{B}_1 \cup \mathcal{B}_2$ that are in $N \setminus M_3$. Since $h$ and $\lambda$ agree on the isotopy class of all these curves, we see that $h([v_3]) = \lambda([v_3])$. Similarly, we have $h([v_i]) = \lambda([v_i])$ for each $i$.
	
\begin{figure}[t]
	\begin{center} 
		\epsfxsize=1.58in    \epsfbox{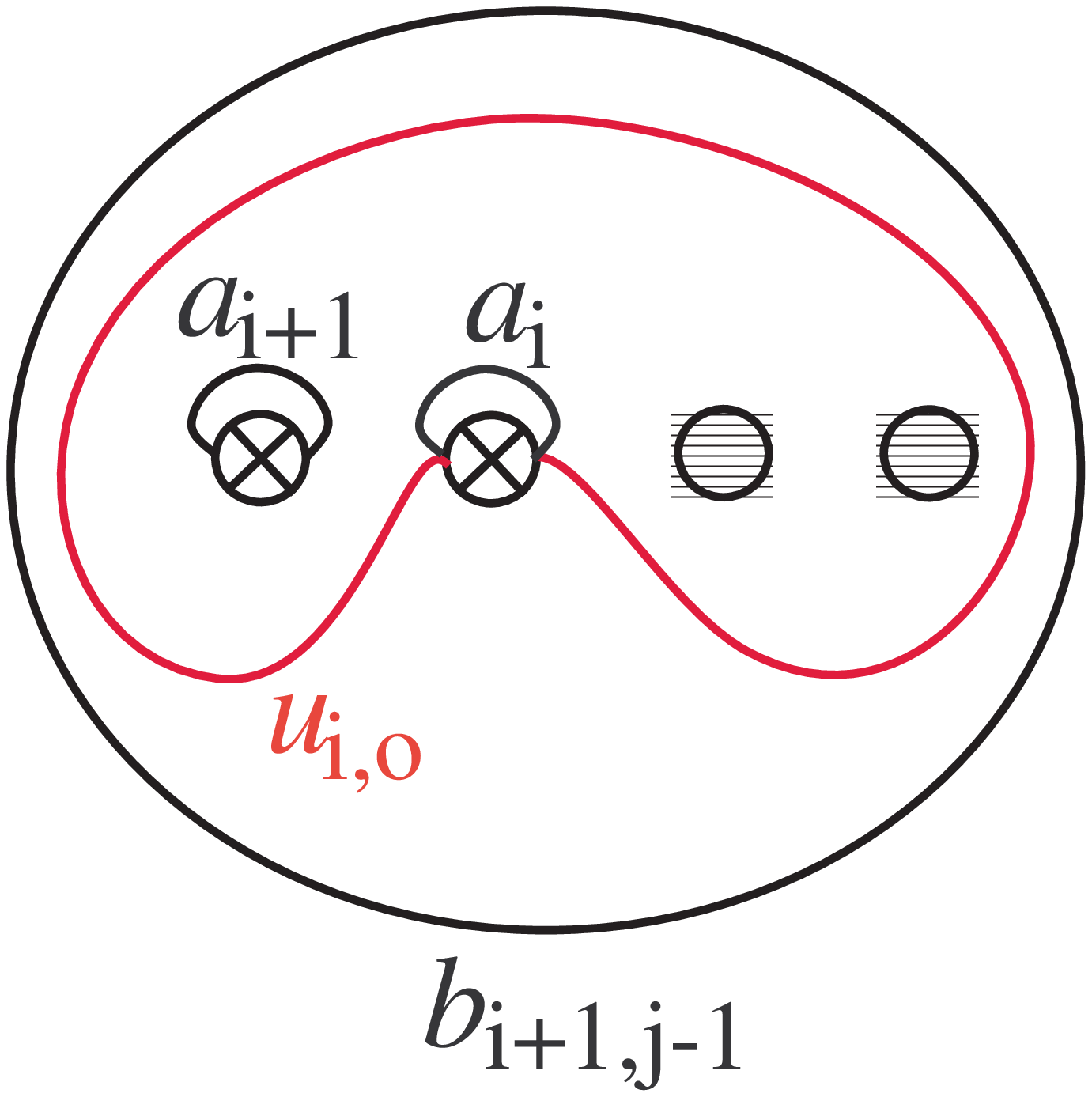}\hspace{0.1cm} 	
		\epsfxsize=1.58in \epsfbox{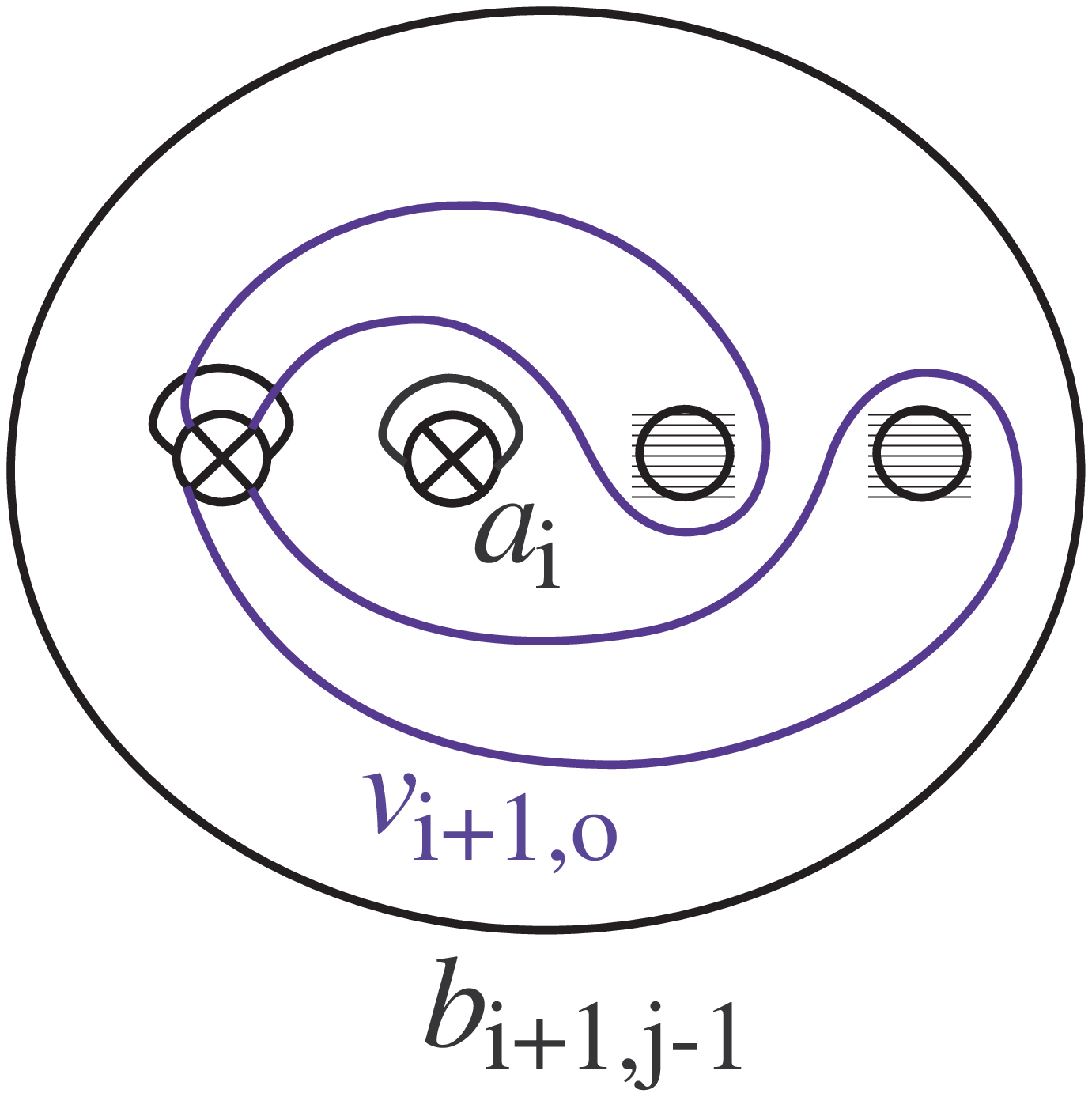}\hspace{0.1cm} 	\epsfxsize=1.58in \epsfbox{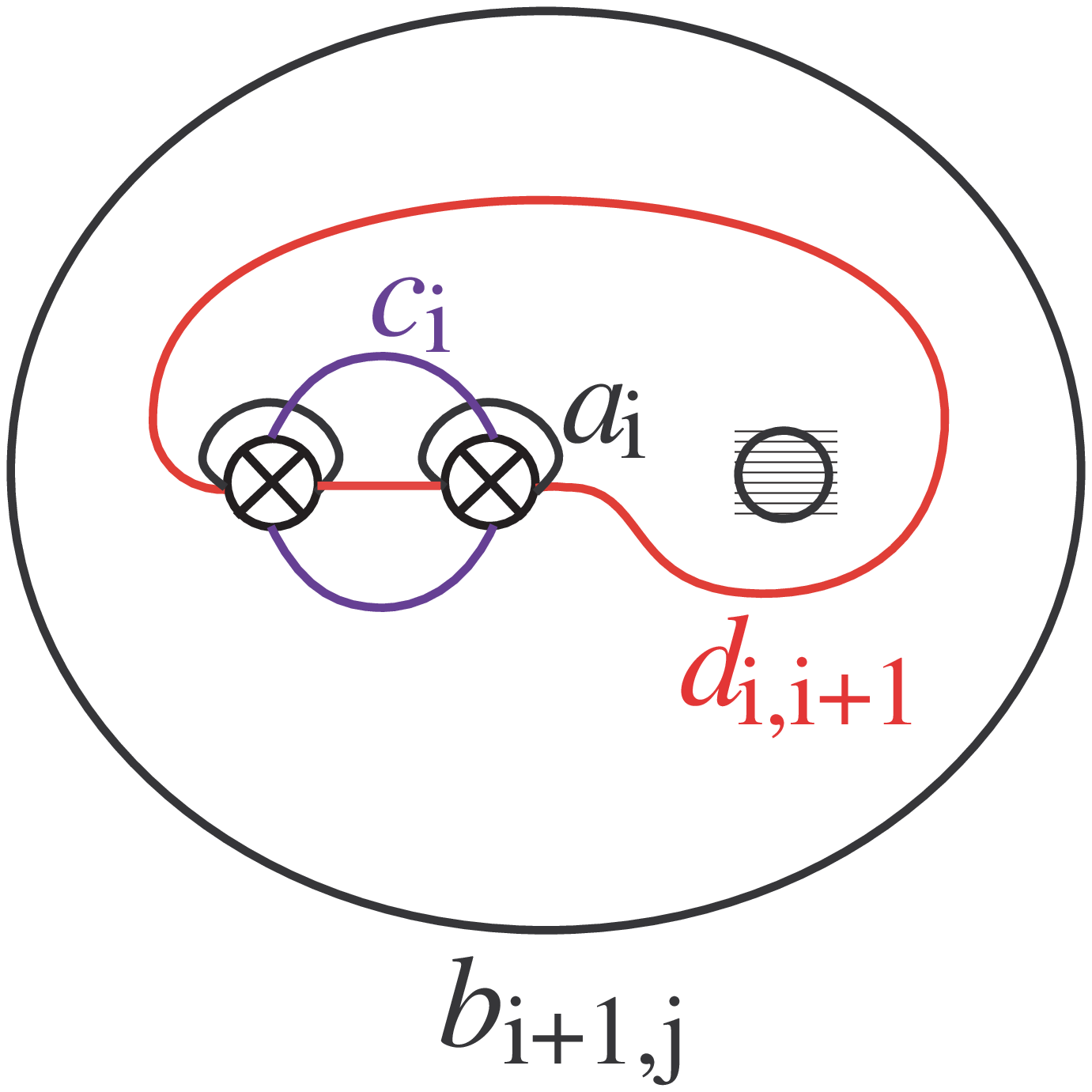}  \hspace{0.1cm}   
		
		(i)  \hspace{3.5cm}   (ii)  \hspace{3.5cm}   (iii)
		
		\caption{Curves in $\mathcal{B}_4$}\label{Fig-30}
	\end{center}
\end{figure} 
	
Claim 3: $h([x_{i}]) = \lambda([x_{i}])$ for all $i =2, 3, \dots, g$.

Proof of Claim 3:  (i) Suppose $g=2$. Cutting $N$ along the curve $b_{3,g+n}$ gives a nonorientable surface, $M_2$, of genus two with two boundary components, containing $a_1$ and $a_2$. The curve $x_2$ is the unique nontrivial simple closed curve up to isotopy disjoint from and nonisotopic to each of $a_1, a_{2,3}, a_{2,g+n}, b_{3,g+n}$ and all the curves of $\mathcal{B}_1 \cup \mathcal{B}_2$ that are in $N \setminus M_2$. Since $h$ and $\lambda$ agree on the isotopy class of all these curves, we see that $h([x_2]) = \lambda([x_2])$. 

(ii) Suppose $g=3$. Cutting $N$ along the curve $b_{3,g+n}$ gives a nonorientable surface, $M_3$, of genus three with one boundary component, containing $a_1, a_2$ and $a_3$. The curve $x_2$ is the unique nontrivial simple closed curve up to isotopy disjoint from and nonisotopic to each of $a_1, a_3,  a_{2,3}, a_{2,g+n}, b_{3,g+n}$ and all the curves of $\mathcal{B}_1 \cup \mathcal{B}_2$ that are in $N \setminus M_3$. Since $h$ and $\lambda$ agree on the isotopy class of all these curves, we see that $h([x_2]) = \lambda([x_2])$. When we cut $N$ along the curve $b_{4,1}$ we get a nonorientable surface, $M_2$, of genus two with two boundary components, containing $a_2$ and $a_3$. The curve $x_3$ is the unique nontrivial simple closed curve up to isotopy disjoint from and nonisotopic to each of $a_2, a_{3,4}, a_{3,1}, b_{4,1}$ and all the curves of $\mathcal{B}_1 \cup \mathcal{B}_2$ that are in $N \setminus M_2$. Since $h$ and $\lambda$ agree on the isotopy class of all these curves, we see that $h([x_3]) = \lambda([x_3])$.

Similarly, we have $h([x_n]) = \lambda([x_n])$ for all $i=2, 3,  \dots, g$ for all $g \geq 4$. Hence, we have $h([x]) = \lambda([x])$ for every $x$ in $\bigcup_{i=1} ^3 \mathcal{B}_i$.\end{proof}\\
 
Let $\mathcal{B}_4 = \{c_1, c_2, c_3, \dots, c_{g-1}, d_{1,2}, d_{2,3}, \dots, d_{g-1,g}, u_{1,o}, u_{2,o}, \dots, u_{g-1,o}, v_{2,o}, v_{3,o},$ $ \dots,$ $ v_{g,o}\}$ where the curves are as shown in Figure \ref{Fig-30} and Figure \ref{Fig-31}.

\begin{lemma} \label{B_4} $\bigcup_{i=1} ^4 \mathcal{B}_i$ is a finite rigid set.\end{lemma}
	
\begin{proof} We will give the proof when $n \geq 2$. The proofs $n \leq 1$ are similar. Let $\lambda: \bigcup_{i=1} ^4 \mathcal{B}_i \rightarrow \mathcal{C}(N)$ be a locally injective simplicial map. It is easy to see that $\lambda$ restricts to a locally injective simplicial map on $\bigcup_{i=1} ^3 \mathcal{B}_i$. By Lemma \ref{B_3}, there exists a homeomorphism $h$ of $N$ such that $h([x]) = \lambda([x])$ for every $x$ in $ \bigcup_{i=1} ^3 \mathcal{B}_i$. 
		
\begin{figure}[t]
	\begin{center} 
		\epsfxsize=2.49in \epsfbox{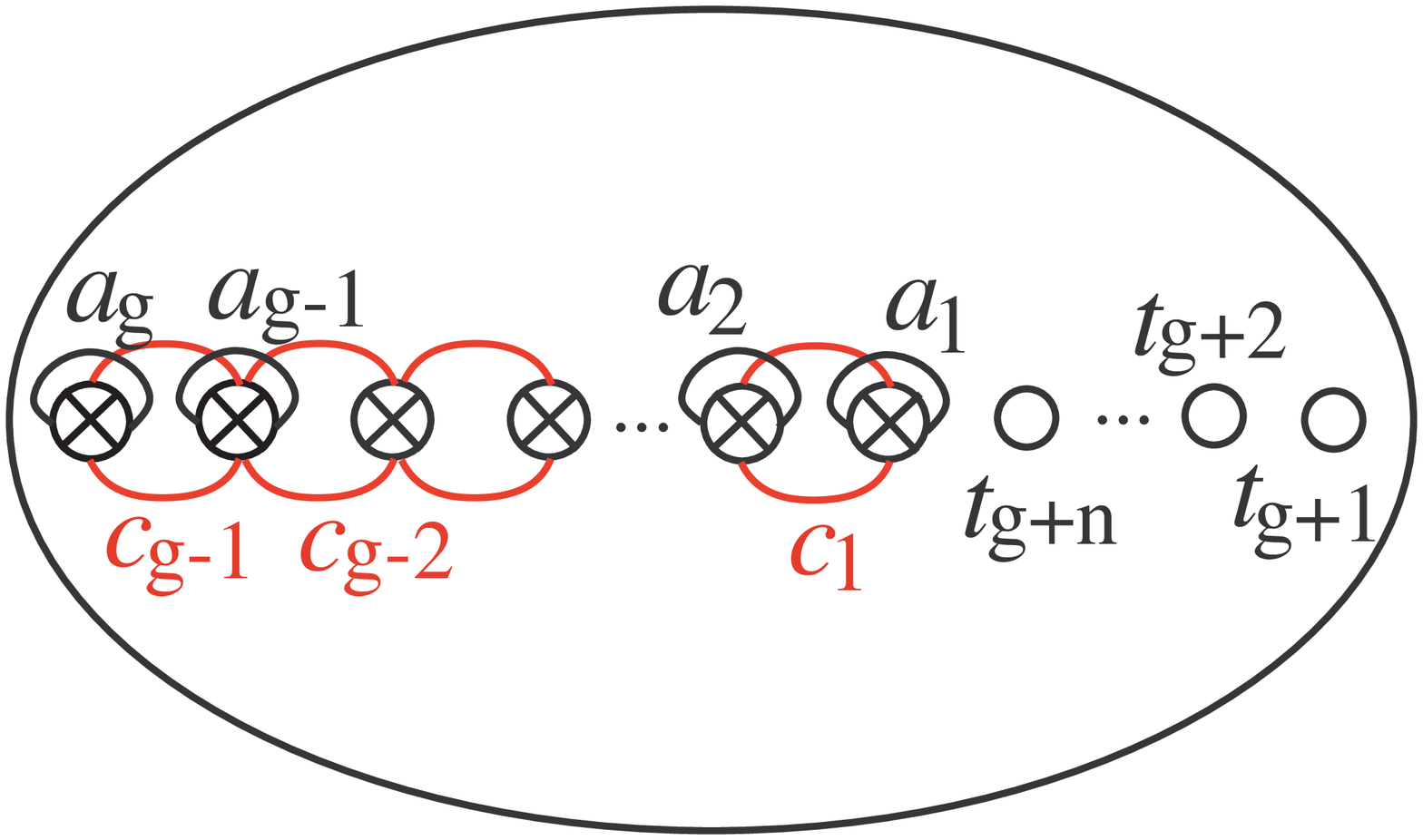} \hspace{0.09cm} 	\epsfxsize=2.49in \epsfbox{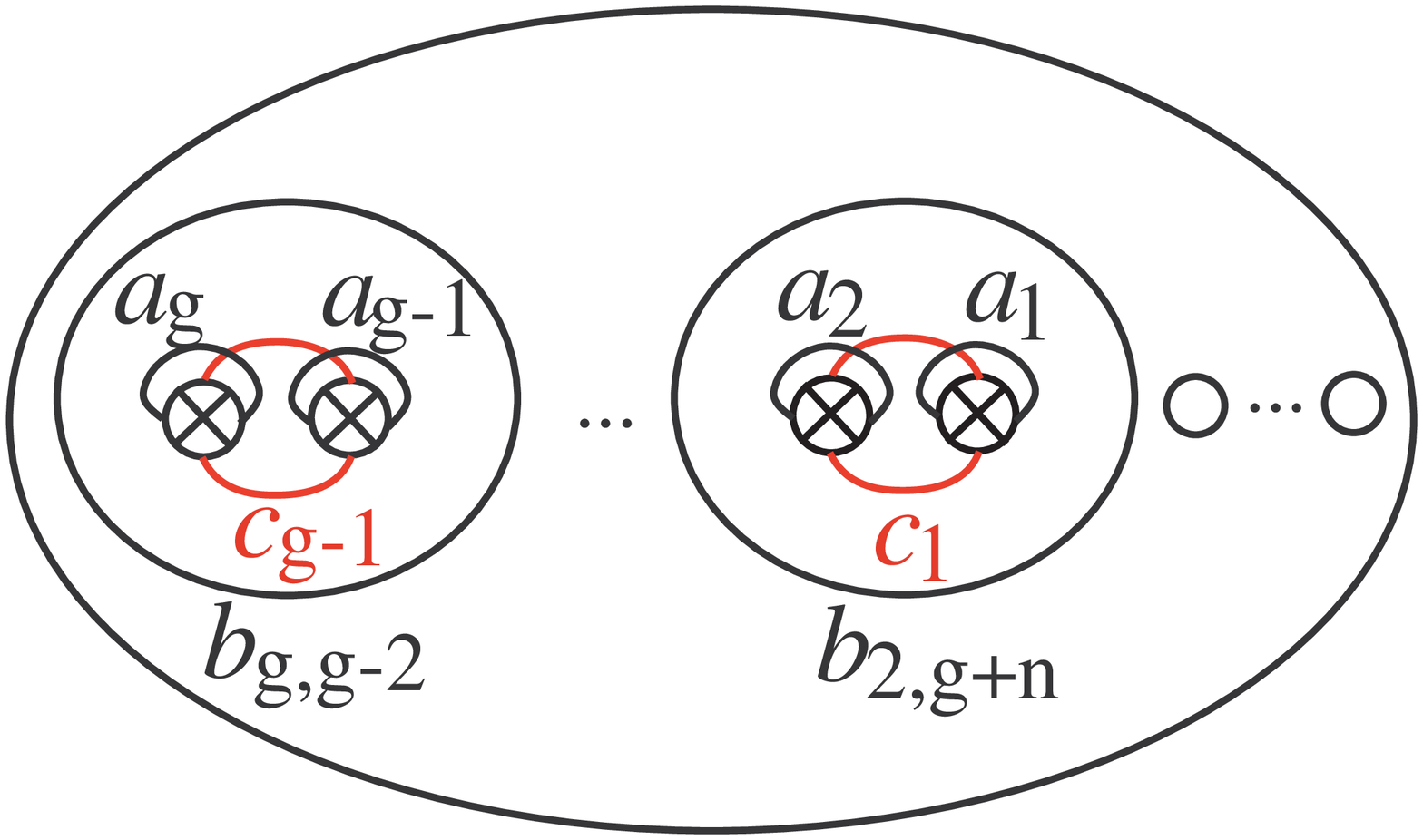} 
		
		\caption{Curves}\label{Fig-31}
	\end{center}
\end{figure} 
		
Claim 1: $h([u_{i,o}]) = \lambda([u_{i,o}])$ for all $i=1, 2, \dots, g-1$.  

Proof of Claim 1: When we cut $N$ along the curve $b_{2,g+n-2}$ we get a nonorientable surface, $M_2$, of genus two with three boundary components, containing $a_1$ and $a_2$. The curve $u_{1,o}$ is the unique nontrivial simple closed curve up to isotopy disjoint from $a_2, r_1, b_{2,g+n-2},$ $ b_{g+n,g+n-2}$, all the curves of $\mathcal{B}_1 \cup \mathcal{B}_2$ that are in $N \setminus M_2$, and also nonisotopic to $a_1$. Since $h$ and $\lambda$ agree on the isotopy class of all these curves, we have $h([u_{1,o}]) = \lambda([u_{1,o}])$. When we cut $N$ along the curve $b_{3,g+n-1}$ we get a nonorientable surface, $M_3$, of genus three with two boundary components, containing $a_1, a_2$ and $a_3$. The curve $u_{2,o}$ is the unique nontrivial simple closed curve up to isotopy disjoint from $a_1, a_3, r_2, b_{1,g+n-1}, b_{3,g+n-1}$, all the curves of $\mathcal{B}_1 \cup \mathcal{B}_2$ that are in $N \setminus M_3$, and also nonisotopic to $a_1$. Since $h$ and $\lambda$ agree on the isotopy class of all these curves, we have $h([u_{2,o}]) = \lambda([u_{2,o}])$. Similarly we get $h([u_{i,o}]) = \lambda([u_{i,o}])$ for all $i =3, 4, \dots, g-1$.

Claim 2: $h([v_{i,o}]) = \lambda([v_{i,o}])$ for all $i =2, 3, \dots, g$.

Proof of Claim 2: When we cut $N$ along the curve $b_{2,g+n-2}$ we get a nonorientable surface, $M_2$, of genus two with three boundary components, containing $a_1$ and $a_2$. The curve $v_{2,o}$ is the unique curve up to isotopy disjoint from $x_2, a_{2,g+n-1}, b_{2,g+n-2}, a_1, u_{2,o}$ and all the curves of $\mathcal{B}_1 \cup \mathcal{B}_2$ that are in $N \setminus M_2$. 
Since $h$ and $\lambda$ agree on the isotopy class of all these curves, we have $h([v_{2,o}]) = \lambda([v_{2,o}])$. 
Similarly, we have $h([v_{i,o}]) = \lambda([v_{i,o}])$ for all 
$i=3, 4, \dots, g$.
		
Claim 3: $h([d_{i,i+1}]) = \lambda([d_{i,i+1}])$ for all $i =1, 2, \dots, g-1$.
	
Proof of Claim 3: When we cut $N$ along the curve $b_{2,g+n-2}$ we get a nonorientable surface, $M_2$, of genus two with three boundary components, which contains $a_1$ and $a_2$. 
The curve $d_{1,2}$ is the unique nontrivial simple closed curve up to isotopy disjoint from and nonisotopic to each of $v_1, v_{2,o}, b_{2,g+n-1}, b_{2,g+n-2}$ and all the curves of $\mathcal{B}_1 \cup \mathcal{B}_2$ that are in $N \setminus M_2$. Since $h$ and $\lambda$ agree on the isotopy class of all these curves, we have $h([d_{1,2}]) = \lambda([d_{1,2}])$. When we cut $N$ along the curve $b_{3,g+n-1}$ we get a nonorientable surface, $M_3$, of genus three with two boundary components, which contains $a_1, a_2$ and $a_3$. The curve $d_{2,3}$ is the unique nontrivial simple closed curve up to isotopy disjoint from and nonisotopic to each of $a_1, v_2, v_{3,o}, b_{3,g+n}, b_{3,g+n-1}$ and all the curves of $\mathcal{B}_1 \cup \mathcal{B}_2$ that are in $N \setminus M_3$. Since $h$ and $\lambda$ agree on the isotopy class of all these curves, we have $h([d_{1,2}]) = \lambda([d_{1,2}])$. Similarly, we have $h([d_{i,i+1}]) = \lambda([d_{i,i+1}])$ for all $i =3, 4, \dots, g-1$.
	
Claim 4: $h([c_i]) = \lambda([c_i])$ for all $i =1, 2, \dots, g-1$. 
	
Proof of Claim 4: When we cut $N$ along the curve $b_{2,g+n}$ we get a nonorientable surface, $M_2$, of genus two with one boundary component, which contains $a_1$ and $a_2$. 
The curve $c_1$ is the unique nontrivial simple closed curve up to isotopy disjoint from and nonisotopic to each of $b_{2,g+n}, d_{1,2}$ and all the curves of $\mathcal{B}_1 \cup \mathcal{B}_2$ that are in $N \setminus M_2$. Since $h$ and $\lambda$ agree on the isotopy class of all these curves, we have $h([c_1]) = \lambda([c_1])$. When we cut $N$ along the curve $b_{1,3}$ we get a nonorientable surface, $K_2$, of genus two with one boundary component, which contains $a_2$ and $a_3$. 
The curve $c_2$ is the unique nontrivial simple closed curve up to isotopy disjoint from and nonisotopic to each of $b_{1,3}, d_{2,3}$ and all the curves of $\mathcal{B}_1 \cup \mathcal{B}_2$ that are in $N \setminus K_2$. Since $h$ and $\lambda$ agree on the isotopy class of all these curves, we have $h([c_2]) = \lambda([c_2])$. Similarly, $h([c_i]) = \lambda([c_i])$ for all $i =3, 4, \dots, g-1$. 
	
We proved that $h([x]) = \lambda([x])$ for every $x$ in $\bigcup_{i=1} ^4 \mathcal{B}_i$. Hence, $\bigcup_{i=1} ^4 \mathcal{B}_i$ is a finite rigid set.\end{proof}\\

\begin{figure}
	\begin{center} 
	\hspace{0.1cm} 	\epsfxsize=1.60in \epsfbox{rigid-code-fig95.eps}\hspace{0.2cm} 	
		\epsfxsize=1.60in \epsfbox{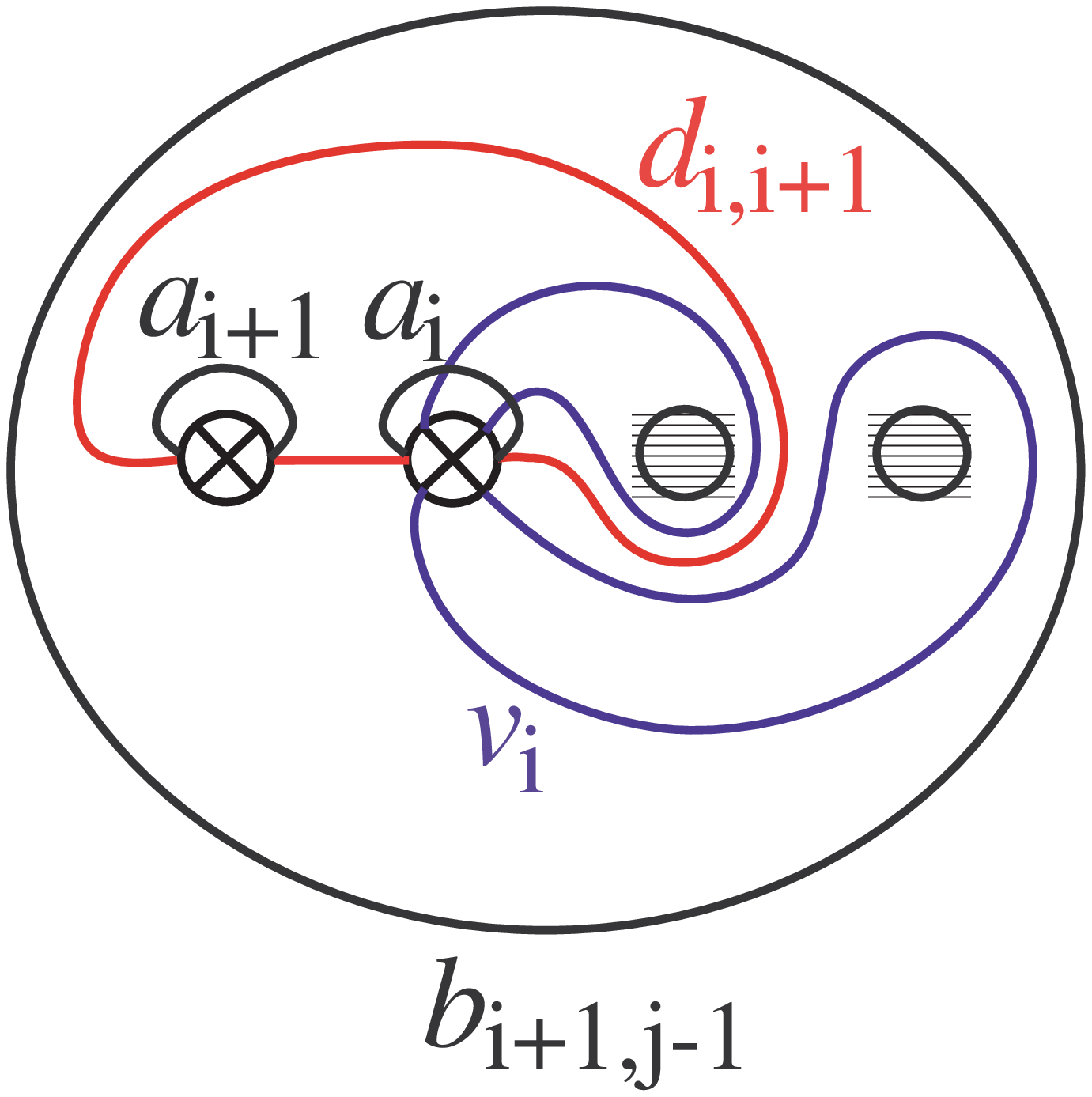}  \hspace{0.1cm} 	\epsfxsize=1.60in \epsfbox{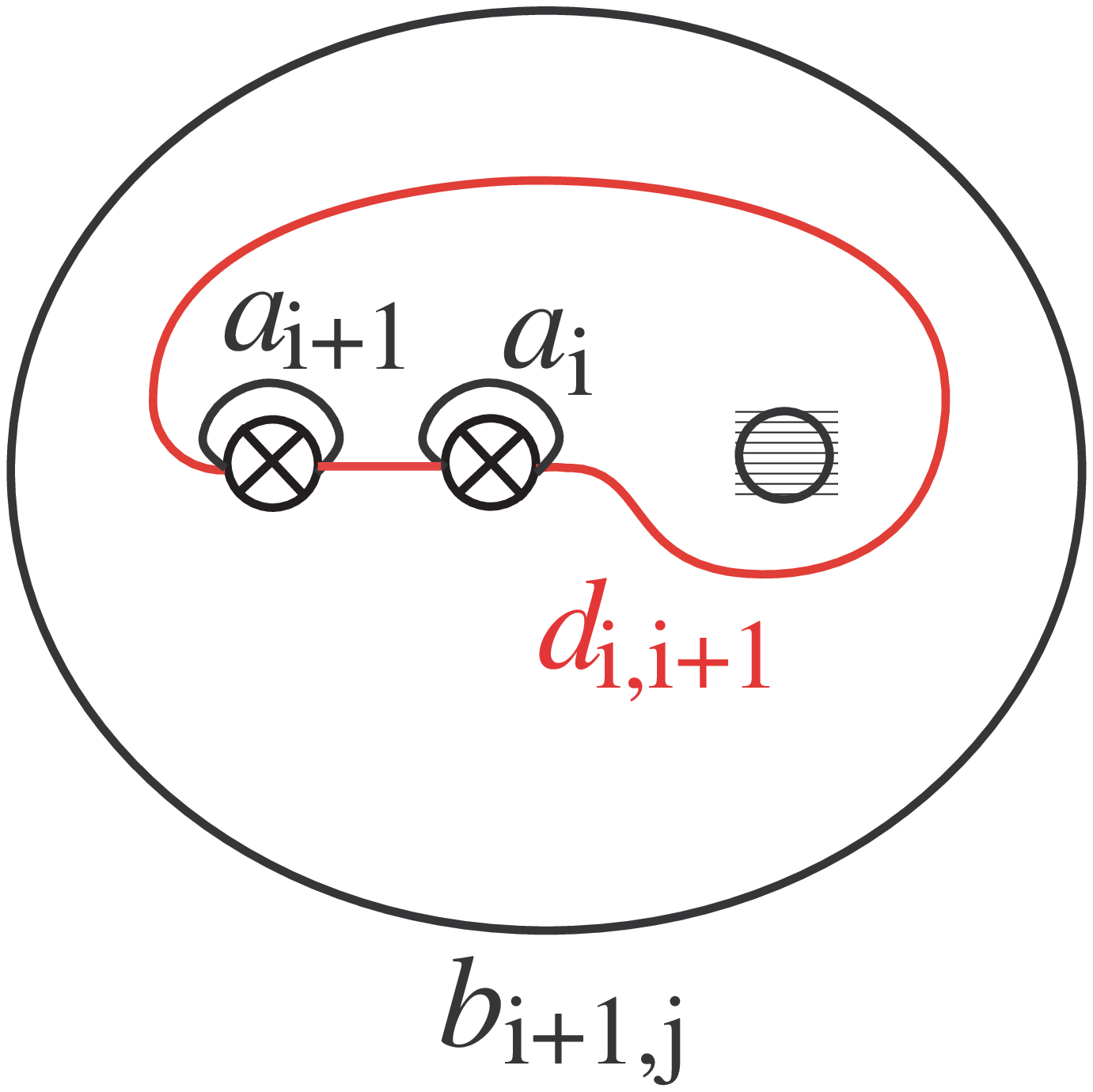}  \hspace{0.1cm}  
		
		(i)  \hspace{3.6cm}   (ii)  \hspace{3.6cm}   (iii) 
		
		\vspace{0.1cm}
		
		\epsfxsize=1.60in \epsfbox{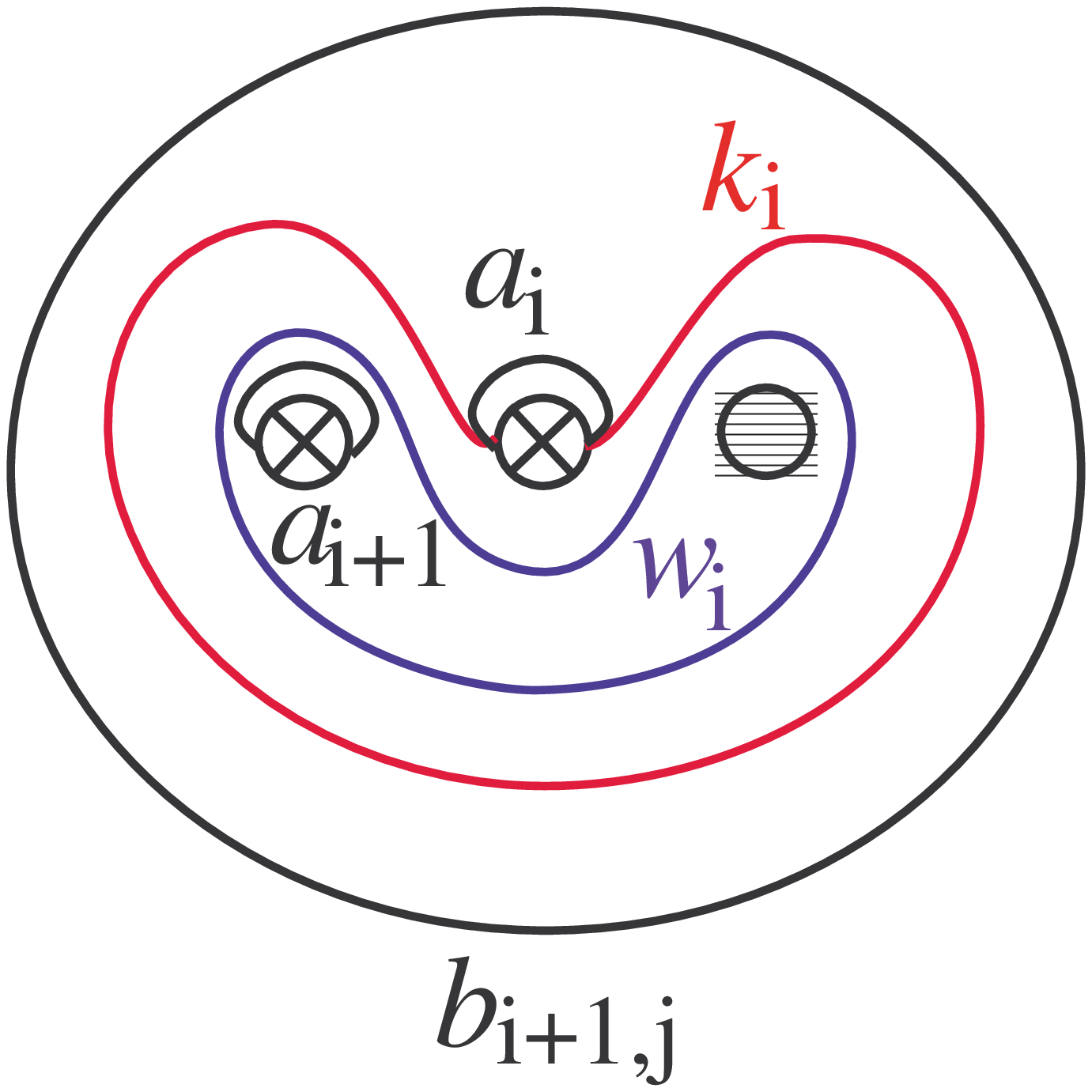} \hspace{0.1cm} 	
		\epsfxsize=1.60in \epsfbox{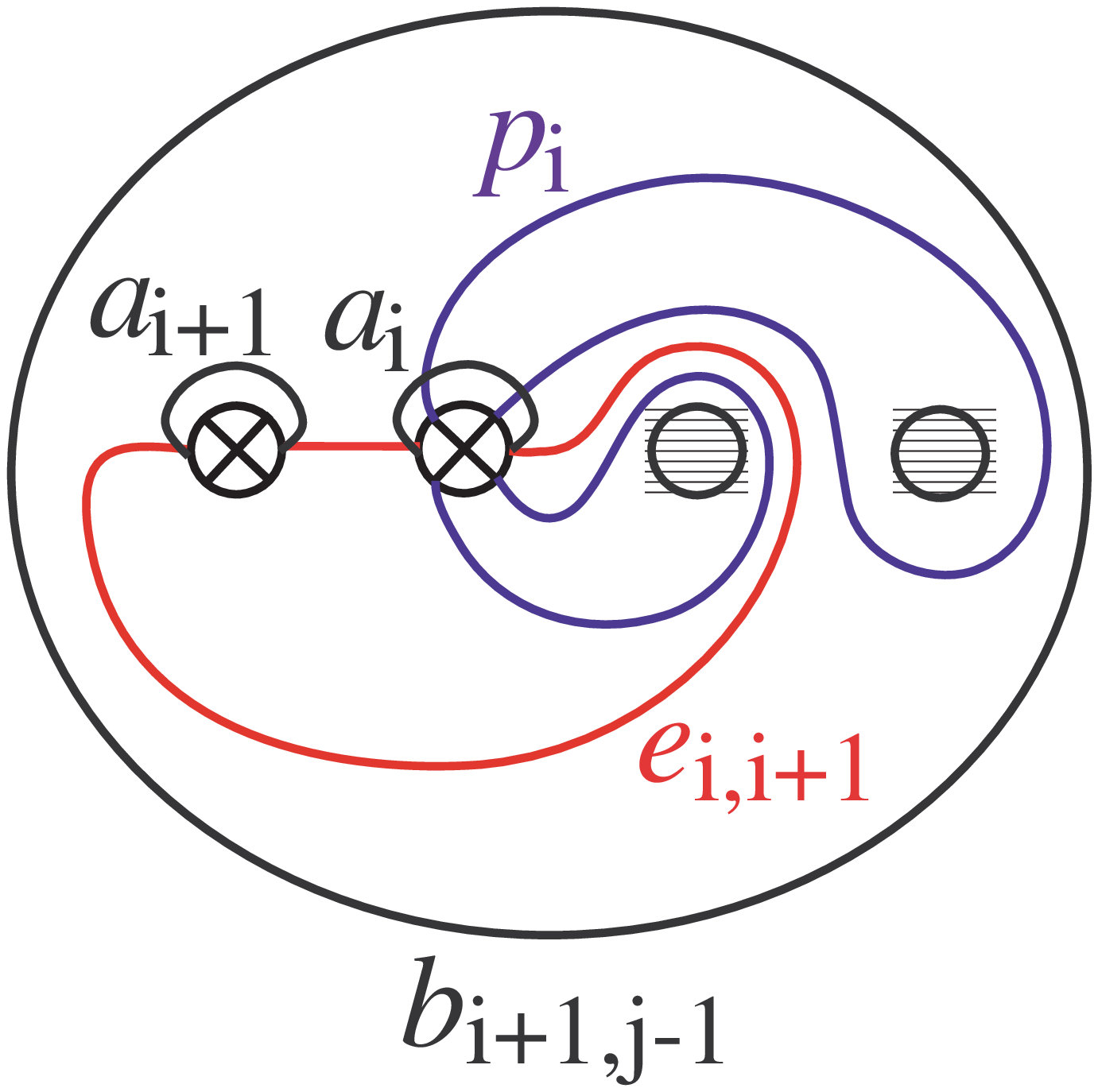}  \hspace{0.1cm} 
		\epsfxsize=1.60in \epsfbox{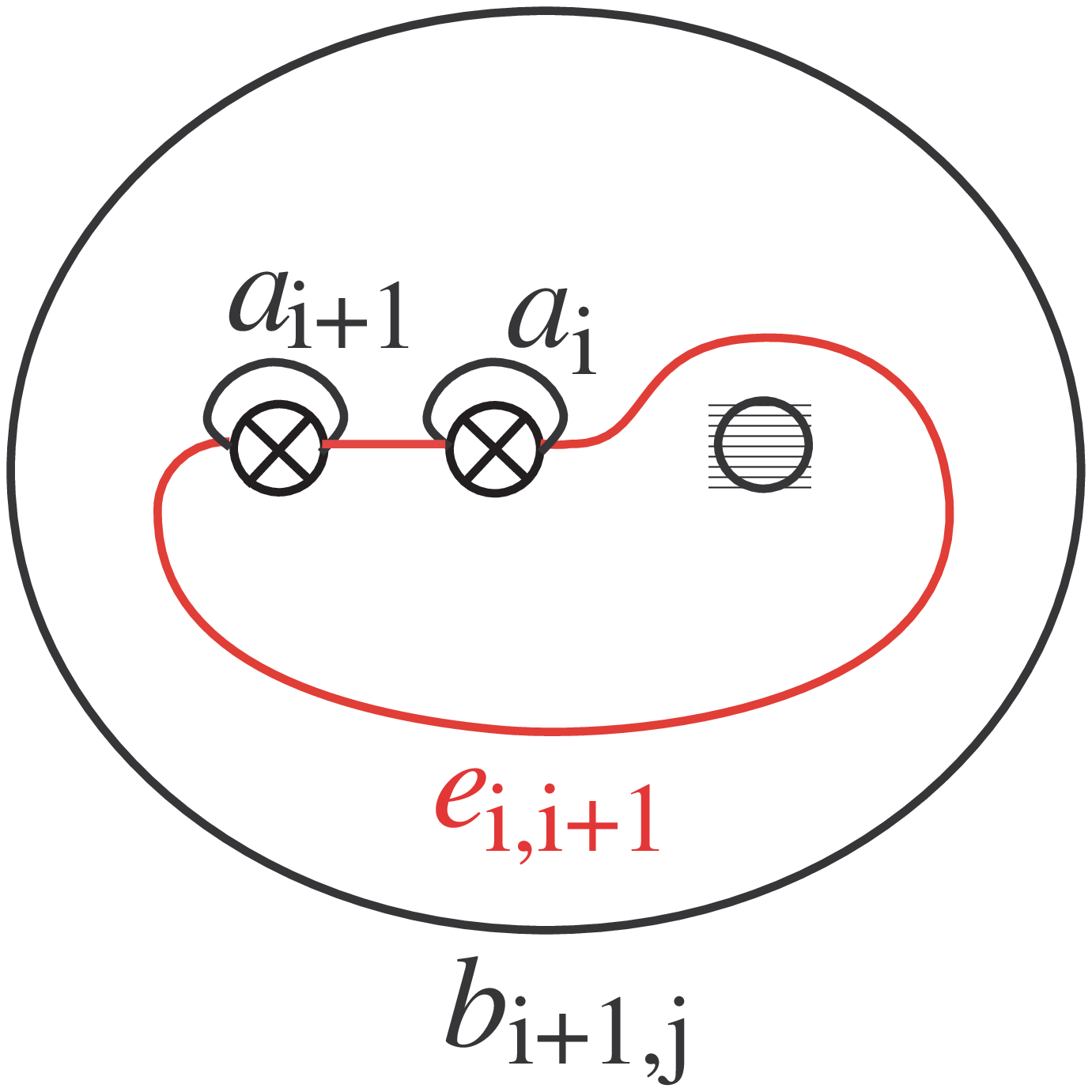} 
		
		(iv)  \hspace{3.6cm}   (v)  \hspace{3.6cm}   (vi) 
		
		\vspace{0.1cm}
		
		\epsfxsize=1.60in \epsfbox{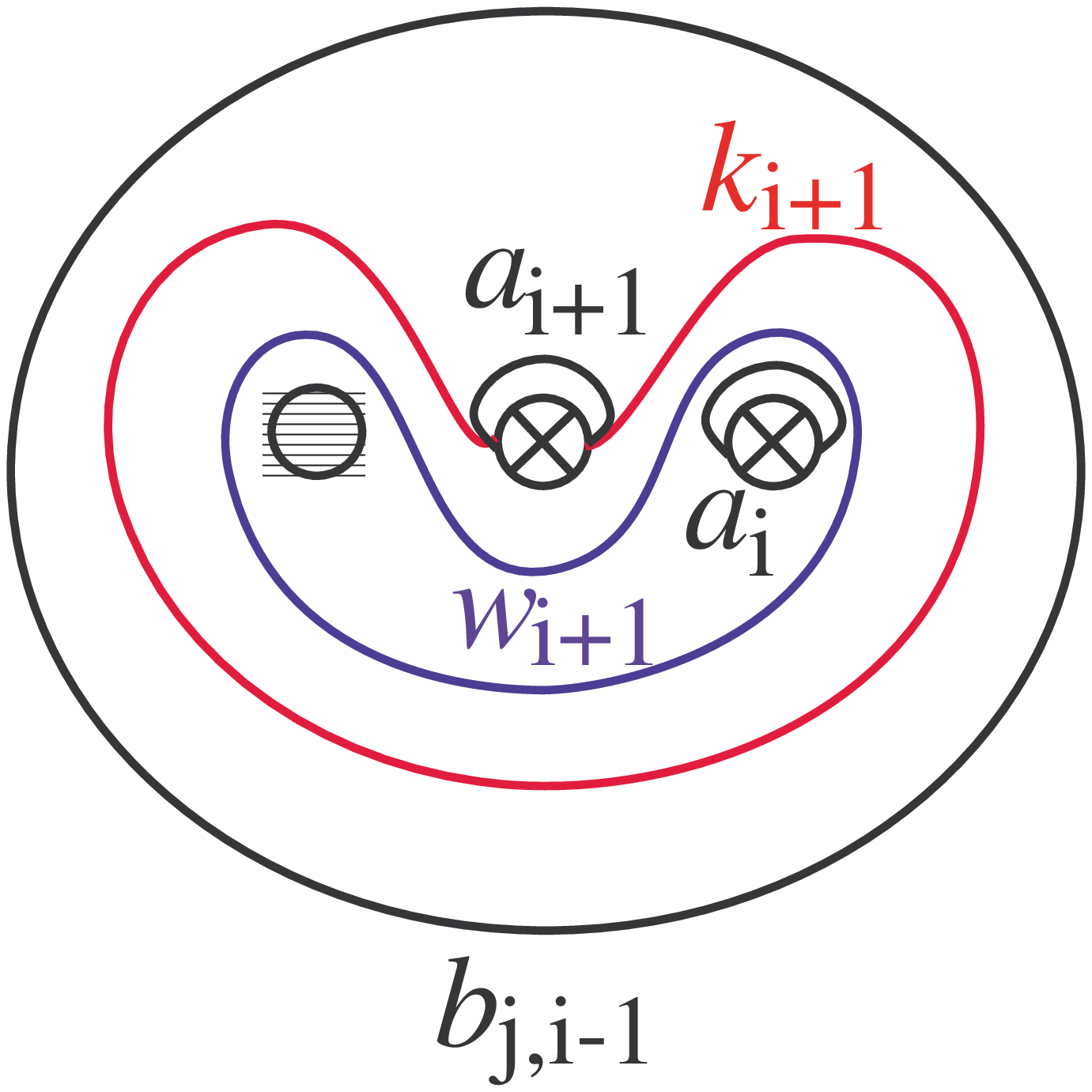} \hspace{0.1cm} 	
		\epsfxsize=1.60in \epsfbox{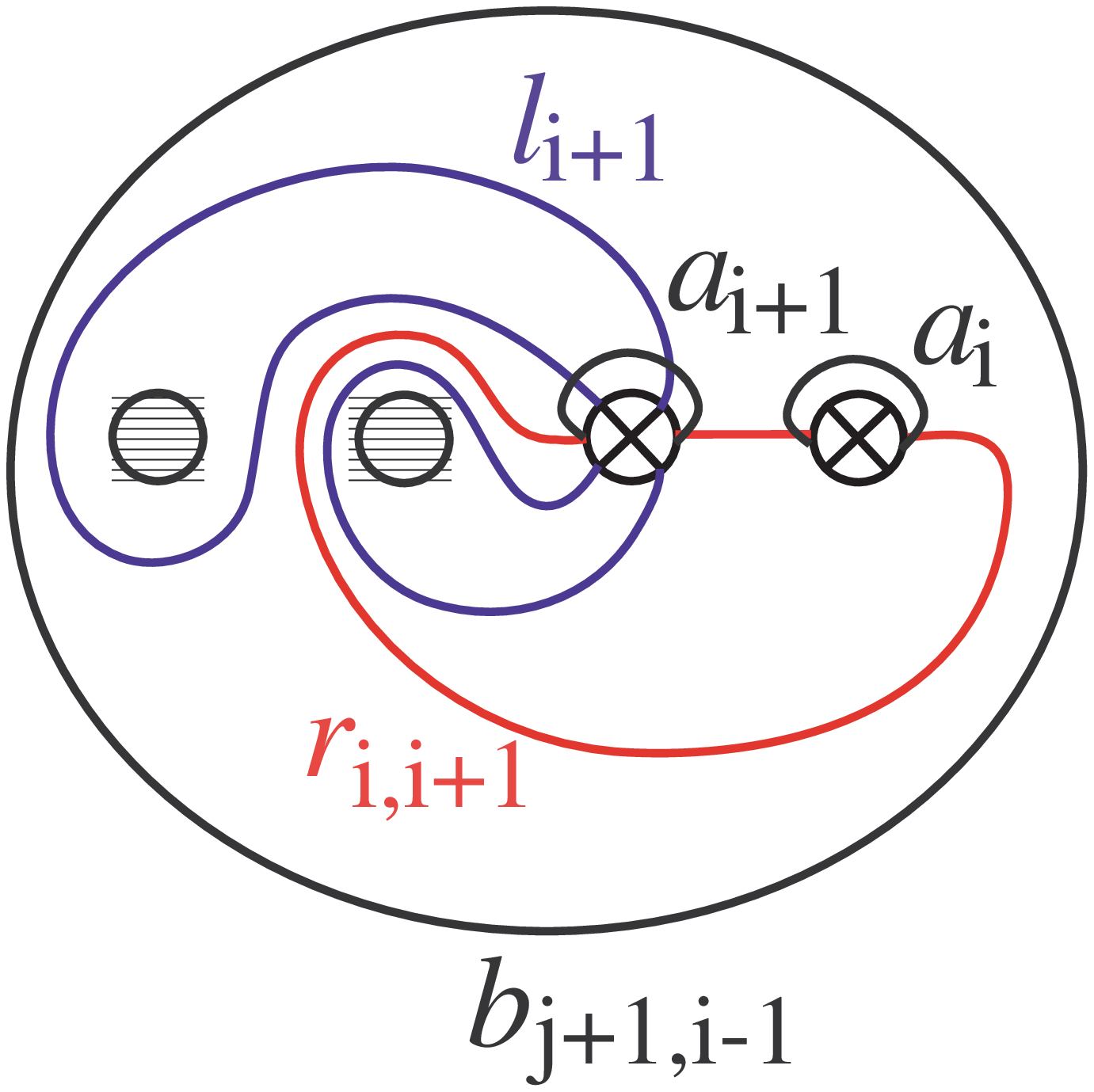}  \hspace{0.1cm} 
		\epsfxsize=1.60in \epsfbox{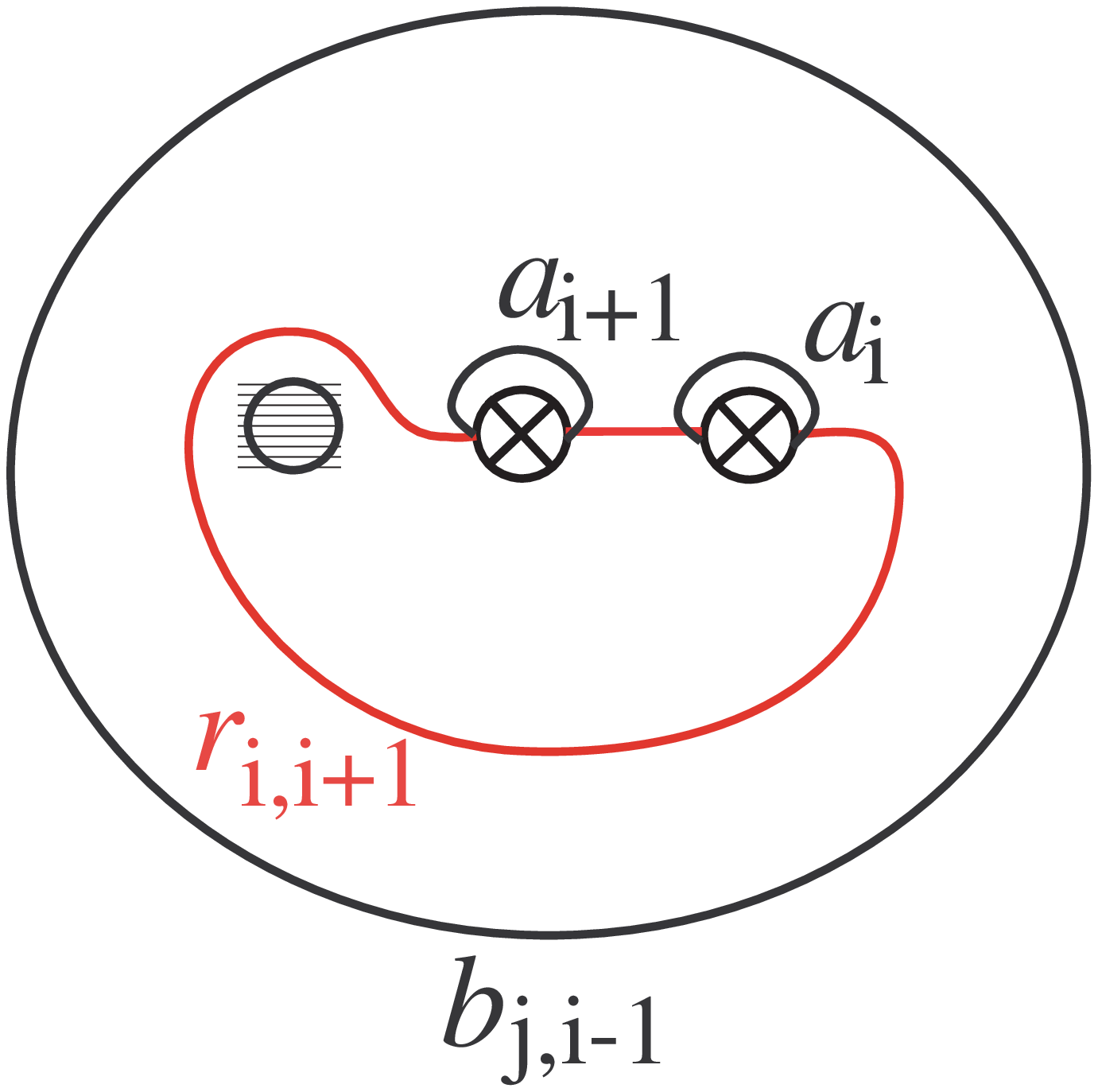} 
		
		(vii)  \hspace{3.6cm}   (viii)  \hspace{3.6cm}   (ix) 	
		
		\vspace{0.1cm}
		
		\epsfxsize=1.60in \epsfbox{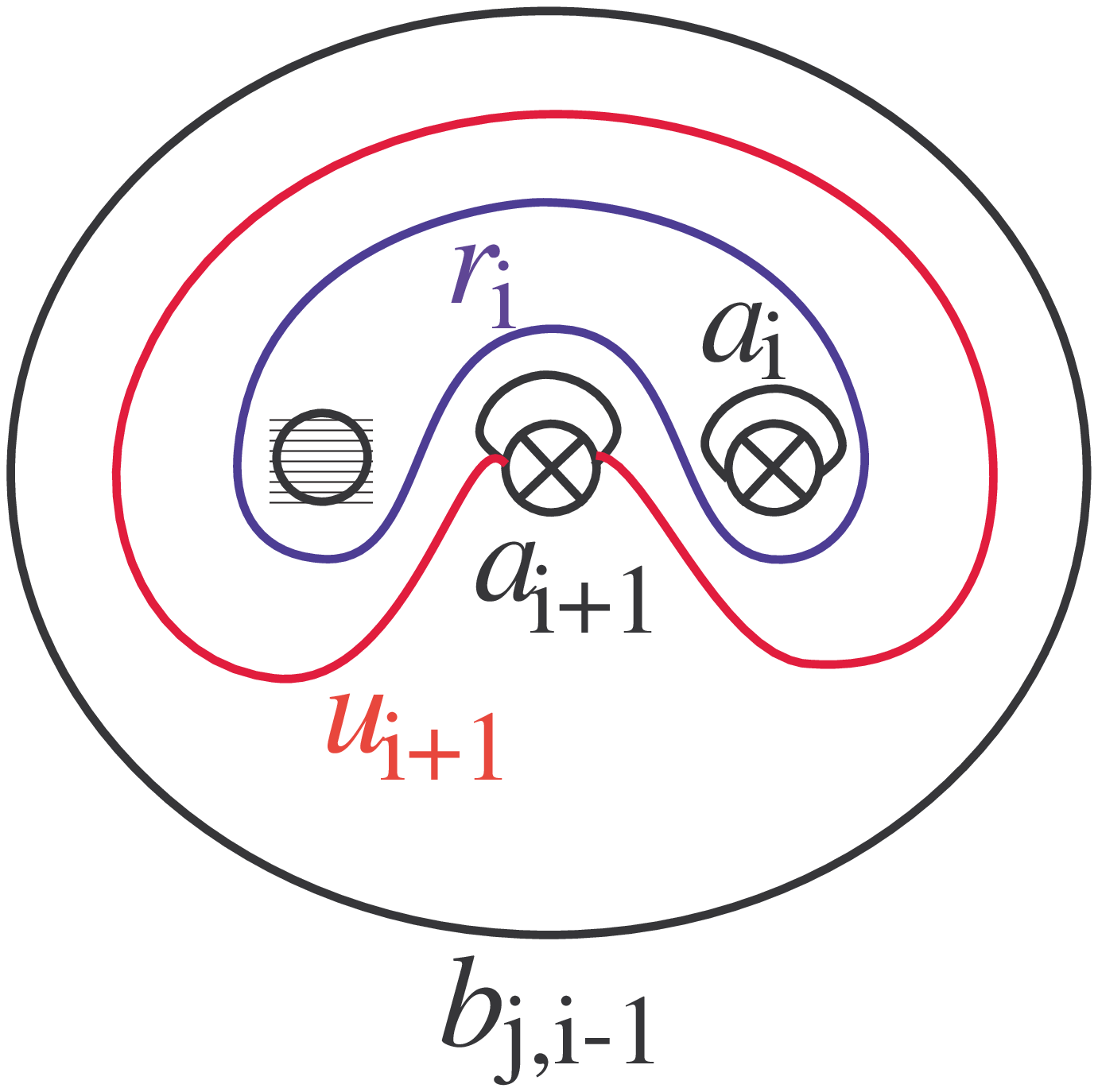} \hspace{0.1cm} 	
		\epsfxsize=1.60in \epsfbox{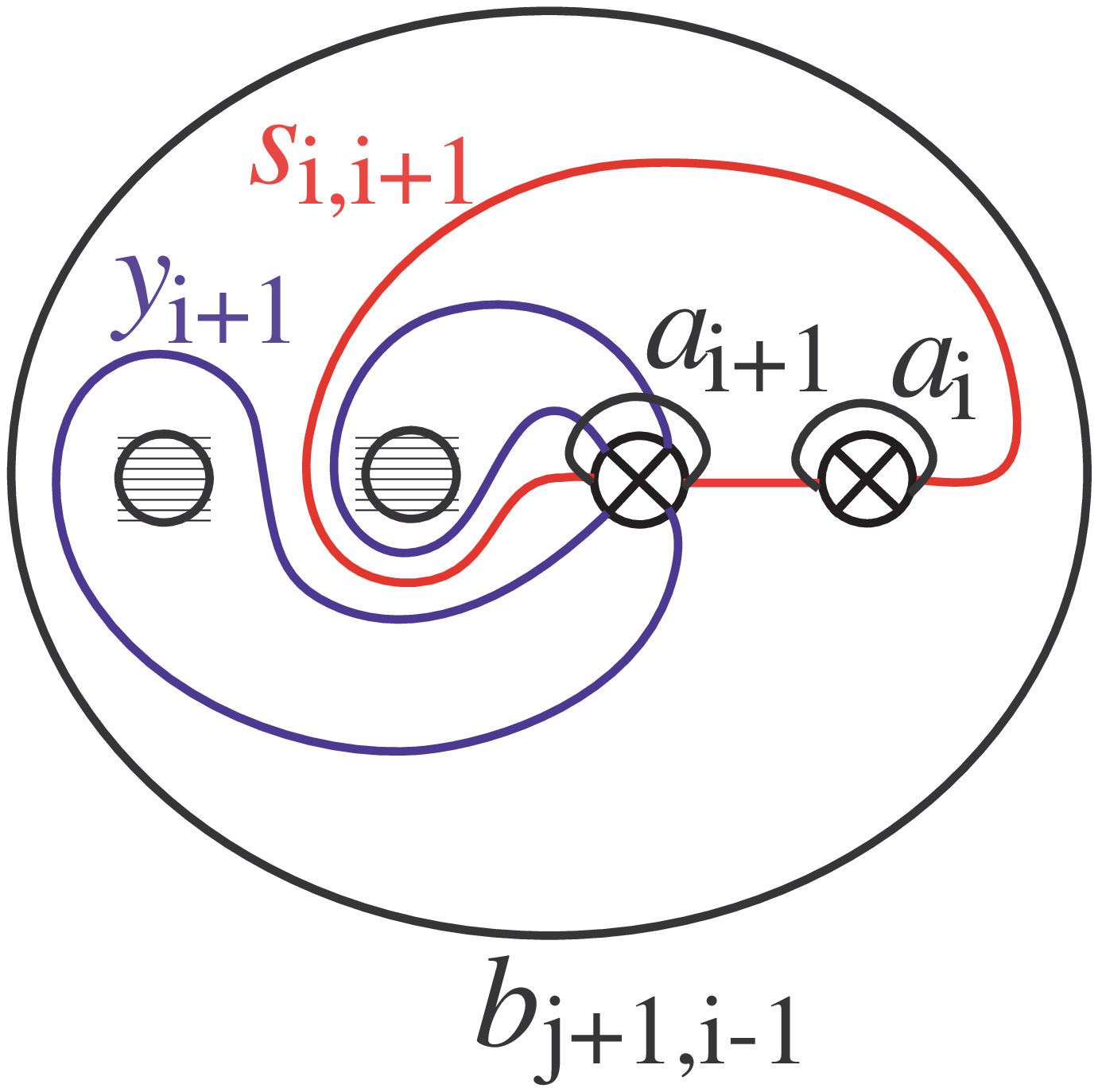}  \hspace{0.1cm} 
		\epsfxsize=1.60in \epsfbox{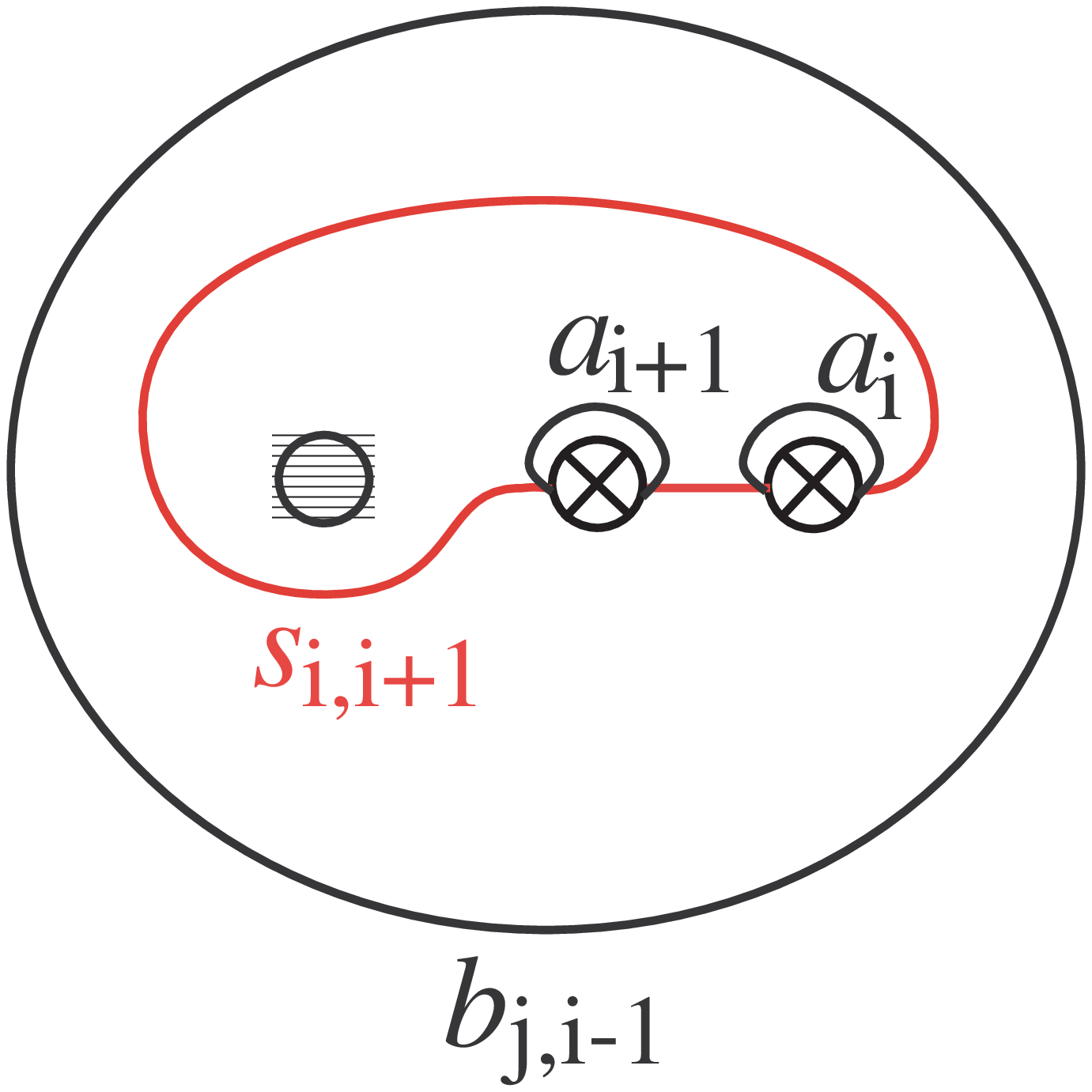} 
		
		(x)  \hspace{3.6cm}   (xi)  \hspace{3.6cm}   (xii) 	
		
		\caption{Curves in $\mathcal{B}_5$}\label{Fig-3-aa}
	\end{center}
\end{figure}
 
Let $\mathcal{B}_5 = \{k_1, k_2, \dots, k_{g-1}, p_1, p_2, \dots, p_{g-1}, e_{1,2}, e_{2,3}, \dots, e_{g-1,g},$ $l_2, l_3,$ $ \dots, l_g, $ $r_{1,2}, r_{2,3}, \dots,$ $r_{g-1,g},$ $y_2, y_3, \dots, y_g, s_{1,2}, s_{2,3}, \dots, s_{g-1,g} \}$ where the curves are as shown in Figure \ref{Fig-3-aa}. 

\begin{lemma} \label{B_5} $\bigcup_{i=1} ^5 \mathcal{B}_i$ is a finite rigid set.\end{lemma}
 
\begin{proof} We will give the proof when $n \geq 2$. The proofs for the other cases when $n \leq 1$ are similar. Let $\lambda: \bigcup_{i=1} ^5 \mathcal{B}_i \rightarrow \mathcal{C}(N)$ be a locally injective simplicial map. The map $\lambda$ restricts to a locally injective simplicial map on $\bigcup_{i=1} ^4$. By Lemma \ref{B_4}, there exists a homeomorphism $h: N \rightarrow N$ such that $h([x]) = \lambda([x])$ for every $x$ in $\bigcup_{i=1} ^4 \mathcal{B}_i$. 

Claim 1: $h([k_i]) = \lambda([k_i])$ for all $i =1, 2, \dots, g-1$.  

Proof of Claim 1: To see $h([k_1]) = \lambda([k_1])$, we cut $N$ along the curve $b_{2,g+n-1}$. This gives us a nonorientable surface, $M_1$, of genus two with two boundary components, containing $a_1$ and $a_2$. The curve $k_1$ is the unique nontrivial simple closed curve up to isotopy disjoint from and nonisotopic to each of $a_2, b_{2,g+n-1}, w_1$, all the curves of $\mathcal{B}_1 \cup \mathcal{B}_2$ that are in $N \setminus M_1$, and also nonisotopic to $a_1$. 
Since $h$ and $\lambda$ agree on the isotopy class of all these curves, we have $h([k_1]) = \lambda([k_1])$. To see $h([k_2]) = \lambda([k_2])$, we cut $N$ along the curve $b_{3,g+n}$. This gives us a nonorientable surface, $M_2$, of genus three with one boundary component, containing $a_1, a_2$ and $a_3$. The curve $k_2$ is the unique nontrivial simple closed curve up to isotopy disjoint from and nonisotopic to each of $a_1, a_3, b_{3,g+n}, w_2$, all the curves of $\mathcal{B}_1 \cup \mathcal{B}_2$ that are in $N \setminus M_2$ and also nonisotopic to $a_2$.  
Since $h$ and $\lambda$ agree on the isotopy class of all these curves, we have $h([k_2]) = \lambda([k_2])$. Similarly, we have $h([k_i]) = \lambda([k_i])$ for each $i$.  

Claim 2: $h([p_i]) = \lambda([p_i])$ for all $i =1, 2, \dots, g-1$.

Proof of Claim 2: To see $h([p_1]) = \lambda([p_1])$, we cut $N$ along the curve $b_{2,g+n-2}$. This gives us a nonorientable surface, $M_3$, of genus two with three boundary components, containing $a_1$ and $a_2$. The curve $p_1$ is the unique nontrivial simple closed curve up to isotopy disjoint from and nonisotopic to each of $k_1, a_{1,g+n-1}$ and $b_{1,g+n-2}$ and all the curves of $\mathcal{B}_1 \cup \mathcal{B}_2$ that are in $N \setminus M_3$. Since $h$ and $\lambda$ agree on the isotopy class of all these curves, we have $h([p_1]) = \lambda([p_1])$. To see $h([p_2]) = \lambda([p_2])$, we cut $N$ along the curve $b_{2,g+n-1}$. This gives us a nonorientable surface, $K_2$, of genus two with two boundary components, containing $a_{2,g+n}$. The curve $p_2$ is the unique nontrivial simple closed curve up to isotopy disjoint from and nonisotopic to each of $k_2, a_{2,g+n}, a_1, b_{2,g+n-1}$ and all the curves of $\mathcal{B}_1 \cup \mathcal{B}_2$ that are in $N \setminus K_2$. 
Since $h$ and $\lambda$ agree on the isotopy class of all these curves, we have $h([p_2]) = \lambda([p_2])$. To see $h([p_3]) = \lambda([p_3])$, we cut $N$ along the curve $b_{3,g+n}$. This gives us a nonorientable surface, $M_4$, of genus three with one boundary component, containing $a_{3,1}$. 
The curve $p_3$ is the unique curve up to isotopy disjoint from and nonisotopic to each of $k_3, a_{3,1}, a_1, a_2$ and $b_{3,g+n}$ and all the curves of $\mathcal{B}_1 \cup \mathcal{B}_2$ that are in $N \setminus M_4$. Since $h$ and $\lambda$ agree on the isotopy class of all these curves, we have $h([p_3]) = \lambda([p_3])$. Similarly, we have $h([p_i]) = \lambda([p_i])$ for each $i$.

Claim 3: $h([e_{i,i+1}]) = \lambda([e_{i,i+1}])$ for all $i =1, 2, \dots, g-1$.

Proof of Claim 3: To see $h([e_{1,2}]) = \lambda(e_{1,2})$, we cut $N$ along the curve $b_{2,g+n-1}$. This gives us a nonorientable surface, $M_5$, of genus two with two boundary
components containing $c_1$. The curve $e_{1,2}$ is the unique nontrivial simple closed curve up to isotopy disjoint from and nonisotopic to each of $c_1, b_{2,g+n-1}$ and $p_1$ and all the curves of $\mathcal{B}_1 \cup \mathcal{B}_2$ that are in $N \setminus M_5$. Since $h$ and $\lambda$ agree on the isotopy class of all these curves, we have $h([e_{1,2}]) = \lambda([e_{1,2}])$. To see $h([e_{2,3}]) = \lambda([e_{2,3}])$, 
we cut $N$ along the curve $b_{3,g+n}$. This gives us a nonorientable surface, $M_6$, of genus three with one boundary component, which contains $c_2$. The curve $e_{2,3}$ is the unique nontrivial curve up to isotopy disjoint from and nonisotopic to each of $a_1, c_2, b_{3,g+n}$ and $p_2$ and all the curves of $\mathcal{B}_1 \cup \mathcal{B}_2$ that are in $N \setminus M_6$. Since $h$ and $\lambda$ agree on the isotopy class of all these curves, we have $h([e_{2,3}]) = \lambda([e_{2,3}])$. Similarly, we have $h([e_{i,i+1}]) = \lambda([e_{i,i+1}])$ for each $i$.

Claim 4: $h([l_i]) = \lambda([l_i])$, $h([y_i]) = \lambda([y_i])$ 
for all $i =2, 3, \dots, g$, and $h([r_{i,i+1}]) =\lambda([r_{i,i+1}])$, $h([s_{i,i+1}]) = \lambda([s_{i,i+1}])$ for all $i =1, 2, \dots, g-1$. 

Proof of Claim 4: The proofs are very similar the proofs of Claim 2 and Claim 3.

We proved that $h([x]) = \lambda([x])$ for every $x$ in $\bigcup_{i=1} ^5 \mathcal{B}_i$. Hence, $\bigcup_{i=1} ^5 \mathcal{B}_i$ is a finite rigid set.\end{proof}\\

{\bf Remark:} From now on, to prove our main theorem we will follow the proofs given by the author in \cite{Ir10}. We note that our set $\bigcup_{i=1} ^5 \mathcal{B}_i$ has more curves (see for example $x_i, u_{i,o}, v_{i+1,o}$) than the curves considered in Lemma 3.8 in \cite{Ir10}. So, even though there are similarities, the sequence we will construct whose elements are rigid sets in the proof of Theorem 1.1 will be different from the sequence we constructed whose elements were superrigid sets in \cite{Ir10}. 

\begin{lemma} \label{curves} $\bigcup_{i=1} ^5 \mathcal{B}_i$ can be completed to a rigid set $\mathcal{B}$ such that $\mathcal{B}$ has nontrivial curves of every topological type on $N$.\end{lemma}

\begin{proof} By Lemma \ref{B_5}, we know $\bigcup_{i=1} ^5 \mathcal{B}_i$ is a finite rigid set. We can complete $\bigcup_{i=1} ^5 \mathcal{B}_i$ to a rigid set $\mathcal{B}$ such that $\mathcal{B}$ has nontrivial curves of every topological type on $N$ by following the proofs of Lemma 3.9 and Lemma 3.10 given by the author in \cite{Ir10}. The proofs there are given for superinjective simplicial maps but the same arguments work for locally injective simplicial maps. The main idea there is to see that if a curve $y$ is uniquely defined by a set $A$ of curves (meaning that $y$ is the unique curve up to isotopy disjoint from and nonisotopic to each curve in $A$), and we know that $\lambda$ agrees with a homeomorphism $h$ on $[x]$ for each $x \in A$, then $\lambda$ and $h$ agree on $[y]$ as well. This argument works for locally injective simplicial maps too. Since the construction is very similar we will not give it here.\end{proof}\\ 
 
Let $\mathcal{B}$ be as in the Lemma \ref{curves}.
 
\begin{lemma} \label{L_f} There exists a generating set $G$ for $Mod_N$ such that $\forall \ f \in G$, $\exists$ a set $L_f \subset \mathcal{B}$ such that $L_f$ has trivial pointwise stabilizer and $f(L_f) \subset \mathcal{B}$.\end{lemma}
 
\begin{proof} To see the generating set $G$, see Theorem 3.11 and Theorem 3.15 given by the author in \cite{Ir10}  and Theorem 4.14 given by Korkmaz in \cite{K2}. Then the proof for this lemma is similar to the proofs of Lemma 3.13 and Lemma 3.17 in \cite{Ir10}. \end{proof}
 
\begin{theorem} If $g+n \geq 5$, then there exists a sequence $\mathcal{E}_1 \subset \mathcal{E}_2 \subset \dots \subset \mathcal{E}_n \subset \dots$  such that 
		
		(i) $\mathcal{E}_i$ is a finite rigid set in $\mathcal{C}(N)$ for all $i \in \mathbb{N}$,
		
		(ii) $\mathcal{E}_i$ has trivial pointwise stabilizer in $Mod_N$ for each $i \in \mathbb{N}$, and 
		
		(iii) $\bigcup_{i \in \mathbb{N}} \mathcal{E}_i = \mathcal{C}(N)$. \end{theorem}

\begin{proof} By using Lemma \ref{L_f}, we can see that the proof follows as in the proofs of Theorem 3.14 and Theorem 3.18 given by the author in \cite{Ir10}. The construction of the sequence is done by induction and is similar to the one we gave in the proof of Theorem \ref{B-3} in section 2.\end{proof}\\

{\bf Acknowledgements}\\

The author thanks Peter Scott for some discussions.

\vspace{0.2cm}
 
University of Michigan, Department of Mathematics, Ann Arbor, 48109, MI, USA, 

e-mail: eirmak@umich.edu\\

\end{document}